\documentclass[11pt]{article}
\usepackage{one-ended}
\setmidtrue
\newcommand{\per}[1]{
	\operatorname{Per}(#1)
}

\newcommand{\PE}{PolExp\ }

\newcommand{\es}{\emptyset}
\newcommand{\m} {^{-1}} 
\renewcommand{\phi}{\varphi}
\renewcommand{\epsilon}{\varepsilon}

\newcommand {\calf} {{\mathcal {F}}}

\newcommand {\caly} {{\mathcal {Y}}}

\newcommand{\grp}[1]{\langle #1 \rangle}

\newcommand{\Out} {{\mathrm{Out}}}

\newcommand{\Aut} {{\mathrm{Aut}}}

\usepackage{ stmaryrd }

\newcommand{\inc}{\subset}

\newcommand{\Rt}{$\R$-tree}

\newcommand{\ti}{_\sharp }

\newcommand{\rotationless}{pure}

\newcommand{\total}{total PolExp growth}
\newcommand{\Total}{Total PolExp growth}

\begin{document}

\title{PolExp growth for automorphisms of toral relatively hyperbolic groups }
\author{R\'emi Coulon, Arnaud Hilion, Camille Horbez, Gilbert Levitt\thanks{
Rémi Coulon, Arnaud Hilion and Camille Horbez acknowledge support from the Agence Nationale de la Recherche under the Grant \emph{Dagger} ANR-16-CE40-0006-01.
Rémi Coulon and Camille Horbez acknowledge support from the Agence Nationale de la Recherche under the Grant \emph{GoFR} ANR-22-CE40-0004.
Rémi Coulon acknowledges the Région Bourgogne-Franche-Comté under the grant \emph{ANER 2024 GGD}.
His institute receives support from the EIPHI Graduate School (contract ANR-17-EURE-0002). Camille Horbez acknowledges support from the European Research Council under Grant 101040507 Artin-Out-ME-OA (views and opinions expressed are however those of the authors only and do not necessarily reflect those of the European Union or the European Research Council; neither the European Union nor the granting authority can be held responsible for them).
Part of this project was conducted during the \emph{Geometric Group Theory} program at the MSRI in Fall 2016.
}}

\date{\today}

\maketitle

\begin{abstract}
      Let $G$ be a toral relatively hyperbolic group, and let $\varphi\in\Aut(G)$. We prove that, under iteration of $\varphi$, the conjugacy length $\norm{\varphi^n(g)}$ of every element $g\in G$ grows like $n^d\lambda^n$ for some $d\in\N$ and some algebraic integer $\lambda\geq 1$. 
      For a given $\varphi$, only finitely many values of    $d$ and    $\lambda$ occur as $g$ varies in $G$. The same statements hold for the growth of the word length $|\varphi^n(g)|$.
     
      For $G$ hyperbolic, we generalize polynomial subgroups: we show that, for a given growth type  $n^d\lambda^n$ other than $1$, there is a malnormal family of quasiconvex subgroups $K_1,\dots,K_p$ such that a conjugacy class $[g]$ grows at most like $n^d\lambda^n$ if and only if $g$ is conjugate into one of the subgroups $K_i$.
 \end{abstract}

\tableofcontents

\section{Introduction}
Given an automorphism $\varphi\in\Aut(G)$ of a finitely generated group $G$, one may study how the length of an element $g \in G$ or a conjugacy class $[g]$ grows under iteration of $\varphi$. As usual, we denote by $\abs g$ the word length with respect to a finite generating set (whose choice is irrelevant when studying the type of growth), and  we let $\norm g=\min_{h\in G} \abs{hgh\m} $;  
note that the growth of $\norm{\varphi^n(g)}$ only depends on the outer class of $\varphi$, i.e.{} the outer automorphism $\Phi\in\Out(G)$ represented by $\varphi$. This   growth of automorphisms should not be confused with  the   growth of balls  in $G$.

As recalled in \autoref{sommets}, both sequences $\norm{\varphi^n(g)}$ and $|\varphi^n(g)|$ grow like a polynomial times an exponential if $G$ is an abelian group (this follows from linear algebra) a surface group (this follows from the Nielsen-Thurston classification of mapping classes), or a free group (this uses train tracks).
Similar results have been obtained 
by Fioravanti for virtually special groups \cite{Fio}.

In particular, there is no intermediate growth for automorphisms of  these groups: the growth is at most polynomial or at least exponential. On the other hand, the first-named author has proved, using the Rips construction,  that very exotic behaviors are possible for automorphisms of arbitrary groups \cite{Coulon:2022ab}.

In this paper we consider toral relatively hyperbolic groups, i.e.\  
torsion-free groups that are hyperbolic relative to a finite  collection of abelian subgroups of finite rank, see for instance \cite{Bowditch:2012ga,Dahmani:2003hm,Osin:2006cf,Hruska:2010iw} for standard references.
Dahmani-Krishna proved that the growth of $\norm{\varphi^n(g)}$ is at most polynomial or at least exponential in these groups \cite{Dahmani:2023re}. 
Our main result states that the behavior in   toral relatively hyperbolic groups is exactly the same as in abelian or free groups.

\begin{defi}[\PE growth,   spectrum]\label{sga}
    Let $G$ be a finitely generated group, and $\varphi\in\Aut(G)$. We say that $\varphi$, and the outer automorphism $\Phi\in\Out(G)$ represented by $\varphi$, have \emph{\PE growth} if, 
  given $g \in G$, there exist an integer $d \geq 0$, a number $\lambda \geq 1$, and a constant $C > 0$ such that\footnote{The additive constant in (\ref{eqn: main theorem}) will only be needed in degenerate cases, for instance here if $g$ is trivial.}
  \begin{equation}
  \label{eqn: main theorem}
    \frac 1C n^d \lambda^n - C 
    \leq \norm{\varphi^n(g)}
    \leq C n^d \lambda^n + C,
    \quad \forall n \geq 1. 
  \end{equation}
  
    The set of pairs $(d,\lambda)$ which occur as $g$ varies will be called the \emph{spectrum} of $\varphi$ and denoted by $\Lambda_\varphi$, or simply $\Lambda$.
  
  We will write $\norm{\varphi^n(g)} \asymp  n^d\lambda^n$ whenever (\ref{eqn: main theorem}) holds.
  In this case we say that the conjugacy class of $g$ \emph{grows like} $n^d \lambda^n$ under $\varphi$.
\end{defi}

\begin{theo}
\label{main}\label{enplus}
Any automorphism $\varphi$ of a toral relatively hyperbolic group $G$ has \PE growth: for each  conjugacy class $[g]$, the sequence 
$\norm{\varphi^n(g)}$ grows like some $n^d\lambda^n$, with $d \in \N$ and $\lambda \geq 1$.

Moreover: 
\begin{enumerate}
    
    \item $\lambda$ is an algebraic integer, i.e.\ a root of a monic polynomial $P\in\Z[X]$. If $G$ is one-ended, $\lambda$ is an algebraic unit, i.e.\ $P$ further  satisfies $P(0)=\pm1$.
    
   \item \label{enu: main - finiteness} 
    
    Only finitely many    different values of    $(d,\lambda)$   occur in the growth of $\norm{\varphi^n(g)}$ as $g$ varies in $G$   (i.e.{} the spectrum $\Lambda_\varphi$  is finite). %, of cardinality at most $  N$).

     \item \label{enu: enplus - unbout}
     If $G$ is one-ended, there exists $N$ depending only on $G$  such that  the integer $d$, the degree of $\lambda$ as an algebraic number, and the cardinality of  $\Lambda_\varphi$, are bounded by $N$.
    
    \item \label{enu: enplus - hyp}If $G$ is one-ended and hyperbolic,   $\norm{\varphi^n(g)}$ is   bounded, grows linearly,   or 
    $\norm{\varphi^n(g)}\asymp\lambda^n$, with $\lambda$ an $r^{\rm th}$ root of the dilation factor of a pseudo-Anosov homeomorphism on a compact surface $\Sigma$, possibly with boundary (with $r$ and $\abs{\chi(\Sigma)}$   bounded in terms of $G$ only).
     
\end{enumerate}
\end{theo}

We believe that (\ref{enu: enplus - unbout}) is true also when $G$ has infinitely many ends, see \autoref{unifo}.
 \autoref{vtf} explains how to extend the theorem to virtually toral relatively hyperbolic groups.

\begin{rema} \label{elts}
   As in Corollary~6.3 of \cite{Levitt:2009hx}, one can control the growth of $\abs{\varphi^n(g)}$   by noting that the element $g$ has the same growth under $\varphi$ as the conjugacy class of $sg$ under the automorphism of $G*\grp{s}$ equal to $\varphi$ on $G$ and sending $s$ to itself.
   It follows that the theorem is also valid for the growth of the sequences $\abs{\varphi^n(g)}$, except that  one has to allow quadratic growth in (\ref{enu: enplus - hyp}). See \autoref{res: word growth fixed conj class} and \autoref {bgex}
   for details. 
   \end{rema}

When $G$ is hyperbolic, PolExp growth and the finiteness of $\Lambda_\varphi$  allow us to use a construction by Paulin \cite{Paulin:1997ba} to define canonical subgroups related to growth, generalizing the polynomial subgroups introduced in \cite{Levitt:2009hx, Dahmani:2023re}. We show in particular:

\begin{theo}\label{raffi}
    Let $G$ be a torsion-free hyperbolic  group, and $\varphi\in\Aut(G)$.
Given $d\in\N$ and $\lambda\geq1$, there exists a malnormal family of quasiconvex subgroups $K_1,\dots,K_p$ such that a non-periodic conjugacy class $[g]$ grows at most like $n^d \lambda^n$ under $\varphi$ if and only if $g$ has a conjugate belonging to some $K_i$. 
\end{theo}

We also prove a version of this theorem for the growth of elements rather than conjugacy classes, see \autoref{slqc}.

Deducing   \autoref{raffi} from \autoref{main}
and \cite{Paulin:1997ba} is rather straightforward, except for
malnormality (recall that the family is malnormal if $gK_ig\m\cap K_j\ne \{1\}$ implies that  $i=j$ and $g\in K_i$). We do not know whether the theorem holds for toral relatively hyperbolic groups, see \autoref{rk:question-rh}. 

\medskip

The proof of our main theorem (\autoref{main}) has two   steps:  
we first handle one-ended groups, and then infinitely-ended groups.

In the one-ended case we use (a variation of) the JSJ decomposition of $G$ (see the survey \cite{Guirardel:2017te} and references therein). Its vertex groups are rigid, abelian, or surface groups. \autoref{main} holds in these groups (see \autoref{sommets}): this is easy for rigid vertex groups because only finitely many  outer automorphisms extend  to automorphisms of $G$; it follows from linear algebra for abelian groups, and from properties of pseudo-Anosov homeomorphisms for surface groups. 

The main step now is to understand the growth of an automorphism   from that  of its restrictions to the vertex groups of the JSJ decomposition. We illustrate the main difficulty of 
  this local-to-global result on a very simple example.
  
Assume that $G$ is a one-ended, torsion-free, hyperbolic group whose 
JSJ decomposition has the form $G = A \ast_C B$, with 
  $C$ a quasiconvex 
  malnormal infinite cyclic subgroup. Also assume that $\varphi$ leaves $A$ and $B$ invariant.
Any element $g \in G$
has a normal form
\begin{equation*}
    g = a_1b_1a_2b_2 \cdots a_pb_p, 
    \quad \text{where} \quad
    a_i \in A, \ b_i \in B,
\end{equation*}
and
\begin{equation*}
    \varphi^n(g) = \varphi^n(a_1)\varphi^n(b_1)\varphi^n(a_2)\varphi^n(b_2) \cdots \varphi^n(a_p)\varphi^n(b_p)
\end{equation*}
is also a normal form.
But uncontrolled cancellations in the edge group $C$ may occur between $\phi^n(a_i)$ and $\phi^n(b_i)$, for instance, so 
knowing growth in $A$ and $B$ is not sufficient.

 To overcome this difficulty, we adopt a more geometric point of view.
Following Scott-Wall \cite{Scott:1979uk}, we view the JSJ decomposition as a graph of spaces $M$ with fundamental group $G$. 
Vertex spaces are rigid, tori, or compact surfaces. 
One may represent (the outer class of) $\varphi$ by a homeomorphism $f$ of this space $M$, and $g$ by a loop $\gamma$. 

To prove \autoref{main} in this setting  
one must   understand the growth of the length of a closed geodesic representing $f^n(\gamma)$. 
As we apply powers of $f$, it picks up length only when it passes through the non-rigid vertex spaces, and 
 the growth of $f^n(\gamma)$ may be estimated from growth in abelian groups and  the Nielsen-Thurston theory of homeomorphisms of surfaces (to be more precise, there may also be linear growth due to twists). 
 
 The problem now is that shortening may occur in the edge spaces.
 Said differently, some complicated loop created by the homeomorphism in a vertex space could be unwrapped by the homeomorphism in another vertex space.
 
 In order to control this,  we equip $M$ with a suitable metric  that is more convenient to manipulate than the word metric of $G$.
  One important feature, coming from hyperbolicity, is that the orbits for the action of edge groups on the universal cover $X$  of $M$ are contracting (denoting by $Y$ such an orbit, the projection of any ball $B$ disjoint from $Y$ onto $Y$ has uniformly bounded diameter), and separated (the projection of one orbit onto another has uniformly bounded diameter).
 This can be profitably used to estimate precisely the length of $f^n(\gamma)$. 
  More details will be given at the beginning of \autoref{espace}.

\medskip

To deal with groups with infinitely many ends, we consider a Grushko decomposition $G=G_1*\dots*G_q*\F_N$ with each $G_i$   one-ended and $\F_N$ free.
We use the completely split train tracks (CT's) introduced by Feighn-Handel \cite{Feighn:2011tt} for free groups and extended to free products by Lyman \cite{Lyman:2022wz}. The numbers $\lambda$ appearing in \autoref{main} come from growth in the groups $G_i$ or from eigenvalues of the transition matrix of a CT.

It turns out, however, that we need more information on the free factors $G_i$ than just the growth of $\norm{\varphi^n(g)}$ and $\abs{\varphi^n(g)}$ for $g\in G_i$. We illustrate this on a simple example. 

\begin{exam}
\label{exa: need palangre growth}
 Consider the automorphism $\varphi$ of 
\begin{equation*}
    G =H* \F_2=H*\grp{a,b}
\end{equation*}
acting on $H$ as some automorphism $\alpha \in \aut{H}$ and sending $a$ and $b$ to $ax$ and $yb$ respectively, for some $x,y\in H$.
Then $\varphi^n(ab)= aw_nb$, where $w_n \in H$ is given by
\begin{equation*}
   w_n = x\alpha(x)\alpha^2(x)\dots \alpha^{n-1}(x)\alpha^{n-1}(y)\dots \alpha^2(y)\alpha(y)y.
\end{equation*}
Cancellations may occur in $w_n$, therefore knowing the growth of conjugacy classes or elements under iteration of $\alpha$ is not enough to control $\norm{\varphi^n(ab)}$. 
\end{exam}

 This leads us to  define   
 \emph{palangres}\footnote{French for longline, used for fishing, in particular near CIRM} as follows.
 
\begin{defi}[Palangres]
For $\varphi\in\Aut(G)$ and $g\in G$, define the \emph{left} and \emph{right palangres}
\begin{align*}
    L_n(\varphi,g) & = g\varphi(g)\varphi^2(g)\dots \varphi^{n-1}(g), \\
    R_n(\varphi,g) & = \varphi^{n-1}(g)\dots \varphi^2(g)\varphi(g)g.
\end{align*}
A term such as $L_n(\varphi, g)R_n(\varphi,h)$ will be called a  \emph{double palangre}.
\end{defi}

What we really need in the one-ended case is the following result, whose first assertion is just a rewording of \autoref{main}.  

\begin{theo}[see \autoref{res: recap one-ended groups}]\label{palalg}
    Let $G$ be toral relatively hyperbolic and one-ended.   
    Let $\varphi\in\Aut(G)$. 
    \begin{enumerate}
        \item (Classes). For any $g \in G$, there exist $d\in\N$ and $\lambda\geq 1$ such that $\norm{\varphi^n(g)} \asymp n^d\lambda^n$. 
        \item (Palangres). For any $g,h\in G$, there exist $d\in\N$ and $\lambda\geq 1$ such that $\abs{L_n(\varphi,g)R_n(\varphi,h)}  \asymp n^d \lambda^n$.
    \end{enumerate}
\end{theo}

\begin{defi}[\Total]
\label{def: growth dichotomy auto}
We say that $\phi \in \aut G$ has (\emph{algebraic}) \emph{\total} if, for every $k \geq 1$, the automorphism $\phi^k$ satisfies both conclusions (‘Classes' and ‘Palangres') of \autoref{palalg}. 
\end{defi}

\begin{rema}
    In the course of the article it will be convenient to replace $\varphi$ by a power.
    In general, however, the growth of palangres does not behave nicely under this operation.
    This is the reason why the definition of \total\ requires that the conclusions of \autoref{palalg} hold for \emph{every} positive power of $\varphi$.
    As a consequence, for any $k \geq 1$, the automorphism $\varphi$ has \total\ if and only if   $\varphi^k$ does (see \autoref{puiss}).
\end{rema}

The next ingredient is a combination theorem for automorphisms of free products (\autoref{infb0}): essentially, it says that automorphisms of $G$
have PolExp growth whenever 
total PolExp growth 
holds in each free factor $G_i$.
Combined with \autoref{palalg}, this implies our main theorem (\autoref{main}).
 (As we were completing this work, Fioravanti released another    combination theorem for free products  \cite[Proposition~A.11]{Fio},
with different assumptions suitable for virtually special groups.)

\medskip

\autoref{palalg} is also true if $G$ has infinitely many ends (for the second assertion, see the trick mentioned in \autoref{sg}). 
To prove it (for $G$ one-ended), we slightly change our point of view:
instead of considering a single automorphism $\varphi$, we work with the mapping torus
\begin{equation*}
    E=G\rtimes_\varphi \Z= \langle G,t \mid tgt^{-1}=\varphi(g),~\forall g\in G\rangle.
\end{equation*}
Elements of $E$ of the form $gt^k$ with $g\in G$ represent automorphisms of $G$ in the outer class of $\varphi^k$, 
and it turns out that palangres have a natural interpretation in $E$: for 
$\alpha=gt^k$ and $\beta=h^{-1}t^k$ in $E$, one has 
    \[L_n(\varphi^k,g)R_n(\varphi^k,h)=\alpha^n\beta^{-n}.\]

We then consider the universal covering $X$ of a graph of spaces $M$   associated to the  JSJ decomposition, as mentioned above. 
It is a geodesic metric space $(X,\distV)$ on which $G$ acts isometrically and cocompactly.

The geometric version of the `Palangres' assertion of \autoref{palalg} 
is the following result (see \autoref{palgeo2} for a more detailed statement). 
\begin{theo}
\label{palgeo}
Let $G$ be toral relatively hyperbolic and one-ended.
Assume that $G$ acts properly, cocompactly, by isometries on a geodesic metric space $(X,\sf d)$.

Given $\alpha=gt$ and $\beta=h\m t$ in $E=G\rtimes_\varphi\Z$, with $g,h\in G$,  there exist $d\in\N$ and $\lambda\geq 1$ such that
$ \dist x{\alpha^n \beta^{-n} x} \asymp n^d \lambda^n$ for some (hence every) $x\in X$.
\end{theo}

Such estimates hold in vertex spaces, and  we prove a combination theorem that allows to pass from local to global (see \autoref{final}).
This is done by extending the isometric action of $G$ to an action of $E$ by quasi-isometries, using the homeomorphism $f$ of $M$ representing $\varphi$.

\medskip
The paper is organized as follows. 
\autoref{algt} is a preliminary section about growth.
The remainder of the article is divided into three parts.
The first two  cover one-ended groups and infinitely-ended groups respectively.
The last one is devoted to further results, including \autoref{raffi}.

In \autoref{part: one-ended} we first introduce the JSJ decomposition of $G$ that we shall use (\autoref{jsjdec}).
Then we prove in \autoref{sommets} that \total\ holds in vertex groups (rigid, abelian, surface). The real work starts in \autoref{espace}, 
 where we construct the spaces $M$ and $X$ described above; they are used in \autoref{final} to prove a combination theorem leading to \autoref{palalg}.
 
 We begin \autoref{part: infinitely-ended} with \autoref{cts}, where we review completely split train tracks (CT's), as introduced by Feighn-Handel and Lyman \cite{Feighn:2011tt, Lyman:2022wz}. 
 Using the growth of palangres (\autoref{palalg}),
 we   then prove our main theorem (\autoref{main}) also in infinitely-ended groups in \autoref{gfp}. 
 
 Finally, in \autoref{words} we study the growth of sequences $\abs{\varphi^n(g)}$, and in \autoref{hier} we prove \autoref{raffi} and related results.

\paragraph{Acknowledgements.}
We thank Elia Fioravanti, Yassine Guerch, Michael Handel, Robert A.\ Lyman for useful conversations.
We also thank François Laudenbach for mentioning the reference \cite{Fathi:1983dy}.  
The first three authors warmly thank Marie-France and Robert Giraud for welcoming them in their house for invigorating mathematical retreats.

 \section{Algebraic and geometric growth}\label{algt}

In this section we review 
growth under iteration of an (outer) automorphism, as well as its geometric counterpart, and we provide a geometric reformulation of \autoref{palalg} from the introduction (see \autoref{palgeo2} below). 

  \subsection{Algebraic growth and mapping torus}

 Let $G$ be a finitely generated group. We fix a finite generating set and we define $\abs g$ as the word length of $g\in G$. The length of the conjugacy class $[g]$ of $g$, denoted by $\norm g$, is the minimum length of elements conjugate to $g$.
 
 \begin{defi}[Growth type]\label{gtype}
     Let $d\in\N$, and $\lambda\geq 1$. Given $f:\N\to\R$, we write $f(n) \asymp n^d\lambda^n$ if there exists $C>0$ such that
     \begin{equation*}
        \frac 1C n^d\lambda^n-C\leq  f(n) \leq C n^d\lambda^n+C,
        \quad \forall n \geq 1. 
     \end{equation*}
     In this case we say that $f$ has \emph{\PE growth}.
    We write  
    $\preccurlyeq$ and $\succcurlyeq$ instead of  $\asymp$ when only one inequality holds. 
     
     Let $\varphi\in \Aut(G)$. 
     We say that $g$ \emph {grows like} $n^d\lambda^n$, or that $g$ has \emph {growth type} $(d,\lambda)$ under $\varphi$, if $\abs{\varphi^n(g)} \asymp n^d\lambda^n$.
     This does not depend on the choice of a finite generating set for $G$.
     
     Similarly, if $\norm{\varphi^n(g)} \asymp n^d\lambda^n$, this is also true for any $\varphi'\in\Aut(G)$ representing the same outer class $\Phi\in\Out(G)$.
     We then say that $[g]$ \emph {grows like} $n^d\lambda^n$, or has \emph{growth type} $(d,\lambda)$, under $\varphi$ and $\Phi$.
  
     Growth types are ordered in the obvious way, with $(d_1,\lambda_1)\leq (d_2,\lambda_2)$ if $n^{d_1}\lambda_1^n\preccurlyeq  n^{d_2}\lambda_2^n$;
     in other words, $(d_1,\lambda_1) \leq (d_2, \lambda_2)$ if $\lambda_1 < \lambda_2$, or $\lambda_1  = \lambda_2$ and $d_1 \leq d_2$.

 \end{defi}

As explained in the introduction, we   need to study the growth of the double palangres, which have a natural description in terms of the cyclic extension of $G$ associated to the automorphism.

\begin{defi}[Mapping torus]\label{mt}
 
Given $\varphi\in\Aut(G)$ representing $\Phi\in\Out(G)$, we let $E_\varphi$, or simply $E$, be the semi-direct product 
\[E_\varphi = G \rtimes_\varphi \Z = \langle G,t \mid tgt^{-1}=\varphi(g),~\forall g\in G\rangle.\]
We let $\pi:E_\varphi\to \Z$ be the natural homomorphism sending $G$ to 0 and $t$ to 1; in other words,  $\pi(gt^k)=k$ for all $g \in G$ and $k \in \Z$.

\end{defi}

Note that the isomorphism class of $E$ only depends on $\Phi$; in particular, $E\simeq G\times\Z$ if $\varphi$ is inner.

There is a natural homomorphism  from $E$ to $\Aut(G)$, defined by sending the element $\alpha=gt^k$ with $g\in G$ and $k\in\Z$ to the automorphism  $x\mapsto g\varphi^k(x)g\m$. It is injective when $\Phi$ has infinite order and $Z(G)$ is trivial.

This enables us to view elements of $E $ as automorphisms of $G$.
For $k=0$, we get the inner automorphisms $x\mapsto gxg\m$. 
For $k=1$, the elements $gt$ with $g$ varying in $G$  
are precisely the automorphisms representing $\Phi$. More generally,  the elements $gt^k$ are the representatives of $\Phi^k$. 
 
The group law in $E$ is expressed by 
\begin{equation*}
    (gt^n) (ht^m)=g\varphi^n(h)\, t^{n+m}, \quad \forall g,h \in G, \ n,m \in \Z.
\end{equation*}
In particular, palangres appear in 
\begin{equation*}
    (gt)^n
    = g\varphi(g)\dots\varphi^{n-1}(g)t^n
    = L_n(\varphi,g)t^n.
\end{equation*}
Also recall that 
\begin{equation}\label{equation:palangres}
    L_n(\varphi^k,g)R_n(\varphi^k,h)=\alpha^n\beta^{-n}
    \end{equation}
for $\alpha=gt^k$ and $\beta=h^{-1}t^k$ with $g,h \in G $ and $k\in\Z$.

\begin{rema}\label{rema:bete}
A computation shows that, for every $g,h\in G$ and every $n\in\N$, one has $\varphi^n(g)R_n(\varphi,h)=R_n(\varphi,\varphi(g)hg^{-1})g$. It follows 
that studying the growth of double palangres is enough to also study the growth of terms of the form $\varphi^n(g)$ or $L_n(\varphi,u)\varphi^n(g)R_n(\varphi,v)$.
\end{rema}

\subsection{
\Total{} 
 }
 \label{sec: polexp palangre growth}

We now consider growth in a broader geometric context.
Let $G$ be a group acting properly and cocompactly by isometries on a proper geodesic metric space $X$ (e.g.\ a Cayley graph).
The distance between two points $x,x' \in X$ is denoted by $\dist x{x'}$.
If $g \in G$, its \emph{translation length}, denoted by $\norm[X]g$, is
\begin{equation*}
	\norm[X]g = \inf_{x \in X} \dist {gx}x.
\end{equation*}

Fix a base point $x\in X$.
We have $\abs g \asymp \dist x{gx} $ and $\norm g \asymp \norm[X]g$ by the Schwarz-Milnor lemma. This will allow us to go back and forth between algebra and geometry.
 
 The following statement is the geometric version of \autoref{palalg}  (we will show in \autoref{equiv} that the two theorems are equivalent, using Formula~(\ref{equation:palangres})).
  
\begin{theo}[Geometric \total{}]\label{palgeo2} 

 Let $G$ be a one-ended toral relatively hyperbolic group. Let $\varphi\in\Aut(G)$ and $E=G \rtimes_\varphi \Z$. Let $X$ be a proper geodesic metric space on which $G$ acts  isometrically, properly, cocompactly.
 Then:
 
\begin{enumerate}
	\item (Classes). For every  
	$g\in G$, there exist $d\in \N$  and $\lambda\geq 1$ such that
		\begin{equation*}
			 \norm[X]{ \varphi^n(g)}
	\ \asymp n^d \lambda^n.
		\end{equation*}
		
 \item (Palangres).
 For every  $\alpha, \beta \in E$ with $\pi(\alpha)=\pi(\beta)\geq 1$, 
 there exist  $d\in \N$  and $\lambda\geq 1$
 such that 
		\begin{equation*}
			 \dist x{\alpha^n \beta^{-n} x} \asymp n^d \lambda^n
		\end{equation*}
	for some (hence every) $x\in X$.
	
	\end{enumerate}
\end{theo}

\begin{rema}\label{dsG}
 In the second assertion, $\pi(\alpha)=\pi(\beta) $ means that there exist $k\in \Z$ and $g,h\in G$ such that $\alpha=gt^k$ and $\beta=ht^k$.
    This implies that each $\alpha^n \beta^{-n}$ belongs to $G$, hence acts (isometrically) on $X$.
    Therefore, the validity of $ \dist x{\alpha^n \beta^{-n} x} \asymp n^d \lambda^n$ is independent of $x\in X $.
\end{rema}

In the rest of this section, $G$ is any finitely generated group, $\varphi\in\Aut(G)$, and $E=G\rtimes_\varphi \mathbb{Z}$.

\begin{defi}[\Total]
\label{def: growth dichotomy action}

We say that $E$ has
 (\emph{geometric})
 \emph{\total} with respect to $X$ if it satisfies both conclusions of \autoref{palgeo2} (‘Classes' and ‘Palangres').

\end{defi}

Geometric  \total{} does not depend on the choice of $X$:

\begin{lemm}
\label{changing space}\label{puiss0}
    
    Let $X$ and $X'$ be two geodesic metric spaces on which $G$ acts properly, cocompactly, by isometries.
    Then $E$ has \total{} with respect to $X$ if and only if it has \total{} with respect to $X'$.
\end{lemm}

\begin{proof}
The action of $G$ being proper and cocompact, there is a  quasi-isometry $H \colon X \to X'$ which is coarsely $G$-equivariant: there exists $C \geq 0$ such that 
\begin{equation*}
    \dist{gH(x)}{H(gx)} \leq C, \quad \forall g \in G,\ \forall x \in X.
\end{equation*}
As we observed in \autoref{dsG}, the element %$\alpha^ng\alpha^{-n}$ and 
$\alpha^n \beta^{-n}$ belongs to $G$ for every $n \in \N$ and $\alpha, \beta \in E$ with $\pi(\alpha) = \pi(\beta)$. Hence all the asymptotic estimates are the same in $X$ and $X'$ up to a multiplicative/additive error bounded in terms of $C$ and the quasi-isometry parameters of $H$ only; in particular, the growth types may be computed in either space.
\end{proof}

Using a Cayley graph of $G$, we deduce:

 \begin{lemm}[Algebraic and geometric \total{} are equivalent]\label{equiv}
Let $X$ be a proper geodesic metric space on which $G$ acts  properly, cocompactly, by isometries.
Then $\phi$ has \total\ (\autoref{def: growth dichotomy auto}) if and only if $E$ has \total{} with respect to $X$ (\autoref{def: growth dichotomy action}).
\end{lemm}

\begin{proof}
  By \autoref{puiss0}, we may assume that $X$ is the Cayley graph of $G$ with respect to the generating set used to define $  \abs g$, with $x=1_G$,
 so that $\dist x{gx}=\abs g$ and $\norm g_X=\norm g$. 
  The equivalence of the assertions about classes  follows. 
  For the assertions about palangres, we use Formula~(\ref{equation:palangres}).
\end{proof}

\begin{rema}\label{dout}
It follows from the geometric viewpoint that the validity of \autoref{palalg} for $\varphi$ only depends on its outer class (this may also be seen directly by a simple computation).
\end{rema}

We end this section with a key lemma, which will allow us to  replace $\varphi$ by a power (and thus 
assume that it is pure in the sense of \autoref{rtl} below).

 \begin{lemm}\label{puissok}\label{puiss2}\label{puiss}
  Let $k\geq1$.
 \begin{enumerate}
     \item An element or conjugacy class in $G$ has growth type $(d,\lambda)$ under iteration of $\varphi$ if and only if it has growth type $(d,\lambda^k)$ under iteration of $\varphi^k$;
     \item $\varphi^k$ has \total\ if and only if $\varphi$ does.
 \end{enumerate}
\end{lemm}

\begin{proof}
Writing $\phi^n = \phi^i \circ \phi^{km}$ with $0 \leq i < k$ shows the first assertion, which covers in particular the `Classes' part of \total.
 To handle the `Palangres' part, we argue geometrically, using
the action of $G$ on  a Cayley graph $X$. 

We have to estimate $\dist {x}{\alpha^n \beta^{-n} x }$, where we choose $x$ to be a vertex of $X$.
Note that the action of $\aut G$ on $G$ induces an action   of $E_\varphi$ on the vertex set of $X$ by quasi-isometries.

We view  $E_{\varphi^k}=G\rtimes_{\varphi^k}\mathbb{Z}$  as a finite-index subgroup of $E_\varphi$, equal to $\pi^{-1}(k\Z)$. 
 If 
 $E_\varphi$ has \total{} with respect to $X$, so does 
 $E_{\varphi^k}$. 
 This shows the ``if'' direction of the lemma.
  We   now prove the converse. 

Since $\alpha^{n +1} \beta^{-(n+1)}$ acts on the vertex set of $X$ as an isometry and $\alpha$ as a quasi-isometry, we have 
\begin{equation*}
    \begin{split}
        \dist x{\alpha^{n +1} \beta^{-(n+1)} x} 
		\asymp \dist {\alpha x}{\alpha^{n +1}\beta^{-(n+1)} \beta x}
		& =\dist {\alpha  x}{\alpha  \alpha^{n} \beta^{-n}x } \\
        & \asymp \dist {x}{\alpha^n \beta^{-n} x }.
    \end{split}
\end{equation*}
Using Euclidean division, we see that the `Palangres’ part of \total\ is true for    
 $\alpha$ and $ \beta$ if it is true for 
 $\alpha^k, \beta^k$. The ``only if'' direction of the lemma      follows because $\alpha^k,\beta^k\in E_{\varphi^k}$ for any $\alpha,\beta\in E_{\varphi}$.
\end{proof}

\begin{rema}\label{rk: growth-type-power}
For reference in \autoref{ee}, we note that, if $\dist x {\alpha^{kn}\beta^{-kn}x}\asymp n^d\lambda^{n}$, then $\dist x {\alpha^{n}\beta^{-n}x}
 \asymp n^d(\lambda^{1/k})^n$.
\end{rema}
 
 \part{One-ended groups}\label{part: one-ended}

\section{The refined JSJ decomposition}
\label{jsjdec}
 
Let $G$ be toral relatively hyperbolic and one-ended. Recall that $G$  acts on its \emph{canonical JSJ tree} $T_0$, which is the unique (up to equivariant isomorphism) JSJ tree of $G$ over abelian groups, relative to non-cyclic abelian subgroups, equal to its own tree of cylinders, see \cite[Corollary~9.20]{Guirardel:2017te}. 
It enjoys the following properties (we denote by $G_v$ the stabilizer of a vertex $v$, by $G_e$ the stabilizer of an edge $e$).

\begin{prop}[JSJ tree]\label{pjsj2}
Let $G$ be a one-ended toral relatively hyperbolic group, and let $T_0$ be its canonical JSJ tree, with vertex set $V(T_0)$. Then: 
\begin{enumerate}
\item $T_0$ is bipartite:
$V(T_0)$ admits a (unique)  $G$-invariant partition  $V(T_0)=V_0\sqcup V_1$ such that  stabilizers of   vertices in $V_0$ are non-abelian, while   stabilizers of vertices in $V_1$ are abelian (we say that vertices are non-abelian or abelian accordingly); each edge joins a vertex in $V_0$ to a vertex in $V_1$;
\item   if $v\in V_0$,  stabilizers of incident edges are maximal abelian subgroups of $G_v$, and two such stabilizers are conjugate in $G_v$ if and only if  
the edges are in the same $G_v$-orbit;
\item vertex and edge stabilizers   of  
$T_0$ are relatively quasiconvex (see for instance Section 3.3 of \cite{Guirardel:2015fi}),  in particular they are toral relatively hyperbolic;
\item   $T_0$ is 3-acylindrical: segments of length 3 have trivial stabilizer; 
\item   $T_0$ is invariant under automorphisms: the  action of $G$ on $T_0$ extends to an action of $G\rtimes \Aut(G)$ preserving the partition of $V(T_0)$.  \qed
\end{enumerate}
\end{prop}

As   in \cite[Corollary~9.20]{Guirardel:2017te}, every vertex $v\in V_0$ is of one of the following two types: 
\begin{itemize}
    \item \emph{rigid}, i.e.\ $G_v$ does not admit any splitting over abelian subgroups relative to the incident edge stabilizers -- which implies that only finitely many elements of $\Out(G_v)$ extend to automorphisms of $G$ (this relies on the Bestvina--Paulin method and Rips's work on $\R$-trees, see \cite{Guirardel:2015fi});
\item \emph{quadratically hanging} (QH), i.e.\ there exists an isomorphism  $G_v\simeq \pi_1(\Sigma_v)$, where $\Sigma_v$ is a compact (possibly non-orientable) hyperbolic  surface,    and     this isomorphism induces a bijection between stabilizers of incident edges and conjugates of   maximal boundary subgroups of  $\pi_1(\Sigma_v)$.
\end{itemize}

Let now $\varphi\in\Aut(G)$, and let $E=G\rtimes_\varphi\Z$. The tree $T_0$ is $\varphi$-invariant, so the action of $G$ on $T_0$ extends to an action of $E$. 
We refine $T_0$ at QH vertices as in \cite{Bestvina:2023fa, Dahmani:2023re}, so as to get a $\varphi$-invariant $G$-tree $T$ for which the action of $\varphi$ on vertex stabilizers is  either pseudo-Anosov (pA-vertex) or trivial in the following sense (R-vertex).

\begin{defi}[Acting trivially]\label{de:trivial}
    We say that   $\varphi\in\Aut(G)$ \emph{acts trivially} on a subgroup $H$ if there exists an inner automorphism $\iota $ of $G$ such that $\varphi$ agrees with $ \iota $ on $H$.
    In other words, there is an automorphism in the outer class of $\varphi$ which is equal to the identity on $H$.
\end{defi}

We sketch the construction of $T$.  We first
 raise $\varphi$ to a power $\psi=\varphi^N$ (with $N\geq 1$) so that  $\psi$ acts (through $E$) as the identity on the (finite) quotient graph $T_0/G$, and acts trivially on each rigid vertex group $G_v$ (as mentioned  above,  automorphisms of $G_v$ induced by automorphisms of  $G$ have finite order in $\Out(G_v)$). Moreover,  if  $\Sigma_v$ is a surface associated to a QH-vertex and  $f_v$ is a homeomorphism representing the restriction of $\psi$, we want  the complementary subsurfaces to the canonical reduction system for $f_v$ to be invariant under $f_v$, and the induced map to be isotopic to the identity or a pseudo-Anosov homeomorphism (see for instance \cite{Farb:2012ws} for the definitions).  
We may also assume that $\psi$ acts trivially on all edge stabilizers, since edges with non-cyclic stabilizer have a rigid endpoint. Note that $N$ may be bounded in terms of $G$ only.

We define the \emph{refined JSJ tree} $T$ (for $\varphi$) by $G$-equivariantly refining $T_0$ at each QH vertex $v$, using the   cyclic splitting of  $G_v=\pi_1(\Sigma_v)$ dual to the canonical  
reduction system  on $\Sigma_v$.
We then subdivide each of the newly added edges at its midpoint, so that the tree remains bipartite.
The $G$-tree $T$ obtained in this way is not invariant under the whole of $\Aut(G)$, but it is invariant under $\varphi$, so $E$ acts on it. 
We thus get:

\begin{prop}[Refined JSJ tree for $\varphi$]\label {pjsj3}
Let $G$ be a one-ended toral relatively hyperbolic group, let $\varphi\in\Aut(G)$, and let $T$ be the refined JSJ tree for $\varphi$  constructed above. It enjoys the first four properties listed in \autoref{pjsj2}, and the  action of $G$ on $T$ extends to an action of $E=G\rtimes_\varphi\Z$ preserving the partition of $V(T)$. 

Moreover, there is a power $\psi=\varphi^N$ (with $N$  bounded in terms of $G$ only)  such that:  
	\begin{enumerate}
 \item $\psi$ acts as the identity on the (finite) quotient graph $T/G$; 
\item the action of $\psi$ on   every edge stabilizer of $T$ is trivial (see \autoref{de:trivial});	
\item  vertices $v \in V_0$ are either R-vertices or pA-vertices, in the following sense: 
\begin{itemize}

\item If $v$ is an R-vertex, $\psi$ acts trivially on $G_v$.
\item If $v$ is a pA-vertex, it is a QH vertex and $\psi$ acts on $G_v$ as a pseudo-Anosov homeomorphism of   the  compact surface $\Sigma_v$.
 \qed
\end{itemize}
\end{enumerate}
\end{prop}

\begin{rema}\label{access}
    Note that the number  of vertices of $T/G$ and the complexity of the surfaces $\Sigma_v$ appearing in the QH-vertices of $T$ are bounded independently of $\varphi$, in terms of the same quantities for $T_0$.
\end{rema}

\begin{rema}\label{inner} 
If $v$ is an R-vertex, then $G_v$ is non-abelian, so fixes a single point of $T$ and is self-normalizing. It follows that, if an automorphism in the outer class of $\varphi$ leaves $G_v$ invariant, then its restriction to $G_v$ is an inner automorphism of $G_v$. 
\end{rema}

\begin{defi}[Pure]\label{rtl}
  We say that   $\varphi\in\Aut(G)$  is \emph{\rotationless } if  \autoref {pjsj3} applies with $\psi=\varphi$ (i.e.\ $N=1$); in particular, $\varphi$ acts trivially on the quotient graph and on stabilizers of edges and R-vertices.
\end{defi}

The goal of the next three sections is to prove \total. 
By \autoref{puissok}, we may restrict to \rotationless\ automorphisms. 

\section{Palangres in vertex groups} \label{sommets}

We prove \total{} for automorphisms appearing  as restrictions to vertex stabilizers of the refined JSJ tree $T$, using the algebraic (\autoref{palalg}) or geometric (\autoref{palgeo2}) version; recall that these stabilizers are toral relatively hyperbolic. 
For convenience, in this section, $G$ denotes such a group and $\varphi$ an automorphism of $G$.

There are three cases: R-vertices, abelian vertices, pA-vertices.

\subsection{R-vertices}\label{rv}

Using \autoref{inner}, we may assume that
  $\phi$ is inner.
 In this case $E_\varphi\simeq G\times \Z$, and  we may assume that $\varphi$ is the identity. \Total\ reduces to computing the growth of sequences of the form $\abs{g^nh^n}$, with $g,h\in G$.

For $G$ toral relatively hyperbolic, such a sequence  is either bounded or grows linearly (because this is true in the parabolics, which are free abelian).   
As a consequence, 
\total{} holds, with only bounded and linear growth occurring.

\subsection{Abelian groups}\label{abg}

We suppose that $G$ is abelian, so $G\simeq\Z^k$.
We view $\varphi\in\Aut(G)$ as a matrix $A\in GL(k,\Z)$ acting on $\C^k$ and we use additive notation.
We compute the growth of $\norm{A^nv}$ and $\norm{(I+A+\dots+A^{n-1})v}$  for $v\in\C^k$, with $\normV$ a suitable Hermitian norm.
The proof is linear algebra, we give it for completeness.

\begin{prop}
For any $A\in GL(k,\Z)$ and $v\in\C^k$, there exist an eigenvalue $\eta$ of $A$ and an integer $d   \leq k-1$,  such that 
\begin{itemize}
    \item $\norm{A^nv} \asymp n^d |\eta |  ^n$, and
    \item $ \norm{(I+A+\dots+A^{n-1})v} \asymp n^{d'} |\eta |  ^n$, where $d'=d$ if $\eta\neq 1$, and $d'=d+1$ if $\eta=1$.
\end{itemize}
In particular, automorphisms of $\Z^k$ have 
\total{}.
\end{prop}

\begin{rema} \label{unit}
 Note that $\eta$ is an algebraic unit of degree at most $k$, and $|\eta|$ is a unit of degree at most $2k^2$.
\end{rema}

\begin{proof}
By suitably choosing a basis of $\mathbb{C}^k$, we can assume that $A$ is upper block-triangular, with Jordan blocks of the form $\eta_p I_p+N_p$, where $\eta_p$ is an eigenvalue of $A$ and $N_p$ is a nilpotent matrix of size $p$.
We fix a Hermitian norm for which this basis is orthonormal.

It is enough to understand the growth when $A=\eta I+N$. For every $n\geq k$, we have \[ A^n = (\eta I+N)^n = \sum_{\ell=0}^n {n\choose \ell}\eta^{n-\ell}N^\ell = \sum_{\ell=0}^{k-1} {n\choose \ell}\eta^{n-\ell}N^\ell\]
since $N^\ell=0$ for $\ell\geq k$.  
Thus the entries of $A^n$ are of the form $P_{ij}(n)\eta^n$, and those of $A^nv$   of the form $P_{i,v}(n)\eta^n$, 
where each $P_{ij}$, $P_{i,v}$ is a polynomial of degree at most $k-1$. Therefore \[{\norm{A^nv}}^2=\sum_{i=1}^k|P_{i,v}(n)|^2|\eta|^{2n},\] and the estimate for $\norm{A^nv}$ follows.

We now estimate the growth of $\norm{(I+A+\dots+A^{n-1})v}$ as $n$ goes to $+\infty$. Again, after decomposing the space according to the Jordan blocks of $A$, we can assume that $A$ has a single complex eigenvalue $\eta$. If $\eta\neq 1$, then $A-I$ is invertible and $(I+A+\dots+A^{n-1})v=(A-I)^{-1}(A^{n}v-v)$. In this case the result for $\norm{(I+A+\dots+A^{n-1})v}$ follows from the above.

So let us finally assume that $A=I+N$, where $N$ is a nilpotent matrix. As above \[A^s=\sum_{\ell=0}^{k-1} {s\choose \ell}N^\ell,\] 
 with ${s\choose \ell}=0$ if $\ell>s$,
so the entries of $A^s$ are polynomials in $s$ of degree at most $k-1$. Therefore, for every $i\in\{1,\dots,k\}$, the $i^{\rm th}$ entry of $A^nv$ is given by a polynomial $P_{i,v}(n)$ of degree at most $k-1$. For a fixed $v$, if $d_0$ is the maximal degree of the polynomials $P_{i,v}$ as $i$ varies, then the entries of the vector $(I+A+\dots+A^{n-1})v$ are all polynomials in $n$,  the maximal degree being $d_0+1$. The conclusion then follows as in the first part of this proof.
 \end{proof}

\subsection{Surface  groups}
\label{sg}

We now suppose that $\varphi$ is induced by a pseudo-Anosov homeomorphism $f$ of a compact surface $\Sigma$. It is well known that $\norm{\varphi^n(g)}$ is constant (if $g$ is represented by a curve contained in $\partial \Sigma$) or grows like $\lambda^n$, with $\lambda$ the dilation factor of $f$, see e.g.\ \cite[Section~14]{Farb:2012ws}.
We now consider palangres.
\paragraph{Surfaces with boundary.}

If $\Sigma$ has boundary, then $G=\pi_1\Sigma$ is free and we can control $L_n(\varphi, g)R_n(\varphi,h)$ through the following trick. Extend $\varphi$ to $G*\F_2$ by sending the first new generator $t_1$  to $  t_1g$ and the second one $t_2$ to $ht_2$.  

As in   \autoref{exa: need palangre growth}, the growth of $L_n(\varphi, g)R_n(\varphi,h)$ is that of the conjugacy class of $t_1t_2$.
\Total{} thus holds
because \autoref{main} is known in free groups, using train tracks (see \cite{Levitt:2009hx}); in the present case there is a single geometric EG stratum, and palangres are bounded, grow linearly, or like $\lambda^n$.

\paragraph{Closed surfaces.}
This trick cannot be used if $\Sigma$ is closed, because $G*\F_2$ is infinitely-ended and our proof of \autoref{main} in that case  requires palangres. We therefore   give a direct argument.
We prove the geometric version of \total{}  (\autoref{palgeo2}).

\begin{prop}\label{res: growth surface}
Let $G=\pi_1(\Sigma)$, where $\Sigma$ is a closed hyperbolic surface, with universal cover $X$ (a hyperbolic plane).
Let $\varphi \in \Aut(G) $ be induced by a pseudo-Anosov homeomorphism $f$  with dilation factor $\lambda$.
Given $\alpha=gt$ and $ \beta=ht$ in $E =G \rtimes_\varphi \Z $, and $x   \in X$, the sequence $\dist x{\alpha^n \beta^{-n} x}$ is bounded or  grows like $\lambda^n$. 
\end{prop}

\begin{rema}
\label{rem: spectrum pA homeo}
Here $\lambda $ is an algebraic unit whose degree may be bounded in terms of $\abs{\chi(\Sigma)}$.
    If $\pi(\alpha)=\pi(\beta)=k$, then $\alpha,\beta$ represent automorphisms in the same outer class as  $\phi^k$, and $\dist x{\alpha^n \beta^{-n} x}$ is bounded or grows like $\lambda^{ | k | n}$.
\end{rema}

\begin{proof}
 Let $T_-$ and $T_+$ be the $\R$-trees associated to the stable and unstable foliations of $f$ as in \cite{Morgan:1991en}.
 They are projectively $\varphi$-invariant, so the (isometric) action of $G$ on $T_{\pm}$ extends to an affine action of $E$, with $t$ multiplying distances by $\lambda^{\pm1}$. 
 
 Using a lift of $f$, we extend the isometric action of $G$ on $X$ to a quasi-isometric action of $E$. 
 There are natural $E$-equivariant maps from $X$ to $T_\pm$, defined using $f$-invariant foliations or a quadratic differential -- see for instance \cite[Chapter~11]{Kapovich:2001uf} -- and 
 for $a,b\in X$ we denote by  $\dist[_{\pm}]ab$ the distance of the  images of $a$ and $b$  in $T_\pm$.  Then $\distV=\sqrt{\sf \distV[+]^2+\distV[-]^2}$ is a $G$-invariant singular flat metric on $X$ which is quasi-isometric to the hyperbolic metric, and we can use it to estimate the distance between   $x$ and ${\alpha^n \beta^{-n} x}$. We study $\distV[+]$ and $\distV[-]$ separately.
 
	Since  $\alpha$ and $\beta$ act on $T_+$ as homotheties of ratio $ \lambda > 1$, the sequence 
	  $\alpha^{-n}x$ converges, as $n\to+\infty$, to the unique fixed point $y_\alpha$ of $\alpha$ 
(it may be in the metric completion of $T_+$ rather than  in $T_+$ itself). Similarly, 
  $\beta^{-n}x\to y_\beta$.

If  $y_\alpha = y_\beta$, then $\dist[_+]{y_\alpha}{\alpha^n\beta^{-n}{y_\alpha}}=0$, so $\dist[_+]x{\alpha^n\beta^{-n}x}$ is bounded.
	If	$y_\alpha \neq y_\beta$, then $\dist[_{+}]{\alpha^{-n}x}{\beta^{-n}x}$ converges to $\dist[_{+}]{y_\alpha}{y_\beta} > 0$, and $\dist[_+]x{\alpha^n\beta^{-n}x}
			=\lambda^n \dist[_+]{\alpha^{-n}x}{\beta^{-n}x}$ grows like $\lambda^n$.

We now consider $\distV[-]$. It follows from the triangle inequality and the fact that $\alpha^n \beta^{-n}$ belongs to $G$, hence acts by isometries on $X$,  that 
		\begin{align*}
			\dist[_-] x{\alpha^n \beta^{-n}x}
			& \leq  \dist[_-] x{\alpha^n x} +\dist[_-]{\alpha^nx}{\alpha^n \beta^{-n}x}  \\
			& \leq \dist[_-]x{\alpha^nx}+\dist[_-]{\beta^n x}x.
		\end{align*}
	Since $\alpha$ and $\beta$ act on $T_-$ as homotheties of ratio $\lambda\m<1$, both $\dist[_-]x{\alpha^n x}$ and $\dist[_-]{\beta^nx}x$ remain bounded as $n\to+\infty$.
	We conclude that $\dist x{\alpha^n \beta^{-n} x}$ is bounded or	grows like $\lambda^n$.
\end{proof}

\section{A metric Scott-Wall construction}
\label{espace}

The goal of the next two sections is to   carry \total{} from the vertex stabilizers of the refined JSJ tree $T$ to the whole group $G$. To do so, we will let $G$ act as covering transformations on a suitable metric space $(X, \distV)$, coming with a $G$-equivariant projection to $T$. 

This space is given by the Scott-Wall construction \cite{Scott:1979uk}, but we will also need to equip it with an appropriate metric and an action of  $E$ by homeomorphisms. This is the contents of \autoref{MSW} below, which is the goal of this section.
 The definition of \total{}   does not require the group $E$ to act on $X$. 
However, this will be used in a crucial way in \autoref{final}.

For a heuristic argument on how the space $X$ is used in \autoref{final}, assume as in the introduction that $\varphi$ preserves a cyclic amalgam $G=A*_CB$,  with $G$ one-ended  and hyperbolic, so that $M:=X/G$ consists of two vertex spaces $M_A,M_B$ joined by an annulus $U$. Let $X_A, X_B$ be adjacent lifts of  $M_A,M_B$ preserved by $A,B$ respectively, and let $Y$ be the strip  joining them (a lift of $U$).
Suppose that  $\varphi $ may be represented by a homeomorphism $\tilde f$ of $X$  lifting a homeomorphism $f$ of $M$ and preserving    $X_A,X_B$.

Let  $o \in Y$ be a basepoint, and consider  an element $g= ab$ in $G$, with $a\in A$ and $b\in B$. Since $C$ is malnormal and $Y$ is quasiconvex, hyperbolicity implies that the closest point projection of $a^{-1}Y$ (\resp $bY$) on $Y$ is essentially a single point, say $y$ (\resp $z$). 

Now any path from $a^{-1}o$ to $bo$ in $X$ crosses $Y$.
An exercise in hyperbolic geometry using the quasiconvexity of $Y$ then shows that 
\begin{equation*}
    \dist o{go} = \dist{a^{-1}o}{bo} \asymp \dist {a^{-1}o}{y} + \dist yz + \dist z{bo}.
\end{equation*}
Since $\tilde f$ is a quasi-isometry fixing $Y$, the projection of $\tilde f(bY) = \varphi(b)Y$ onto $Y$ is essentially $z$, up to a \emph{bounded error} that does not depend on $a$ and $b$.

Iterating the observation, we get that the projection $\varphi^n(b)Y$ onto $Y$ is essentially $z$, up to a \emph{linear error} in $n$.
The same goes with $\varphi^n(a^{-1})Y$.
Reasoning as before, we get
\begin{align*}
    \dist o{\varphi^n(g)o} 
    & = \dist{\varphi^n(a^{-1})o}{\varphi^n(b)o} \\
    & \asymp \dist {\varphi^n(a^{-1})o}{y} + \dist yz + \dist z{\varphi^n(b)o} + O(n),
\end{align*}
where $O(n)$ grows at most linearly.

It turns out that we understand the restriction of $\varphi$ to each factor $A$ and $B$ well enough to prove that the above linear error is actually bounded (this is a property which we call \emph{quasi-equivariant projections} in \autoref{def: compatible structure}).

Now recall that $a$ and $b$ belong to the factors $A$ and $B$, on which we understand the behavior of $\varphi$.
The previous estimate thus provides a control of the growth of $\abs{\varphi^n(g)}$.
A similar argument works to estimate the growth of $\norm{\Phi^n(g)}$.

 The same strategy applies when $G$ is a toral relatively hyperbolic group. 
The main feature of negative curvature that we use in this context is that the $G$-orbit $\mathcal Y$ of the strip $Y$ is separated and uniformly contracting (see the definitions below).
In order to control the action of $\varphi$ on $\mathcal Y$, we actually need a few more properties.
These are captured by the notion of a \emph {compatible peripheral structure} defined in the next subsection.

\subsection{Peripheral structures}\label{periph}
Let $X$ be a proper geodesic metric space.
By abuse, we will often confuse a geodesic $c \colon \intval ab \to X$ with its image (seen as a subset of $X$).

\subsubsection{Projection, entry/exit point, $D$-neighborhood, contracting}

Let $Y$ be a non-empty closed subset of $X$.
A \emph{projection of $x \in X$ onto $Y$} is a point $p \in Y$ such that $\dist xp = \dist xY$.
Such a point always exists since $X$ is proper and $Y$ is closed.
If $Z$ is another subset of $X$, the \emph{projection of $Z$ onto $Y$}, denoted by $\Pi_Y(Z)$, is the set of all projections  of points of $Z$ onto $Y$.
Formally
\begin{equation*}
	\Pi_Y(Z) = \set{y \in Y}{\exists z \in Z,\ \dist zy = \dist zY}.
\end{equation*}
Let $c \colon \intval ab \to X$ be a path intersecting $Y$.
The \emph{entry} and \emph{exit points} of $c$ in $Y$ are the points $c(t_-)$ and $c(t_+)$, where 
\begin{equation*}
	t_- = \min \set{t \in \intval ab}{c(t) \in Y}
	\quad \text{and} \quad
	t_+ = \max \set{t \in \intval ab}{c(t) \in Y}.
\end{equation*}
The \emph{$D$-neighborhood} of $Y$, denoted by $Y^{+D}$, consists of all points $x \in X$ such that $\dist  x Y  \leq D$.

\begin{defi}[Contracting]
\label{def: contracting set}
	Let $D>0$. % \in \R^*_+$
	A closed subset $Y$ of $X$ is \emph{$D$-contracting} if, for every geodesic $c \colon I \to X$ satisfying $\dist c  Y  \geq D$, the projection $\Pi_Y(c)$ has diameter at most $D$.
	A subset $Y$ is \emph{contracting} if it is $D$-contracting for some  $D \in \R_+$.
\end{defi}

For instance, it is a standard fact that any closed quasiconvex subset of a hyperbolic space is contracting.

The next statements are (direct) consequences of the definition.
Their proofs are left to the reader, see  for instance Yang \cite{Yang:2020ub}.

\begin{lemm}
\label{res: projection contracting set}
	Let $Y \subset X$ be a $D$-contracting subset.
	Let $x,x'\in X$, and let $c \colon \intval ab \to X$ be a geodesic from $x$ to $x'$.
	Let $p$ and $p'$ be respective projections of $x$ and $x'$ onto $Y$.
	If $\dist x Y  < D$ or $\dist p{p'} > D$, then the following hold:
	\begin{enumerate}
		\item $\dist c Y   < D$; in particular, $Y^{+D}\cap c$ is non-empty;
		\item \label{enu: projection contracting set - geo} the entry point (\resp exit point) of $c$ in $Y^{+D}$ is $2D$-close to $p$ (\resp $p'$).
	\end{enumerate}
\end{lemm}

\begin{rema}
\label{rem: Lipschitz proj}
\label{rem: projection contracting set2}
We note the following consequences of Point~\ref{enu: projection contracting set - geo}.
	\begin{enumerate}
		\item 
		The nearest point projection onto $Y$ is large-scale $1$-Lipschitz.
		More precisely, for every subset $Z \subset X$, we have
		\begin{equation*}
			\diam (\Pi_Y(Z)) \leq \diam(Z) + 4D.
		\end{equation*}
		\item 
 		If $p$ is a projection of some point $x\in X$ onto $Y$, then
		\begin{equation*}
			\dist x{y} \geq \dist xp + \dist p{y}   - 4D, \quad \forall y\in Y. 
		\end{equation*}
		\item If $p$ and $p'$ are respective projections of $x$ and $x'$ on $Y$ such that $\dist p{p'} > D$, then
	    \begin{equation*}
	             \dist x{x'} \geq \dist xp + \dist p{p'} +  \dist {p'}{x'} - 8D.\qedhere 
	    \end{equation*}
	\end{enumerate}
\end{rema}

\begin{rema}
\label{rem: growth in qc sbgp}
    Note that, if $g$ in an isometry of $X$ leaving a $D$-contracting subset $Y$ invariant, then the first item of \autoref{rem: Lipschitz proj} implies that 
	\begin{equation*}
	    \norm[X] g \leq \norm[Y]g \leq \norm[X]g + 4D.
	\end{equation*}
	In particular, if $H$ is a quasiconvex subgroup of a hyperbolic group $G$ that is invariant under some automorphism $\varphi \in \aut G$, then, for every $h \in H$, the growth type of $\norm{\phi^n(h)}$ is the same when computed in $H$ or in $G$.
\end{rema}

\begin{lemm}[Quasi-convexity]
\label{res: contraction vs quasi-convex}
	Let $A \geq 0 $.
	Let $Y\inc X$ be a $D$-contracting subset.
Then    any geodesic $c$  joining two points of $Y^{+A}$  lies in the $C$-neighborhood of $Y$, with $C =  \max \left\{ A, D\right\} + 3D/2$.
\end{lemm}

\begin{lemm}
\label{res: diameter intersection vs projection}
	Let $Y$ and $Z$ be two $D$-contracting sets.
	For every $A \in \R_+$, we have
	\begin{equation*}
		\diam\left( Y^{+A} \cap Z^{+A}\right) \leq \diam\left( \Pi_Y(Z)\right) + 2A + 22D. \qedhere
	\end{equation*}
\end{lemm}

\subsubsection{Peripheral structure}

We can now define peripheral structures.

\begin{defi}[Peripheral structure]
\label{def: contracting family}
	A family $\mathcal Y$ of closed subsets of $X$ is a \emph{peripheral structure} if there exists $D \in\R_+$ such that:
	\begin{itemize}
		\item  (\emph{Uniform contraction}). Every element $Y \in \mathcal Y$ is $D$-contracting.
		\item  (\emph{Separation}). For any distinct $Y, Y' \in \mathcal Y$, the projection $\Pi_Y(Y')$ has diameter at most $D$.
	\end{itemize}
\end{defi}

Observe that any subfamily of a peripheral structure is also a peripheral structure.

Relatively hyperbolic groups provide   examples of peripheral structures.
More precisely, we have the following statement.

\begin{prop}
\label{res: example periph structure}
	Let $G$ be a group hyperbolic relative to $\{P_1, \dots, P_n\}$.
	Assume that $G$ acts properly cocompactly on a geodesic metric space $X$.
    Fix $C>0$, and 	for each $i \in \{1, \dots, n\}$  let $Y_i $  be a $P_i$-invariant subspace of $X$ such that 
$\diam(Y_i/P_i) \leq C$. Then the collection 
	\begin{equation*}
		\mathcal Y = \set{ gY_i}{ i \in \intvald 1n, \ g \in G/P_i}
	\end{equation*}
	is a peripheral structure.
\end{prop}

\begin{proof}
    According to Gerasimov and Potyagailo \cite[Proposition 8.5]{Gerasimov:2016uc}, each subset $Y \in \mathcal Y$ is contracting.
    Since $\mathcal Y$ consists of finitely many $G$-orbits, there is $D >0$ such that every element of $\mathcal Y$ is $D$-contracting, which proves uniform contraction.
    
 In order to prove separation, we introduce two useful numbers $M$ and $N$. Since $G$ is hyperbolic relative to $\{P_1, \dots, P_n\}$, we can find $M \in \N$ such that, for every $i,j \in \intvald 1n$ and $g \in G$, the following malnormality holds: if $P_i \cap gP_jg^{-1}$ contains more than $M$ elements, then $i = j$ and $g \in P_i$, and therefore $P_i = gP_jg^{-1}$ (this is clear from the definition of relative hyperbolicity in \cite{Bowditch:2012ga}).

    The action of $G$ on $X$ is proper and cocompact, therefore there is $N \in \N$ such that, for every $x \in X$, the set 
    \begin{equation*}
        \set{g \in G}{\dist x{gx} \leq 17D + 2C}
    \end{equation*}
    contains at most $N$ elements.
    
  Now consider distinct $Y, Y' \in \mathcal Y$, and denote by $P,P'$ the conjugates of some $P_i,P_{i'}$ respectively which act on $Y,Y'$ with   quotient of diameter at most $C$. Since we want to bound the diameter of $\Pi_Y(Y')$, we may assume that it is larger than $D$.
  
      Fix two points $x$ and $y$ in $\Pi_Y(Y')$ with $\dist xy > D$.
      Applying Lemmas~\ref{res: projection contracting set}(\ref{enu: projection contracting set - geo}), and \ref{res: contraction vs quasi-convex} to a geodesic joining points of $Y'$ projecting onto $x$ and $y$ respectively, we see that $x$ is $(9D/2)$-close to a  point $x'$ belonging to $Y'$. 
      Similarly,   $y$ is $(9D/2)$-close to some $y'\in Y'$.
     Denote by $c \colon \intval 0\ell \to X$ a geodesic from $x$ to $y$. By \autoref{res: contraction vs quasi-convex} applied with $A=9D/2$, it lies in the $6D$-neighborhood of both $Y$ and $Y'$.

    Recall that the action of $P$ (\resp $P'$) on $Y$ (\resp $Y'$) is cobounded.
    Thus, for every $t \in \intval 0\ell $, there are $h(t) \in P$ and $h'(t) \in P'$ such that 
    \begin{equation}
    \label{eqn: example periph structure}
        \dist{c(t)}{h(t)x} \leq 6D + C
        \quad \text{and} \quad
        \dist{c(t)}{h'(t)x'}   \le 6D + C.
    \end{equation}
    In particular, since $\dist x{x'}  \leq 9D/2$, we get $\dist{ h(t)^{-1}h'(t)x}x \leq 17D + 2C$.
   
    Fix $a = 12D + 2C + 1$, and 
    suppose   that the distance $\ell$
     between $x$ and $y$ is larger than $L = MNa$.
  It follows from our choice of $N$ that the map $t \mapsto h(t)^{-1}h'(t)$ takes at most $N$ values.  Thus there is a subset $I \subset a\N\cap \intval 0\ell $ with more than $M$ elements such that 
    \begin{equation*}
        h(s)h(t)^{-1} = h'(s)h'(t)^{-1} \quad \forall s,t \in I.
    \end{equation*}
   
  According to (\ref{eqn: example periph structure}), the elements 
  $h(t)$ 
  are pairwise distinct for $t\in I$.
  Fixing $s$ and varying $t$,
 we see that  $P \cap P'$ contains more than $M$ elements, and therefore $P=P'$ and $Y=Y'$, a contradiction. This shows $\diam\left(\Pi_Y(Y')\right) \leq \max\{ D, L\}$ whenever $Y\ne Y'$.
    In other words, the collection $\mathcal Y$ is separated.
\end{proof}

We now suppose that $E=G\rtimes_\varphi \Z$ acts on $X$ by quasi-isometries, with $G$ acting by isometries.

\begin{defi}[Compatible structure]\label{res: ess trivial action on peripheral struct}\label{def: compatible structure}

A peripheral structure $\mathcal Y$ on $X$  is \emph{compatible} (with the action of $E$) if it is $E$-invariant and: 
	\begin{itemize}
		\item (\emph{Projections are quasi-equivariant}). 
	 There is $D \in \R_+$ such that, for every $Y \in \mathcal Y$, $\alpha \in E$, and $x \in X$,
		\begin{equation*}
			\diam \left( \Pi_{\alpha Y}(\alpha x) \cup \alpha \Pi_Y(x) \right) \leq D.
       	\end{equation*}
		\item (\emph{Transversality}). For every $Y \in \mathcal Y$, $\alpha \in E$, and $x \in X$, if no power of $\alpha$ stabilizes $Y$, then the set 
		$ 	\bigcup_{n \in \N} \Pi_Y(\alpha^n x)$
		is bounded.
	\end{itemize}
\end{defi}

\begin{rema} 
\label{eqproj}
Recall that $G$ acts on $X$ by isometries.
	Consequently
	\begin{equation*}
		\Pi_{g\alpha Y}(g\alpha x) \cup g\alpha \Pi_Y(x)=g\left(\Pi_{\alpha Y}(\alpha x) \cup \alpha \Pi_Y(x)\right),
	\end{equation*}
	so one may replace $\alpha$ by any $g\alpha$ if convenient when proving the projections property. In particular, it suffices to check the property for the powers of a single $\alpha$ of the form $gt$.
	
	Similarly,
	\begin{equation*}
		\Pi_{\alpha (gY)}(\alpha (gx)) \cup \alpha \Pi_{gY}(gx)=\Pi_{(\alpha g)Y}\left((\alpha g) x\right) \cup (\alpha g) \Pi_Y\left(x\right),
	\end{equation*}
	hence it suffices to prove the projections property for  one $Y$ per $G$-orbit. If the number of orbits is finite, we may focus on a single $Y$  and define $D$ as the supremum of the bounds associated to each orbit.
\end{rema}

 \subsubsection{The baby example: trivial automorphisms} \label{periphb}

	\begin{lemm}
	\label{res: baby example}
	    Assume that $\varphi$ is an inner automorphism. 
	    Then any $G$-invariant peripheral structure $\mathcal Y$ is compatible with $E$.
	\end{lemm}
	
		\begin{proof}
    Since    $\varphi$ is inner,
    we can identify $E$ with $G\times\Z$ in such a way that the action of $E$ factors through the projection onto $G$. In particular, $E$ acts isometrically and projections are equivariant in the usual sense: $\Pi_{\alpha Y}(\alpha x) = \alpha \Pi_Y(x)$. 
    
In proving transversality, we may restrict to elements of $E$ belonging to $G$.
We therefore  consider $Y \in \mathcal Y$, $g \in G$ and $x\in X$. Let $D$ be associated to $\caly$ as in \autoref{def: contracting family}.

For each $n\geq 0$, we let $p_n$ be a projection of $g^nx$  on $Y$. Note that $\dist {p_n}{p_{n+1}} \leq \dist x{g x} + 4D$ because the projection on $Y$ is large-scale Lipschitz (\autoref{rem: Lipschitz proj}).   	
We assume that  the sequence $(p_n)$ is unbounded, and we claim  that  $Y^{+4D}\cap g Y^{+4D}$ is unbounded. \autoref{res: diameter intersection vs projection} will then imply that $\Pi_Y(g Y)$ also is unbounded, so that $g Y=Y$ by separation of $\caly$. The claim thus implies transversality.

To prove the claim, we consider $n$ such that $\dist{p_0}{p_n}$ is sufficiently large.
For now, we only require that $\dist{p_0}{p_n}$ and $\dist{p_1}{p_{n+1}}$ be larger than $D$.
 Let $\gamma$ be a geodesic from $x$ to $g^nx$. Let $e_0, s_0$ be   the entry and exit point of $\gamma$ in $Y^{+D}$. By \autoref{res: projection contracting set}(\ref{enu: projection contracting set - geo}), they are $2D$-close to $p_0$ and $p_n$ respectively. 
 Let $\gamma_0$ be the subarc of $\gamma$ between $e_0 $ and $s_0$.  By \autoref{res: contraction vs quasi-convex}, applied with $A=2D$, it is contained in   $Y^{+4D}$. It has length at least $\dist{p_0}{p_n}-4D$, and  the initial arc of $\gamma$ between $x$ and $e_0$ has length at most $\dist{x}{p_0}+2D$.

Next we perform the same construction, replacing  $\gamma$ by   the geodesic $g\gamma$, joining $gx$ to $g^{n+1}x$. 
We get a subarc $\gamma_1$ of $g\gamma$ contained in $Y^{+4D}$,  of length at least $\dist{p_1}{p_{n+1}}-4D$, and the arc between $gx$ and the entry point of $g\gamma$ in $Y^{+D}$ has length at most $\dist{gx}{p_1}+2D$.

Now consider $\gamma_2=g\gamma_0\cap \gamma_1$. 
It is a subarc of $g\gamma$ contained in both $ Y^{+4D}$ and $g Y^{+4D}$. Recalling that $\dist{p_n}{p_{n+1}}$ is bounded by $\dist x{gx} + 4D$, we see that the  length of $\gamma_2$, which is a lower bound for  the diameter of $Y^{+4D}\cap gY^{+4D}$, is at least $\dist{p_0}{p_n}-C$ for some number $C$ independent of $n$. The claim follows: 
if $(p_n)$ is unbounded, so is $Y^{+4D}\cap g Y^{+4D}$. 
	\end{proof}

\subsubsection{A second example: surfaces with boundary}\label{be}

	Let $\Sigma$ be a compact hyperbolic surface with geodesic boundary and $X$ its universal cover, seen as a convex subset of $\mathbb H^2$.
	The free group $G=\pi_1(\Sigma)$ is hyperbolic relative to the collection $\{ \group {g_1}, \dots,\group{g_m}\}$, with $g_1,\dots, g_m$ elements of $G$ representing the boundary geodesics $\gamma_1, \dots, \gamma_m$ of $\Sigma$.
	
	Let $\caly_i$ be the full preimage of $\gamma_i$ in $X$, and $\caly=\cup \caly_i$.
\autoref{res: example periph structure}, applied with $P_i=\group {g_i}$, implies that $\mathcal Y$ is a peripheral structure.

	We now let $f$ be a pseudo-Anosov homeomorphism of $\Sigma$ equal to the identity on the boundary. Choosing a basepoint in $\partial \Sigma$, it induces an automorphism $\varphi\in\Aut(G)$, and $E=G\rtimes_\varphi \Z$ acts on $X$:
	the elements of $G$ act by deck transformations, and the generator of $\Z$ acts as a lift of $f$.
	
	\begin{prop}\label{res: compatible structure}
		The peripheral structure  $\mathcal Y$ on $X$ is compatible with the action of $E$.
	\end{prop}
	
	\begin{proof}
		We start with the quasi-equivariance property for projections.
		 Since $f$ equals the identity on $\partial\Sigma$ and $\mathcal Y/G$ is finite ($\partial\Sigma$ has finitely many components), \autoref{eqproj}  allows us to fix $Y$ (a component of $\partial X$) and 
		   assume that $\alpha$ is represented by a homeomorphism $\tilde f$ equal to the identity on $Y$. We then have to bound the diameter of 
		$ \Pi_{  Y}(\tilde f^n (x)) \cup   \Pi_Y(x)$, uniformly  for $x\in X$  and $n\in\Z$.

		Let $\calf\subset X$ be the preimage of one of the  $f$-invariant measured foliations. 
		Let $\ell$ be an infinite  half-leaf (separatrix) originating at a singularity $q$ of $\calf$ contained in $Y$. It is $\tilde f$-invariant, quasi-geodesic (this is the modern way of stating Lemma 1 of \cite{Levitt:1983wg}), and the point at infinity of  $\ell$ is not a point at infinity of $Y$
		(in terms of the geodesic lamination  $\mathcal L$ associated to $\calf$, the point at infinity of $\ell$ is a cusp of the component of $X\setminus \mathcal L$ containing $Y$).

		Let $g$ be a generator of the stabilizer of $Y$ in $\pi_1(\Sigma)$. 
	All half-leaves $g^p {\ell}$ (with $p\in\mathbb{Z}$) are $\tilde f$-invariant. 
	In particular, $\ell, g\ell$ and the arc of $Y$ between $q$ and $g q$ bound an $\tilde f$-invariant fundamental domain $U$ for the action of $\grp g$ on $X$. 

Being quasigeodesics, $\ell $ and $g\ell$ are at a bounded distance from   actual geodesics. These geodesics  have no point at infinity in common with $Y$, so   the projection  of $U$ on $Y$ has finite length. 
	Given $x\in X$, its whole $\tilde f$-orbit  is contained in some $g^p(U)$, and the projections property  follows.

		We now consider transversality.
		Let $Y \in \mathcal Y$, $\alpha \in E$ and $x \in X$. 
		Let $\tilde f_\alpha$ be a 
		homeomorphism of $X$ representing $\alpha$. 
		The result is clear if $x$ is $\tilde f_\alpha$-periodic. Otherwise,  after possibly raising $\tilde f_\alpha$ to some power, the sequence $\tilde f_\alpha^n(x)$   converges as $n\to+\infty$ to a point $\xi \in \partial X$ fixed by (the extension) of $\tilde f_\alpha$  by \cite{Fathi:1983dy}.
		
 Assuming that 
		no power of $\alpha$ stabilizes $Y$, the point $\xi $ cannot be a point at infinity of $Y$:
		  since some $\tilde f_\alpha^k$ fixes $\xi$, it would mean that $Y$ and $\alpha^k Y$ (two boundary components of $X$) share an endpoint at infinity, hence are equal, contradicting our assumption. 
It then follows from hyperbolic geometry that the set
			$ 	\bigcup_{n \in \N} \Pi_Y(\alpha^n x)$
			is bounded.
	\end{proof}

\subsection{Statement of the result} \label{stat}

As explained in the introductory paragraph of this section, 
we will let $E$  act on a suitable metric space $X$. Topologically, $X$ is given by the Scott-Wall construction \cite{Scott:1979uk}. 

Given $\varphi\in\Aut(G)$,   let   $E = G \rtimes_\varphi \Z$ and the  refined JSJ tree $T$ for $\varphi$ be as in  \autoref{pjsj3}. 
Let $\Gamma=T/G$ be the graph of groups associated to $T$, with vertex groups $G_{\tt v}$
and edge groups $G_{\tt e}$. As a general convention throughout this section, we will use typewriter letters $\tt v,\tt e$ for vertices and edges of the quotient graph $\Gamma$, and italic letters $v,e$ for vertices and edges in $T$.

Recall that Scott-Wall define a CW-complex $M$, which is a graph of spaces and has fundamental group $G$. It consists of vertex spaces $M_{\tt v}$ (one per vertex of $\Gamma$, with $\pi_1 M_{\tt v}\simeq G_{\tt v}$) joined by edge  spaces of the form $M_{\tt e}\times [0,1]$ (one for each non-oriented edge $\tt e$ of $\Gamma$, with $\pi_1 M_{\tt e}\simeq G_{\tt e}$), with  $M_{\tt e}\times\{0\}$ and $M_{\tt e}\times\{1\}$ attached to the relevant vertex space.

The universal covering $X$ of $M$ is a tree of spaces, it is equipped with an action of $G$ by deck transformations and an equivariant projection $p:X\to T$.

We will elaborate on the Scott-Wall construction in two ways, carefully choosing the spaces $M_{\tt v}$ and $M_{\tt e}$ to suit our purposes (as a minor technical inconvenience, we will not quite have $\pi_1 M_{\tt v}\simeq G_{\tt v}$). First we extend the action of $G$ on $X$ to an action of $E$  such that the projection $p:X\to T$ is $E$-equivariant (this amounts to representing $\Phi\in \Out(G)$ by a homeomorphism of $M$). 

Second, we define a $G$-invariant metric on $X$.
This metric will greatly simplify our treatment of local-to-global phenomena, hence allowing us to control how the length of curves grows under iteration from the data given by the vertex spaces.

This will be summarized in \autoref{MSW}, 
which is the goal of this section. 
 It follows from \autoref {puiss2} that we are free to replace $\varphi$ by a power when proving \autoref {palgeo2}. This allows us to assume that $\varphi$ is \rotationless{}  (see   \autoref{rtl}); in particular, it acts as the identity on $T/G$.

The following notation will be used throughout, with $p:X\to T$ the projection.

\begin{nota} \label{XY}
	Given an edge $e$ of $T$, we let $Y_e = p^{-1}(m_e)$,  where   $m_e$ is the midpoint of $e$ (all edges have length $1$).
	If $v$ is a vertex of $T$, we let 
	\begin{equation*}
		X_v = p^{-1} \left(\bar B(v,1/2)\right)
		\quad \text{and} \quad
		\mathcal Y_v = \set{Y_e}{e\ \text{edge of $T$ containing}\ v},
	\end{equation*}
    with $\bar B$ denoting the closed ball.
	We view the family $\mathcal Y_v$ as the set of boundary components  of $X_v$. 

If $v$ is a vertex of 
$T$, we denote by $E_v$ the stabilizer of $v$ for the action of $E$ (whereas $G_v$ is the stabilizer for the action of $G$). Note that $E_v$ is a semi-direct product $G_v \rtimes \Z$, with $\Z$ generated by any $gt$ which fixes $v$.	
	\end{nota}

The next definition is a geometric analogue of \autoref{de:trivial}.

\begin{defi}
\label{def: ess trivial action on peripheral struct}
Let $\mathcal Z$ be a family of subsets of $X$. The action of $E$ on $X$ is \emph{essentially trivial} in restriction to  $\mathcal Z$ if, for every $Z \in \mathcal Z$ and  every $\alpha \in E$, there exists $g \in G$ such that $\alpha$ agrees with $g$ when restricted to $Z$.
\end{defi}

\begin{theo}\label{MSW}
Let $G$ be a one-ended toral relatively hyperbolic group, and let $\varphi\in\Aut(G)$ be \rotationless{}.
Let $T$ be the refined JSJ tree for $\varphi$, with vertex set $V = V_0 \sqcup V_1$ (see \autoref{pjsj3}).

 There exist a proper geodesic metric space $X$  with an action of $E$, and a projection $p\colon X\onto T$, with the following properties:
 
\begin{enumerate}
 
 \item The action of $G$ on $X$ is proper, cocompact, by isometries, and the action of $E$ on $X$ is by quasi-isometries. 
 \item The projection $p$  is $E$-equivariant and Lipschitz.
 \item\label{enu: MSW - vertex}
 Point preimages of $p$ are connected.  For every vertex $v$,  the space $X_v$  is convex in $X$.
 \item\label{enu: MSW - edge} The action of $t$ on $X$, hence also that of every $\alpha\in E$, 
  is essentially trivial 	in restriction to each $Y_e$;
  in particular,  the restriction of  $\alpha$ to   $Y_e$ is an isometry.
\item  \label{enu: MSW - compatible} If $v\in V_0$,  the collection $\mathcal Y_v$ is a compatible peripheral structure on $X_v$, equipped with  the action   of $E_v=G_v\rtimes\Z$ (see Definitions \ref {def: contracting family} and \ref{def: compatible structure}).
 \end{enumerate} 
\end{theo}

Recall that $V_0$ is the set of non-abelian vertices of $T$. Convexity ensures that $X_v$, equipped with  the restriction of the  distance function, is geodesic. 

The following definition will be useful  in \autoref{spec} to state a general combination theorem.

\begin{defi}[Metric decomposition]\label{def: metric decomposition}
Let $G$ be a finitely generated  group and $\varphi\in\Aut(G)$.
A \emph{$\varphi$-adapted metric decomposition} of $G$
 is a map $p:X\twoheadrightarrow T$, where $X$ is a geodesic metric space with an action of $E=G\rtimes_\varphi\Z$,
 and $T$ is a bipartite $E$-tree with vertex set $V=V_0\sqcup  V_1$ satisfying all conclusions (1)--(5) from \autoref{MSW}.     
\end{defi}

 Note that the action of $\varphi$ on $T/G$ has to be trivial. 
Moreover the action of $G$ on the tree $T$ has to be acylindrical.

\begin{rema}[Acylindricity of $T$]
\label{res: pt stabilizer ray}
If $e,e'$ are distinct edges of $T$ containing a vertex   $v\in V_0$, then $H=G_e\cap G_{e'}\inc G$ preserves the projection of $Y_e$ to $Y_{e'}$, which has diameter bounded by some $D$ because   the collection $\mathcal Y_v$ is separated, so $H$ is finite of bounded order by properness and cocompactness of the action of $G$ on $X$. Thus there exists $C$ such that, 
for the action of $G$, segments of length at least 3 have stabilizer of order at most $C$ (compare \autoref{pjsj2}(\ref{enu: MSW - edge})).
 \end{rema}

\begin{rema}\label{rem: metric decomposition free splitting}
In \autoref{def: metric decomposition} the tree $T$ is not necessarily a JSJ tree. For instance, if $T$ is the Bass-Serre tree of any $\phi$-invariant free splitting of $G$, one can easily produce a $\phi$-adapted metric decomposition $p:X\twoheadrightarrow T$ of $G$.
\end{rema}

The remainder of this section is devoted to the proof of \autoref{MSW}.

\subsection{Constructing a space}  \label{cas}

To prove \autoref{MSW}, we construct a graph of spaces $M$ using the graph of groups $\Gamma$ associated to the refined JSJ tree $T$. We first define   vertex and edge spaces $M_{\tt v},M_{\tt e}$ 
 (which we call \emph{local spaces}). Every  space $M_{\tt e}$ will be a torus, with   fundamental group $G_{\tt e}$. 
The  fundamental group of $M_{\tt v}$ will be $G_{\tt v}$, except when $\tt v$ is an R-vertex (as a consequence, $\pi_1M$ will not be equal to $G$  but only map onto it). 

We also equip each space $M_{\tt v},M_{\tt e}$ with a length structure (it is a Riemannian metric except at R-vertices), and we define a locally isometric  attaching map from $M_{\tt e}$ to $M_{\tt v}$ whenever $\tt v$ is an endpoint of $\tt e$.

Recall that $\Gamma$ is   bipartite with  
  vertex set  $\tt V = \tt V_0 \sqcup \tt V_1$, where $\tt V_0$ is the set of non-abelian vertices   and $\tt V_1$ is the set of  abelian vertices.
We orient edges   from $\tt V_0$ to $\tt V_1$.
Each edge $\tt e$ 
comes with two attaching maps $\alpha_{\tt e} \colon G_{\tt e} \into G_{o(\tt e)}$ and $\omega_{\tt e} \colon G_{\tt e} \into G_{t(\tt e)}$, with $o(\tt e)\in \tt V_0 $ and $t(\tt e)\in \tt V_1 $.

In order to control the metric, we perform the construction in the following order: abelian vertices, edges, non-abelian vertices. To handle R-vertices, we  need a construction due to Groves.
 
 \begin{nota}
 	If $P$ is a free abelian group of rank $r$, the tensor product $Z = P \otimes \R$ is isomorphic to the vector space $\R^r$.
	Note that, if $Q$ is a subgroup of $P$ (hence a free abelian group of rank $s \leq r$), then the embedding $Q \into P$ induces a canonical embedding $Q \otimes \R \into P \otimes \R$. 
	Hence the advantage of this notation is that it remembers the relation between a group and its subgroups.
 \end{nota}

\begin{prop}[{\cite[Lemmas~4.9 and~4.10]{Groves:2009bj}}]
\label{res: groves toral}
	Let $G$ be a toral relatively hyperbolic group, with  
	free abelian parabolic subgroups $\{P_1, \dots, P_n\}$.
	Let $Z_i = P_i \otimes \R$, endowed with the metric induced by some scalar product, so that the natural action of $P_i$ on $Z_i$ is by isometries.
	
	There exists a geodesic metric space $Z$ with the following properties:
	\begin{enumerate}
		\item The group $G$ acts on $Z$ freely, isometrically, properly, cocompactly. 
		\item For every $i \in \intvald 1n$, there is a $P_i$-equivariant isometric embedding $\kappa_i:Z_i \into Z$ with convex image.
		\item If $Z_i\cap gZ_j\ne \es $ with $g\in G$,  then $i=j$ and $g\in P_i$. \qed 
	\end{enumerate}
\end{prop}

\begin{rema}\
    \begin{itemize}
         \item 
     	Formally, Groves's statement assumes that $P_i \otimes \R$ is endowed with a Euclidean metric for which some basis of $P_i$ is orthonormal, but
	  the construction works verbatim without this assumption.
	\item Since $G$ acts freely and properly, the quotient map $Z\to Z/G$ is a regular covering with group $G$ (unfortunately, $Z$ is not simply connected). By the third property, the maps $\kappa_i$ induce $\pi_1$-injective  embeddings of the tori $T_i=Z_i/P_i$ into $Z/G$ with disjoint images. \qedhere
    \end{itemize}
\end{rema}

\paragraph{Local spaces.}

We can now define the spaces $M_{\tt v},M_{\tt e}$.

\medskip

$\bullet$ Let $\tt v \in \tt V_1$ be an abelian vertex of $\Gamma$.
	By assumption, $G_{\tt v}$ is a free abelian group of the form $G_{\tt v} = \Z^k$ for some $k \in \N$.
	We endow the space $G_{\tt v} \otimes \R$ with the canonical Euclidean metric.
	The quotient $\mathbb T_{\tt v} = (G_{\tt v} \otimes \R) / G_{\tt v}$ is a $k$-dimensional torus.
	
  We define the star   $\str {\tt v}$ as the closed ball of radius $\frac{1}{2}$ centered at $\tt v$ in $\Gamma$; we view it as a union of arcs $[{\tt v},m_{\tt e}]$ of length $\frac12$,
 with $\tt e$ any edge  starting at $\tt v$ and
  $m_{\tt e}$ its midpoint.
	We let $M_{\tt v} = \mathbb T_{\tt v} \times \str {\tt v}$, 
	  endowed   with the product metric.
	Introducing the star of $\tt v$ will be needed only to construct the action of $E$ on $X$.
	
	\medskip 

$\bullet$	Let $\tt e$ be an edge of $\Gamma$ and ${\tt v} = t (\tt e)\in \tt V_1$ its  abelian endpoint. 
	The homomorphism $\omega_{\tt e} \colon G_{\tt e} \into G_{\tt v}$ induces an $\omega_{\tt e}$-equivariant embedding $G_{\tt e} \otimes \R \into G_{\tt v} \otimes \R$.
	Identifying $G_{\tt e} \otimes \R$ with its image in $G_{\tt v} \otimes \R$ provides a metric structure on $G_{\tt e} \otimes \R$; it is induced by a scalar product.
We define $M_{\tt e} = (G_{\tt e} \otimes \R) / G_{\tt e}$, a flat   torus. 
	
	We now define an attaching map 	$\iota^\omega_{\tt e}$ from $M_{\tt e}$ to $M_{\tt v} = \mathbb T_{\tt v} \times \str {\tt v}$ by sending  $x \in M_{\tt e}$ to $(\iota(x),m_{\tt e})$, with $\iota \colon M_{\tt e} \to \mathbb T_{\tt v}$ induced 
by the embedding $G_{\tt e} \otimes \R \into G_{\tt v} \otimes \R$ and $m_{\tt e}\in\str {\tt v}$ as defined above.  Note that the map $\iota^\omega_{\tt e}$ may fail to  be injective.
	It is locally an isometry, though.
	
	\medskip
	$\bullet$ Let $\tt v$ be a pA-vertex.
	We let $M_{\tt v}$ be the underlying (topological) surface $\Sigma_{\tt v}$, with   fundamental group   $G_{\tt v}$. We now define a metric on $\Sigma_{\tt v}$.

	If $\tt e$ is an edge of $\Gamma$ starting at $\tt v$, then $\alpha_{\tt e}(G_{\tt e})$ is a maximal cyclic subgroup of $G_{\tt v}$ corresponding to a boundary component.
	In the previous step we assigned a metric to $M_{\tt e} \simeq S^1$.
	Let $\ell_{\tt e}$ be the length of this circle. 
	We fix a hyperbolic metric on $M_{\tt v}$ such that the boundary curve associated to each edge $\tt e$ starting at $\tt v$ is totally geodesic and has length $\ell_{\tt e}$  \cite[Section~10.6.3]{Farb:2012ws}.
	We define an isometric attaching map $\iota^\alpha_{\tt e} \colon M_{\tt e} \to M_{\tt v}$ by identifying $M_{\tt e}$ with the corresponding boundary curve (with the orientation prescribed by the embedding $\alpha_{\tt e}:G_{\tt e}\to G_{\tt v}$).

    \medskip
	$\bullet$ Let $\tt v$ be an R-vertex. As pointed out in Lemma 3.8 of \cite{Guirardel:2015fi}, the group $G_{\tt v}$ is hyperbolic relative to a family of free abelian parabolic subgroups $P_1,\dots, P_n$;  this family contains
  the incident edge groups $P_1,\dots, P_k$ (recall that they are pairwise not conjugate in $G_{\tt v}$). When $\tt v$ is QH, the $P_i$'s are the incident edge groups and we can take   $M_{\tt v}=\Sigma_{\tt v}$   as in the pA case; we are now concerned  with the rigid vertices.

For $i\leq k$ the group $P_i$ is an edge group and  the space  $P_i\otimes \R$ has been assigned a metric (coming from the  inclusion of $P_i$ into an abelian vertex group). We use these metrics (and arbitrary metrics for $i>k$) to apply \autoref{res: groves toral}. 
	   We get   a geodesic metric space $Z_{\tt v}$ endowed with a proper  cocompact action of $G_{\tt v}$ such that, for every edge $\tt e$ starting at $\tt v$, we have an $\alpha_{\tt e}$-equivariant isometric embedding $G_{\tt e} \otimes \R \into Z_{\tt v}$ with convex image.
		Let $M_{\tt v} = Z_{\tt v} / G_{\tt v}$, endowed with the quotient metric. For each edge $\tt e $ starting at $\tt v$, 
	we get an    isometric embedding $\iota_{\tt e}^\alpha \colon M_{\tt e} \into M_{\tt v}$.

	The fundamental group of $M_{\tt v}$ is not $G_{\tt v}$, but as mentioned above the projection $Z_{\tt v} \onto M_{\tt v}$ is a regular covering map whose deck transformation group is $G_{\tt v}$, so there is an epimorphism $\pi_{\tt v}:\pi_1(M_{\tt v})\onto G_{\tt v}$. Note that it is injective on the fundamental groups of the tori $\iota_{\tt e}^\alpha (M_{\tt e})$.

\paragraph{$M$ and $X$ as topological spaces.}
We now define a space $M$ by combining the spaces $M_{\tt v},M_{\tt e}$ into a graph of spaces based on $\Gamma$, as in \cite{Scott:1979uk}. 
Denoting by $\tt V=\tt V_0\sqcup\tt V_1$ the set of vertices of $\Gamma$, and by $\tt E$ the set of edges of  $\Gamma$, oriented from $\tt V_0$ to $\tt V_1$ as above, $M$ is the quotient of 
\begin{equation*}
	\left( \bigsqcup_{\tt v \in \tt V} M_{\tt v} \right) \sqcup \left( \bigsqcup_{\tt e \in\tt E} M_{\tt e}\times \intval 01 \right)
\end{equation*}
by the identifications prescribed by the attaching maps 
$\iota^\alpha_{\tt e}$ and $ \iota_{\tt e}^\omega$: for each $\tt e\in\tt E$ and $x \in M_{\tt e}$, we identify $(x,0)$ with $\iota^\alpha_{\tt e}(x) \in M_{o(\tt e)}$ and $(x,1)$ with $\iota_{\tt e}^\omega(x) \in M_{t(\tt e)}$.

Note that the space $M$ comes with a natural projection $M \onto \Gamma$:  for every vertex $\tt v \in \tt V$ it maps the subspace $M_{\tt v}$ to $\tt v$, and for every edge $\tt e \in\tt  E$ it %sends the point $(x,s)$ in 
projects $M_{\tt e}  \times \intval 01$ to $\intval 01$, which we view as a parametrization of the oriented edge $\tt e$ of $\Gamma$.

The spaces $Z_{\tt v}$  provided by  \autoref{res: groves toral}
are not simply connected, so the fundamental group of $M$ is not isomorphic to $G$. 
It is the fundamental group of a graph of groups based on $\Gamma$, with the same vertex and edge groups as $\Gamma$ except that, for $\tt v$ an R-vertex, the vertex group is $\pi_1(M_{\tt v})$ rather than $G_{\tt v}$. 

Now consider the normal subgroup $N$ of $\pi_1(M)$ generated by the kernels of the maps $\pi_{\tt v}:\pi_1(M_{\tt v})\onto G_{\tt v}$, for $\tt v$ an R-vertex. Since $\pi_{\tt v}$ is injective on incident edge groups, the quotient $\pi_1(M)/N$ is the fundamental group of the graph of groups where $\pi_1(M_{\tt v})$ has been replaced by $G_{\tt v}$ (the quotient map is a composition of  vertex morphisms in the sense of   \cite{Dunwoody}). The embeddings of the incident edge groups are the same as in $\Gamma$, so $\pi_1(M)/N$ is  isomorphic to $G$. 

We define $X$ as the  covering space of $M$ associated to $N$, with the action of $G$ by deck transformations. Passing to the cover, the projection $M \onto \Gamma$ lifts to  a $G$-equivariant projection $p \colon X \to T$, where $T$ is the refined JSJ tree.

\paragraph{Extending the action to $E$.}

 Our next goal is to extend the action of $G$ on $X$ to an action of   $E$. To do that, we  represent  $\varphi$ by a homeomorphism $f$ of $M$, in the following sense: the induced automorphism of $\pi_1(M)$ descends to an automorphism of $G$ belonging to  the same outer class as $\varphi$. Once this is done, we choose a lift $\tilde f$ of $f$ to $X$ and we let the element $t$ of $E$ act as $\tilde f$. Since the action of $G$ on $X$ is proper and cocompact, $E$ acts by quasi-isometries.
Moreover, the projection $p \colon X \to T$ is $E$-equivariant.

 The homeomorphism $f$ will be the identity on edge spaces, ensuring that the action of $E$ on $X$ is essentially trivial in restriction to the spaces $Y_e$, as required in (\ref{enu: MSW - edge}) of \autoref{MSW}.

  We assume that $\varphi$ is \rotationless, so as in 
   \autoref{pjsj3} it acts trivially on $T/G$, on edge stabilizers,   on stabilizers of R-vertices, and as a pseudo-Anosov homeomorphism on pA-vertices. The homeomorphisms of $M$ that we shall construct will all be the identity on edge spaces and vertex spaces associated to R-vertices (in particular, they will induce automorphisms of $G$).  We now define $f$ on  $M_{\tt v}$, for $\tt v$ a pA or abelian  vertex.

If $\tt v$ is a pA vertex, then $M_{\tt v}$ is a surface $\Sigma_{\tt v}$. 
The  automorphism $\varphi$ acts on $G_{\tt v}=\pi_1(\Sigma_{\tt v})$ as a pseudo-Anosov homeomorphism sending each boundary component to itself in an orientation-preserving way, and we define $f$ on $M_{\tt v}$ as such a homeomorphism, making sure that it is the identity on   $\partial \Sigma_{\tt v}$ so that it may be extended as the identity to the annuli attached to $\Sigma_{\tt v}$.

If $\tt v$ is an abelian  vertex, then $G_{\tt v} \simeq \Z^k$ is free abelian and $ \varphi_{|G_{\tt v}}$ can be represented by a matrix $A \in {\rm GL}(k,\Z)$.
	In particular it induces an affine homeomorphism $h$ of the torus $\mathbb T_{\tt v} = (G_{\tt v} \otimes \R)/ G_{\tt v}$, 
which 	we extend as $h\times{\rm id}$  to  $M_{\tt v}=\mathbb T_{\tt v} \times \str {\tt v}$.
	The assumptions on $\varphi$ imply that $h$ is the identity on the subtori where   edge spaces are attached.

We thus obtain a homeomorphism $f$ of $M$ which induces an automorphism   of $\pi_1(M)$, and also an automorphism $\psi$ of its quotient $G$.  Unfortunately, it does not have to be in the same outer class as $\varphi$: we only ensured that $\psi$ has the same action as $\varphi$ on vertex groups.

As in Section 4 of \cite{Guirardel:2015fi}, $\psi\varphi\m$ belongs to the group  of twists, and more precisely is a product of twists near   vertices in $\tt V_1$ because edge groups of $\Gamma$ incident to a vertex   $\tt v\in \tt V_0$ are maximal abelian subgroups  of $G_{\tt v}$ (see Corollary 4.4 of \cite{Guirardel:2015fi} and its proof).

It therefore suffices to represent any twist near a vertex   $\tt v\in \tt V_1$ by a homeomorphism $h$ of $M$. Such a twist is determined by an incident edge $\tt e$ and an element   $u\in G_{\tt v}$. The homeomorphism $h$ will be supported in the subspace 
$\mathbb T_{\tt v} \times[{\tt v},m_{\tt e}]$ of $\mathbb T_{\tt v} \times \str {\tt v} $ associated to $\tt e$. 

Identifying the arc
$ [{\tt v},m_{\tt e}]$
  with the interval $\intval 0{1/2}$, we 
define a bijection $H_{\tt e}$ of  $(G_{\tt v} \otimes \R) \times \  [{\tt v},m_{\tt e}]$ by sending $(x,s)$ to $(x + 2su, s)$. 
Passing to the quotient, $H_{\tt e}$ induces a homeomorphism $h_{\tt e}$ of $\mathbb T_{\tt v} \times [{\tt v},m_{\tt e}]$   
which pointwise fixes $\mathbb T_{\tt v} \times \{{\tt v},m_{\tt e}\}$  (when $\mathbb T_{\tt v}$ is a circle, $h_{\tt e}$ is a usual Dehn twist on an annulus). We extend it by the identity to the complement of $\mathbb T_{\tt v} \times [{\tt v},m_{\tt e}]$ in $M$.

\paragraph{The metric structure on $X$.}

We now define a $G$-invariant metric on $X$. The space $X$ is a tree of spaces, whose vertex spaces $Z_v$ are covering spaces of the vertex spaces of $M$. We lift the length structure defined above  to them. 

 We obtain the product  of a Euclidean space $ G_v \otimes \R$ by a graph $\str v$ if $v$ is abelian,  the universal covering of  a hyperbolic surface  $\Sigma_{\tt v}$ (a convex subspace of the hyperbolic plane bounded by disjoint geodesics) if $v$ is pA,  a  space $Z_v$ provided by \autoref{res: groves toral} if $v$ is an R-vertex. 

The edge spaces are of the form $ (G_e \otimes \R)\times\intval 01$. We have defined a Euclidean metric on $ G_e \otimes \R$, and we equip $ (G_e \otimes \R)\times\intval 01$ with the product metric.
We defined the attaching maps $\iota^\alpha_{\tt e},   \iota^\omega_{\tt e}$ in such a way that they lift to isometric embeddings of $ (G_e \otimes \R)\times \{0\}$ and $ (G_e \otimes \R)\times\{1\}$ into the vertex spaces, with   convex images. 

This allows us to patch the length structures defined on each vertex and edge space together, so as to obtain a global length structure on $X$. 
 Gluing the vertex and edge spaces  successively, and applying inductively \cite[Chapter~I, Lemma~5.24]{Bridson:1999ky} to the tree of spaces,
we see that the length structure defines a genuine distance function on $X$ (not a pseudo-distance) making $X$ a geodesic metric space. The edge and vertex spaces are convex, hence so are the spaces $X_v=p^{-1}(\bar B(v,1/2))$ appearing in \autoref{MSW}.

\paragraph{Local peripheral structure.}

The metric space $X$ just constructed clearly satisfies the first three conditions of \autoref{MSW}. We have established essential triviality, there remains to show that $\mathcal Y_v
= \set{Y_e}{e\ \text{edge of $T$ containing}\ v}
$ is a compatible peripheral structure on $X_v$ when $v\in V_0$ is a non-abelian  vertex of $T$.

This follows from the   example given in \autoref{be} if $v$ is a pA-vertex (even though one has attached Euclidean annuli to $\Sigma_{\tt v}$). 
If $v$ is an R-vertex, the space $M_v$ was constructed using \autoref{res: groves toral}, and 
$\mathcal Y_v$ is a   peripheral structure on $X_v$ by \autoref{res: example periph structure}. It is compatible with the action of $E_v$ because $\varphi$ acts trivially on $G_v$ 
  (see the baby example in \autoref{periphb}).

\section{A combination theorem for \total{}}
\label{final}

The goal of this section is to prove the following combination theorem for \total{}. We refer to \autoref{stat} for the notion of a $\varphi$-adapted metric decomposition of $G$ (and the  associated notations); it collects the relevant properties of the refined JSJ tree $T$  of \autoref{pjsj3}  and  the Scott-Wall space  $X$ constructed in \autoref{MSW}.

\begin{theo}\label{thm: combination without spectra}
Let $G$ be a group. Let $\varphi\in\Aut(G)$. Let $p:X\to T$ be a $\varphi$-adapted metric decomposition of $G$. Assume that  $E_v$ has \total{} with respect to $X_v$ for every vertex $v$ of $T$ (see \autoref{def: growth dichotomy action}).

Then  $E$ has \total{} with respect to $X$.
\end{theo}

In order to prove the ``moreover'' in \autoref {main}, we will   state and prove \autoref{res: combination growth}, a   more precise version of \autoref{thm: combination without spectra} that compares all possible growth types in $G$ to those in $G_v$. With \autoref{palgeo2} in mind, we shall define two notions of spectra recording these growth types. A reader only interested in the main assertion of \autoref{main} may safely ignore all spectra.

The proof of \autoref{res: combination growth} will be completed in \autoref{sec: proof combination growth}. Before that, we introduce some extra tools. In \autoref{secgps}, we endow the space $X$
%metric decomposition 
with a (global) peripheral structure. This provides a metrically useful way to decompose a path in $X$ into a concatenation of local contributions; precise metric estimates will be given in \autoref{sec: metric estimates}. We then take advantage of this property to estimate the growth of $\phi$.

Finally, armed with \autoref{res: combination growth}, we will complete the proof of \total{} for automorphisms of one-ended toral relatively hyperbolic groups in \autoref{sec: endgame-one-ended}.

\subsection{Spectra}\label{spec}

 Recall that $\pi:G\rtimes_{\varphi}\Z \to\Z$ is the canonical projection.

\begin{defi}[Spectrum, palangre spectrum]\label{de: spectrum}
    Let $G$ be a group, let $\varphi\in\Aut(G)$, let $E=G\rtimes_{\varphi}\Z$, and let $X$ be a proper geodesic metric space on which  $G$ acts properly, cocompactly by isometries. 
    \begin{itemize}
    \item The \emph{spectrum} of  $E$ (with respect to $X$) is the set $\Lambda$ of all pairs $(d,\lambda)\in\N\times [1,+\infty)$ for which there exists $g\in G$ such that 
			 $\norm[X]{ \varphi^n(g)}
	\ \asymp n^d \lambda^n.$
        \item The \emph{palangre spectrum} of   $E$ (with respect to $X$) is the set $\Lambda_{\mathrm{pal}}$ of all pairs $(d,\lambda)\in\N\times [1,+\infty)$ for which there exist $\alpha,\beta\in E$ with $\pi(\alpha)=\pi(\beta)$ positive
          such that $\dist{x}{\alpha^n\beta^{-n}x}\asymp n^d\lambda^{n\pi(\alpha)}$ for some (equivalently, any) $x\in X$.
    
   \end{itemize} 
   
    Both spectra contain $(0,1)$ (bounded growth).
    By \autoref{equiv}, the spectra only depend on $\varphi$. The spectrum $\Lambda$ is the spectrum of $\varphi$, as defined in \autoref{sga}. The set $\Lambda_{\mathrm{pal}}$ will be called the 
\emph{palangre spectrum} of $\varphi$; it may be defined as the set of $(d,\lambda)$ such that there exist $g,h\in G$ and $k\geq1$ such that  $\abs{L_n(\varphi^k,g)R_n(\varphi^k,h)}\asymp n^d \lambda^{kn}$.

\end{defi}

The following remark will enable us to replace $\varphi$ by a power when computing spectra.

\begin{rema}\label{rem: spectrum-power}
Let $k\in \N \setminus\{0\}$.
Denote by $\Lambda$, $\Lambda_{\rm pal}$ and $\Lambda^k$, $\Lambda_{\rm pal}^k$ the respective spectra of $\varphi$ and $\varphi^k$. 
One has $\Lambda^k=\{(d,\lambda^k) \mid (d,\lambda)\in \Lambda\}$ by the first assertion of \autoref{puiss},  and also $\Lambda_{\mathrm{pal}}^k=\{(d,\lambda^k)\mid (d,\lambda)\in\Lambda_{\mathrm{pal}}\}$.

In the definition of $\Lambda_{\mathrm{pal}}$, one might be tempted to only consider elements $\alpha,\beta$ such that $\pi(\alpha)=\pi(\beta)=1$. This might however not lead to the same definition in general, 
because not all elements $\alpha\in E$ with $\pi(\alpha)=k$   arise as $k^{\rm th}$ powers of elements projecting to $1$ under $\pi$. The definition we gave is the correct one to ensure that $\Lambda^k_{\rm pal}$ is as stated above.
It turns out that the two possible definitions of $\Lambda_{\rm pal}$ coincide
when $G$ is a toral relatively hyperbolic group, but this is a consequence of our proof, and is not \emph{a priori} obvious.
\end{rema}

\begin{nota}\label{de: spectrum2}\  
    \begin{itemize}
\item   Recall that growth types are ordered in the obvious way, with $(d_1,\lambda_1) \leq (d_2, \lambda_2)$ if and only if $n^{d_1}\lambda_1^n \preccurlyeq n^{d_2} \lambda_2^n$ 
(see \autoref{gtype}).

    \item Given a set $\Delta$ of growth types $(d,\lambda)$, we define
    \begin{equation*}
        \Delta^+ = \Delta \cup \set{ (d+1, 1)}{(d,1) \in \Delta}.
    \end{equation*}
    Note that $(1,1)$ (linear growth) belongs to $\Delta^+$ if $ \Delta$ is a set $\Lambda $ or $\Lambda_{\mathrm{pal}}$ as in \autoref{de: spectrum}.
\end{itemize}  
\end{nota}

The following theorem is a refined version of \autoref{thm: combination without spectra}, that keeps track of the spectra of the actions.

\begin{theo}
\label{res: combination growth}
	Let $G$ be a group.
	Let $\phi \in \aut G$.
	Let $p \colon X \to T$ be a $\varphi$-adapted metric decomposition of $G$, and denote by $V$ the vertex set of $T$.
	Assume that $E_v$ has \total{} with respect to $X_v$, for every vertex $v\in V$. Denote its spectrum by $\Lambda_v$, and  its palangre spectrum by $\Lambda_{\mathrm{pal},v}$.

	Then $E$ has \total{} with respect to $X$; its spectra $\Lambda$ and $\Lambda_{\rm pal}$ satisfy
	\begin{equation*}
	  \bigcup_{v \in V} \Lambda_v \subset \Lambda \subset \bigcup_{v \in V} (\Lambda_v\cup\Lambda_{\mathrm{pal},v})
	\end{equation*}
	and
	\begin{equation*}
	  \bigcup_{v \in V} \Lambda_{{\rm pal}, v} \subset \Lambda_{\rm pal} \subset {\bigcup_{v \in V} \Lambda_{\mathrm{pal},v}^+}.
	\end{equation*}
\end{theo}

\begin{rema}
The leftmost inclusion for the palangre spectrum $\Lambda_{\rm pal} $ is clear, using the convexity of $X_v$; the inclusion  for the spectrum  $\Lambda$ will follow from \autoref{clas elliptic}. If  $\Lambda_{\mathrm{pal},v}\inc \Lambda_v^+$ for all $v$, as is the case when $G$ is toral relatively hyperbolic   and $p$ is as in \autoref{MSW}, then $\bigcup_{v \in V} \Lambda_v \subset \Lambda \subset \bigcup_{v \in V} \Lambda_v^+$.
\end{rema}

\subsection{Global peripheral structure}
\label{secgps}

The tree $T$ is bipartite, with 
vertex set $V=V_0 \sqcup V_1$.
For concreteness, we shall refer to the  vertices of $T$ in $V_0$ and $V_1$ as \emph{non-parabolic} and \emph{parabolic} respectively (in \autoref{pjsj3}, the vertices in $V_1$ are the abelian ones).

Recall from (\ref{enu: MSW - compatible}) of \autoref{MSW}  that, if $w\in V_0$ is a non-parabolic vertex of $T$,  the space $X_w$  is  equipped with a ``local'' compatible peripheral structure  
\begin{equation*}
    \caly_w  = \set{Y_e}{e\ \text{edge of $T$ containing}\ w}.
\end{equation*}
 We now use the parabolic vertices to get  a ``global'' peripheral structure $\caly $ on $X$.

\begin{prop}
\label{res: global peripheral structure}
    Let $p \colon X \to T$ be a $\varphi$-adapted metric decomposition of $G$.
    The   collection  
    \begin{equation*}
        \mathcal Y = \set{X_v}{v \in V_1}
    \end{equation*}
    is a compatible peripheral structure on  $X$ (see Definitions \ref {def: contracting family} and \ref{def: compatible structure}).
\end{prop}

\begin{proof}
    By construction, the set $\mathcal Y$ is $E$-invariant and consists of closed subsets of $X$.	Since $V_0/G$ is finite, there exists  $D>0$ such that  the following hold for every $w \in V_0$:
	\begin{itemize}
		\item The elements of $\mathcal Y_w$ are $D$-contracting.
		\item If $e$ and $e'$ are two distinct edges of $T$ containing $w$, the projection of $Y_{e'}$ on $Y_e$ has diameter at most $D$.
	\end{itemize}
 
 We establish the four properties appearing in Definitions \ref {def: contracting family} and \ref{def: compatible structure}.
	\paragraph{Uniform contraction.}
	Let $v \in V_1$.
	We are going to prove that $X_v$ is $5D$-contracting.
	Let $c \colon \intval ab \to X$ be a geodesic of $X$ such that $\dist c{X_v} \geq 5D$.
	Then   $p\circ c$ is contained in a single connected component of $T \setminus \{v\}$.
	We let $e=vw$ be the (unique) edge starting at $v$ whose interior is contained in this component.
	Observe that $\Pi_{X_v}(c) = \Pi_{Y_e}(c)$.
 
    If $c$ does not meet $X_w$, then $X_w$ separates it from 
$ X_v$ and there exists an  edge $e'=wv'$ other than $e$ such that $\Pi_{X_v}(c)= \Pi_{Y_e}(c) \subset \Pi_{Y_e}(Y_{e'})$.
	Since $w\in V_0$,  the local peripheral structure $\mathcal Y_w$ of $X_w$ is separated by assumption, so the latter projection has diameter at most $D$.

	Now assume   that $c$ intersects $X_w$. Let $c(a_0)$ and $c(b_0)$ be the entry and exit points of $c$ in $X_w$ (note that  $a\leq a_0\leq b_0\leq b$).
    Say that a subinterval $\intval{a'}{b'} \subset \intval {a_0}{b_0}$ is   an  \emph{excursion interval} if 
    \begin{equation*}
        c([a',b'])\cap X_w=\{c(a'),c(b')\}.
    \end{equation*}
    We call the set of points $c(t)$, where $t\in\intval {a_0}{b_0}$ does not belong to the interior of an excursion interval, the \emph{interior part} of $c$.
    Since $X_w$ is closed, it is exactly the set of points of $c$ contained in $X_w$.

    The argument given above using $\mathcal Y_w$ shows that $\Pi_{Y_e}(c(I))$  has diameter at most $D$ if $I$ is an excursion interval, or $I=[a,a_0]$, or $I=[b_0,b]$.
    So it suffices to prove that the projection of the whole interior part has diameter at most $3D$.
 
    Consider the restriction of $c$ to $\intval{a_0}{b_0}$. 
    We can define a geodesic $c_0$ with the same endpoints but entirely contained in $X_w$, by modifying $c$ on each excursion interval using the convexity of $X_w$. 
    Note that $c_0$ contains the interior part of $c$.
 
    Say that an excursion interval  $\intval{a'}{b'} $ of $c$ is  \emph{bad}   if $c_0([a',b'])$ is at distance less than $2D$ from $Y_e$. 
    There are three cases.
   \begin{itemize}
        \item   
       If there is no bad interval, the whole of $c_0$ is $2D$-far from $Y_e$ (recall that $c$ is $5D$-far from $Y_e$). 
        Since $Y_e$ is $D$-contracting in $X_w$, the projection of $c_0$ (hence of the interior part of $c$) onto $Y_e$ has diameter at most $D$.
       
       \item 
        Now suppose that there are two bad intervals $[a_1,b_1]$, $[a_2,b_2]$, with $a_1<b_1<a_2<b_2$. 
        Choose $s_i\in [a_i,b_i]$ with $d(c_0(s_i),Y_e) < 2D$. 
        By \autoref{res: contraction vs quasi-convex}, the geodesic $c_0(\intval{s_1}{s_2})$ lies in the $7D/2$-neighborhood of $Y_e$. 
        This contradicts the fact that $c_0(b_1) = c(b_1)$   is at least $5D$-far from $X_v$. 
        
        \item 
        The last case is when there is exactly one bad interval $[a_1,b_1]$. 
        Assuming this, we consider $c_0(\intval{a_0}{a_1})$, $c( \intval{a_1}{b_1})$, $c_0( \intval{b_1}{b_0})$. 
        As explained above, the projection of each of these three sets onto $Y_e$ has diameter at most $D$, so their union has diameter at most $3D$. 
        This union contains the projection of the interior part of $c$ and the result is proved. 
    \end{itemize}
	
    \paragraph{Separation.}
	Let $v, v' \in V_1$ be distinct parabolic vertices. 
	We are going to show that $\Pi_{X_v}\left(X_{v'}\right)$ has diameter at most $D$.
	Let $e_1=vw$ and $e_2$ be the first two edges of the geodesic $\geo v{v'} \subset T$. 
    Every path joining  $X_v$ and $X_{v'}$ contains a subpath joining $Y_{e_1}$ to $Y_{e_2}$.
	Consequently
	\begin{equation*}
		\Pi_{X_v}\left(X_{v'}\right) \subset \Pi_{Y_{e_1}}\left(Y_{e_2}\right)
	\end{equation*}
	has diameter at most $D$ because $\mathcal Y_w$ is separated.
    
	\paragraph{Quasi-equivariant projections.}
    We fix $Y = X_v$ with $v\in V_1$, as well as $x\in X$ and $\alpha\in E$.
    We are going to prove that
    \begin{equation*}
		\diam\left( \Pi_{\alpha Y}(\alpha x) \cup \alpha \Pi_Y(x) \right) \leq 3D.
	\end{equation*}
	There exists $v'$ in $ V_0$ or $V_1$ such that $x \in X_{v'}$. 
	We assume $v'\ne v$, as the result is clear  if $v'=v$, and we denote by $e=vw$ the first edge along the geodesic $\geo v{v'} \subset T$. 
	Note that $w\in V_0$.

    By \autoref{eqproj}, we may replace $\alpha$ by some $g\alpha$ and assume that $\alpha$ fixes $e$.
    Then $\alpha$ leaves $Y,Y_e, X_w$ invariant, so that $\Pi_Y(x) = \Pi_{Y_e}(x)$ and $\Pi_{\alpha Y}(\alpha x) = \Pi_{ \alpha Y_e}(\alpha x)$ (we keep writing $\alpha Y_e$ even though $\alpha Y_e=Y_e$). 
    If $v' = w$, the result follows from the properties of $\mathcal Y_w$. 
    Otherwise, we write $e'$ for the second edge along $\geo vw$ and choose a point $y\in Y_{e'}$.
    Observe that
    \begin{equation*}
		\Pi_{Y_e}(x) \cup \Pi_{Y_e}(y) \subset \Pi_{Y_e}(Y_{e'})
		\quad \text{and} \quad
		\Pi_{\alpha Y_e}(\alpha  x) \cup \Pi_{\alpha  Y_e}(\alpha y) \subset \Pi_{\alpha Y_e}(\alpha Y_{e'}).
	\end{equation*}
    Since $\mathcal Y_w$ is separated, $\Pi_{Y_e}(Y_{e'})$ and $\Pi_{\alpha Y_e}(\alpha Y_{e'})$ have diameter at most $D$.
    Since the restriction of  $\alpha$ 
    to $Y_e$ is isometric by essential triviality, we get in particular 
    \begin{equation*}
		\diam\left(\alpha \Pi_{Y_e}(x) \cup \alpha\Pi_{Y_e}(y)\right) \leq D
		\quad \text{and} \quad
		\diam\left(\Pi_{\alpha Y_e}(\alpha  x) \cup \Pi_{\alpha  Y_e}(\alpha y) \right) \leq D.
	\end{equation*}
 	Quasi-equivariance of projections in $\mathcal Y_w$ yields
	\begin{equation*}
		\diam\left( \Pi_{\alpha Y_e}(\alpha y) \cup \alpha \Pi_{Y_e}(y) \right) \leq D,
	\end{equation*}
	whence the result.
    
	\paragraph{Transversality.}
	The proof of transversality uses  a simple lemma about trees, whose proof is left as an exercise.

\begin{lemm}
\label{res: geo starting at non-periodic pt}
    Let $g$ be an isometry of a simplicial tree $S$.
    Let $v,w $ be vertices, with $v$ not periodic under $g$. 
    There is $n_0 \in \N$ such that the following holds  for every integer $n \geq n_0$:
 
    \begin{enumerate}
        \item if $g$ is elliptic, then $\geo v{g^n(w)}$ contains the geodesic $\geo vp$, with $p$  the projection of $v$ on the set of periodic points of $g$; in particular, the first edge of $\geo v{g^n(w)}$ is independent of $n$.
     
        \item if $g$ is   hyperbolic, then  $\geo v{g^n(w)}$ contains the geodesic $\geo vp\cup \geo p {g^2p}$, with $p$ the projection of $v$ on the axis of $g$; in particular, the first two edges of $\geo v{g^n(w)}$ are independent of $n$.
        \qed
    \end{enumerate}
\end{lemm}

    Let  $Y = X_v$, with $v \in V_1$ a parabolic vertex.
	Let $\alpha \in E$ and $x \in X$. 
	Let $w\in V$ be such that $x\in X_w$.
	We assume that no power of $\alpha$ preserves $Y$ (equivalently, $v$ is not $\alpha$-periodic). 
	Let $n_0$ be the integer given by \autoref{res: geo starting at non-periodic pt} applied to $\alpha$ acting on $T$.
	It suffices to prove that $\Pi_Y(\set{\alpha ^n x}{n \geq n_0})$ is bounded.
	
    Let $n \geq n_0$.
    Let $e=vu$ be the first edge on $\geo v{\alpha^n(w)}$ (which does not depend on $n$).
    Note that $u\in V_0$ and $\Pi_Y(\alpha^n x)=\Pi_{Y_e}(\alpha^n x)$.

    First suppose that $u$ is not $\alpha$-periodic (so $u\ne p$ if $\alpha$ is   elliptic). According to \autoref{res: geo starting at non-periodic pt}, the second edge of $\geo v{\alpha^n(w)}$ is also independent of  $n\geq n_0$.
    We denote it by $e'=uv'$. 
    Observe now that $\Pi_{Y_e}(\alpha^n x)$ is contained in $ \Pi_{Y_e}(Y_{e'})$, which is bounded by separation of $\mathcal Y_u$.
  
  Now suppose that $u$ has period $k$ under $\alpha$. Writing \[\Pi_Y(\set{\alpha ^n x}{n \geq kn_0})=\bigcup_{i=0}^{k-1}\Pi_Y\left(\set{\alpha^{kn}(\alpha^i(x))}{n \geq n_0}\right),\] 
we can replace $\alpha$ by $\alpha^k$ and use the points $\alpha^i(x)$ (which belong to $X_{\alpha^i(w)}$), and thus assume that $\alpha$ fixes $u$. 

If $w=u$, boundedness of $\Pi_Y(\set{\alpha ^n x}{n \geq n_0})$
follows from the transversality of $\mathcal Y_u$.  If not, let $e'=uv'$ be the second edge of $\geo{v}{\alpha^n(w)}$. The set $\Pi_{Y_e}( \alpha^n x)$ is contained in $\Pi_{Y_e}( \alpha^n Y_{e'})$, which has diameter at most $D$ by separation of $\mathcal Y_u$, so it suffices to show that $\Pi_{Y_e}(\set{\alpha ^n y}{n \geq n_0})$ is bounded for some fixed $y\in Y_{e'}$. But this is true by transversality of $\mathcal Y_u$.
     This completes the proof of \autoref{res: global peripheral structure}.
\end{proof}

\subsection{Metric estimates}
\label{sec: metric estimates}

We now explain how the global peripheral structure
\begin{equation*}
        \mathcal Y = \set{X_v}{v \in V_1}
    \end{equation*}
    of \autoref {res: global peripheral structure}
can be used to estimate distances in $X$, knowing local data from the vertex spaces $X_v$.  
We fix $D \in \R_+$ with the following properties: every $Y \in \mathcal Y$ is $D$-contracting; the diameter of $\Pi_Y(Y')$ is at most $D$ for any distinct $Y,Y' \in \mathcal Y$.

\begin{lemm}
\label{res: dist decomposition single periph}
	  Let $v,v' $ be vertices of $T$.
	Let $w \in V_1$ be a parabolic vertex on $\geo v{v'}$.
	Let $(x,x') \in X_v \times X_{v'}$,  and let $y$ be a projection of $x$ on $X_w$.
	Then 
	\begin{equation*}
		 { \dist xy + \dist y{x'} 	\leq\dist x{x'}} +   4D.
	\end{equation*}
\end{lemm}

\begin{proof}
	Let $c \colon \intval ab \to X$ be a geodesic from $x$ to $x'$. Since $w$ belongs to $\geo v{v'}$, there exists $t \in \intval ab$ such that $c(t) \in X_w$.
    Since $X_w$  is $D$-contracting, it follows from \autoref{rem: projection contracting set2} that
	\begin{equation*}
		\dist x{c(t)} \geq \dist xy + \dist y{c(t)} -   4D.
	\end{equation*}
	Hence  
	\begin{align*}
		\dist x{x'} 
		=\dist x{c(t)} + \dist{c(t)}{x'} 
		& \geq \dist xy + \dist y{c(t)} + \dist{c(t)}{x'} - 4D \\
		& \geq \dist xy + \dist y{x'} - 4D. \qedhere
	\end{align*}
\end{proof}

\begin{lemm}
\label{res: dist decomposition several periph} 
	Let $v,v' \in V$.
	Let $v_1, v_2, \dots, v_n \in V_1$ be a sequence of pairwise distinct parabolic vertices aligned in this order along $\geo v{v'}$.
	Let $(x_0, x_{n+1}) \in X_v \times X_{v'}$.
	Let $x_1$ be a projection of $x_0$  onto $X_{v_1}$.
	For every $k \in \intvald 2n$, let $x_k$ be a projection  of some point in $X_{v_{k-1}}$ onto $X_{v_k}$.
	Then 
	\begin{equation*}
		\sum_{k=0}^n\dist{x_k}{x_{k+1}} \leq \dist {x_0}{x_{n+1}} +  6nD.
	\end{equation*}
\end{lemm}

\begin{rema}
	The vertices $v_1, \dots, v_n$ are pairwise distinct, but	we allow $ v_1=v$ and $ v_n=v'$.
\end{rema}

\begin{proof}
	According to \autoref{res: dist decomposition single periph}, we have
	\begin{equation}
	\label{eqn: dist decomposition several periph - init}
	    \dist {x_0}{x_1} + \dist {x_1}{x_{n+1}} \leq \dist {x_0}{x_{n+1}} +  4D.
	\end{equation}
	We prove by descending induction that, for every $j \in \intvald 1n$,
	\begin{equation}
	\label{eqn: dist decomposition several periph}
		\sum_{k=j}^n\dist{x_k}{x_{k+1}} \leq \dist {x_j}{x_{n+1}} +  6(n-j)D. \tag{$\Delta_j$} 
	\end{equation}
Combined with (\ref{eqn: dist decomposition several periph - init}), the statement $(\Delta_1)$ will provide the result. Note that $(\Delta_n)$ is obvious.
	Let $j \in \intvald 2n$ for which $(\Delta_j)$ 
	holds.
	Let $p$ be a projection of $x_{j-1}$ on $X_{v_j}$.
	By \autoref{res: dist decomposition single periph} applied with $v=v_{j-1}$ we have
	\begin{equation*}
		\dist {x_{j-1}}p + \dist p{x_{n+1}} \leq \dist{x_{j-1}}{x_{n+1}} +  4D.
	\end{equation*}
	Recall that $x_j$ is a projection  of \emph{some} point in $X_{v_{j-1}}$ on $X_{v_j}$.
	Since the projection of $X_{v_{j-1}}$ on $X_{v_j}$ has diameter at most $D$, we get $\dist p{x_j} \leq D$.
	Hence
	\begin{equation*}
		\dist {x_{j-1}}{x_j} + \dist {x_j}{x_{n+1}}
		\leq  \dist{x_{j-1}}{x_{n+1}} +  6D.
	\end{equation*}
	Adding to $(\Delta_j)$ shows that $(\Delta_{j-1})$ holds.
\end{proof}

Recall that we defined the translation length as $\norm[X] g =\inf_{x\in X} \dist x{gx}$. 
The following lemma tells us how to approximate it. 
It will be used in \autoref{classconj} to understand the growth of conjugacy classes   under iteration of an automorphism.

\begin{lemm}
\label{res: computing translation length}
	Let $g \in G$ be a hyperbolic element (for its action on $T$).
	Let $v \in V_1$ be a parabolic vertex on the axis of $g$.
	Let $z$ be any  point in the projection of $X_v$ onto $gX_v$.
	Then 
	\begin{equation*}
		 { \norm[X] g \geq \dist z{gz}}-  10D.
	\end{equation*}
\end{lemm}

\begin{proof}
	 Let $x \in X_w$, for some $w \in V$. %and $x \in X_w$.
	We have to bound $\dist  x{gx}$ from below.
	The vertex $v$ lies on the axis of $g$.
	Up to translating simultaneously $x$ and $w$ by a power of $g$, we can assume that $v$ belongs to the geodesic $[w, gw)$.
	Let $y$ be a projection of $x$ onto $X_v$.
	It follows from \autoref{res: dist decomposition single periph} that
	\begin{equation*}
		\dist x{gx} 
		\geq \dist x{y} + \dist y{gx} - 4D 
		= \dist {gx}{gy} +\dist y{gx}    -  4D
		\geq \dist y{gy} - 4D.
	\end{equation*}
	Since $\mathcal Y$ is $D$-separated and $z$ lies in the projection of $X_v$ onto $gX_v$, the point  $z$ is $D$-close to any projection   $y'$ of $y$ onto $gX_v$.
	We get from \autoref{rem: projection contracting set2} that
    \begin{align*} 
         \dist y{gy} & \geq \dist y{y'}+\dist {y'}{gy}  - 4D \\
         & \geq  \dist yz+\dist z{gy}  - 6D \\
         & \geq \dist{gy}{gz} + \dist z{gy}  - 6D \\
         & \geq \dist z{gz} - 6D.
    \end{align*}
	Consequently $\dist x{gx} \geq \dist z{gz} -  10D$. 
	This   holds for every $x \in X$, whence the result.
\end{proof}

 The next lemma provides a way to decompose a palangre into two ``simpler'' palangres.
 By iterating this decomposition, we will be able later to reduce our understanding of palangre growth in $X$ to that of palangre growth in the local spaces $X_v$ (see \autoref{ee}).

\begin{lemm}
\label{chacal}
	Let $\alpha,\beta$ be in  $E = G \rtimes_\phi\Z$, with $\pi(\alpha)=\pi(\beta)$. 
	Let $e=v_0v_1$ be an edge of $T$ with $v_1\in V_1$ parabolic. 
	Call $m$ its midpoint. 
	Let $v,v'$ be vertices such that, for all but finitely many $n \in \N$, the points $\alpha^{-n}v$ (\resp $\beta^{-n}v'$) belong to the same component of $T\setminus\{m\}$ as $v_0$ (\resp $v_1$). 
	Let $\gamma\in G\alpha$ acting as the identity on $Y_e$.
	
 	If no power of $\alpha$ fixes $v_1$, then for every $x,x', y \in X$ we have
	\begin{equation*}
		\dist x{\alpha^{n}\beta^{-n}x'} \asymp \dist x{\alpha^n \gamma^{-n}y} + \dist y{\gamma^n \beta^{-n}x'}.
	\end{equation*}
\end{lemm}

\begin{rema}
	Note that the existence of $\gamma$ is guaranteed by the fact that, in a $\varphi$-adapted metric decomposition of $G$, the action of every element of $E$  is essentially trivial in restriction to $Y_e$ (Item~(4) from \autoref{MSW}).
\end{rema}
 
\begin{proof}
	For simplicity, we write $Y = X_{v_1}$.
	Recall that $\alpha^n\beta^{-n}$ and $\gamma^n \beta^{-n}$ belong to $G$,  which acts isometrically on $X$.
	Hence, without loss of generality, we can assume that $x \in X_v$, $x' \in X_{v'}$.
	
	For every $n \in \N$, denote by $p_n$ and $q_n$ projections of $\alpha^{-n}x$  on $Y$  and of $x$  on  $\alpha^nY$ respectively.
	By assumption, for sufficiently large $n \in \N$, the vertex $v_1$ lies on the geodesic $[\alpha^{-n}v, \beta^{-n}v']$, and similarly $\alpha^nv_1 \in [v,\alpha^n \beta^{-n}v']$.
	Consequently, $p_n$ and $q_n$ actually belong to $Y_e$ and $\alpha^nY_e$.
	Moreover, combining \autoref{res: dist decomposition single periph} with the quasi-equivariance of projections, we get
	\begin{align*}
    		\dist x{\alpha^n \beta^{-n}x'}
    		& \asymp \dist x{q_n} + \dist {q_n}{\alpha^n \beta^{-n}x'} \\
    		& \asymp \dist x{\alpha^n p_n} + \dist {\alpha^np_n}{\alpha^n \beta^{-n}x'}.
    	\end{align*} 
	Recall that $\gamma$ fixes $p_n$.
	Hence 
	\begin{equation*}
    		\dist x{\alpha^n \beta^{-n}x'}
    		\asymp \dist x{\alpha^n \gamma^{-n}p_n} + \dist {\alpha^n\gamma^{-n}p_n}{\alpha^n \beta^{-n}x'}.
    	\end{equation*} 
	It follows from transversality that the sequence  $(p_n)$ is bounded.
	Since every $\alpha^n \gamma^{-n} \in G$ is an  isometry  of $X$, we get
	\begin{align*}
    		\dist x{\alpha^n \beta^{-n}x'}
		& \asymp \dist x{\alpha^n \gamma^{-n}y} + \dist {\alpha^n\gamma^{-n}y}{\alpha^n \beta^{-n}x'} \\
    		& \asymp \dist x{\alpha^n \gamma^{-n}y} + \dist {y}{\gamma^n \beta^{-n}x'}. \qedhere
    	\end{align*} 
	
\end{proof}

\subsection{Proof of \autoref{res: combination growth}}
\label{sec: proof combination growth}

 We can now prove \autoref{res: combination growth}. We let 
	$\mathcal Y = \set{ X_v}{v \in V_1}$
be the (global) peripheral structure provided by 
\autoref{res: global peripheral structure}.
To simplify notations, we fix for every $Y \in \mathcal Y$ a projection map $q_Y \colon X \to Y$, i.e. we choose a point $q_Y(x) \in \Pi_Y(x)$ for every $x \in X$.

We start the proof of  \autoref{res: combination growth} by  the following statement, which shows that  $E$ satisfies the `Palangres' property of \autoref{palgeo2} with respect to $X$. The `Classes' property will be proved in \autoref{classconj}.
Recall 
 that, by assumption,  $E_v$ has \total{} with respect to $X_v$ for every $v \in V$,
with spectrum $\Lambda_{v}$ and palangre spectrum $\Lambda_{\mathrm{pal},v}$.

\begin{prop}
\label{res: growth general/general}
    Let $x \in X$.
	Let $\alpha, \beta \in E$ with $\pi(\alpha) = \pi(\beta)$ positive.
	There exists $(d,\lambda) \in \N \times [1, \infty)$ such that   $ \dist x{\alpha^n \beta^{-n}x}$ grows like $n^d \lambda^{n\pi(\alpha)}$.
	
	Moreover, denoting   
	\begin{equation*}
	    \breve \Lambda_{\rm pal}= {\bigcup_{v \in V} \Lambda_{\mathrm{pal},v}},
	\end{equation*}
	we have  $(d, \lambda) \in \breve\Lambda_{\rm pal}$ if   $\alpha$ and $\beta$ both act elliptically on $T$, and $(d, \lambda) \in \breve\Lambda_{\rm pal}^+$ otherwise.
\end{prop}
 
By Remarks~\ref{rk: growth-type-power} and \ref{rem: spectrum-power}, the reader not interested in the spectra may ignore the terms $\pi(\alpha)$ and $\pi(\beta)$.

We split the proof into three cases, depending on the nature of $\alpha$ and $\beta$ acting as isometries of $T$: both elliptic, 
hyperbolic/elliptic,
or both hyperbolic. 

Since $\dist x{\alpha^n\beta^{-n}x}=\dist x{\beta^n\alpha^{-n}x}$, we will be free to swap $\alpha$ and $\beta$ when needed.

\subsubsection{Elliptic-elliptic pairs}
\label{ee}

	The proof in this case is by induction on the distance between the sets of periodic points $\per \alpha$ and $\per \beta$ in $T$.
	Suppose first that $\per \alpha \cap \per \beta \neq \emptyset$.
	By definition there exist $k \in \N \setminus\{0\}$ and a vertex $v$ of $T$  such that both $\alpha^k$ and $\beta^k$ belong to $E_v$. 
	 Since $\varphi$ sends $G_v$   to a conjugate, 
	 the images of $\alpha^k$ and $\beta^k$ under the canonical projection $E_v \onto \Z$ are both equal to $k \pi(\alpha)$, even though $\alpha$ and $\beta$ do not necessarily belong to $E_v$.
	According to the assumption on  $E_v$, there is $(d,\lambda) \in \Lambda_{\mathrm{pal},v}\inc\breve \Lambda_{\rm pal}$ such that, for every $x \in X_v$,
	\begin{equation*}
		\dist x{\alpha^{kn} \beta^{-kn} x} \asymp n^d\lambda^{kn\pi (\alpha)}.
	\end{equation*}
	(Recall that $X_v$ is convex in $X$, hence the asymptotic behavior  of growth is the same, regardless of whether it is computed in $X_v$ or in $X$.)
	The conclusion now follows from \autoref{rk: growth-type-power}. 
	
	Suppose now that $\per \alpha$ and $\per \beta$ are disjoint.
	Denote by $\geo vw$ the shortest geodesic from $\per\alpha$ to $\per\beta$.
	Let $u \in V_1$ be a parabolic vertex on $\geo vw$.
	Up to permuting $\alpha$ and $\beta$, we may assume that $u \neq v$.
	We write $e$ for the first edge of $\geo uv$.
	We choose $\gamma \in G\alpha$ acting as the identity on $Y_e$ (this is possible by essential triviality of the action on edge spaces).
	In particular $\gamma$ fixes $e$, thus its endpoints.
	Therefore the distance between $\per \alpha$ and $\per \gamma$ -- \resp between $\per \gamma$ and $\per \beta$ -- is smaller than the one between $\per \alpha$ and $\per \beta$.
	
	Let $x \in X$.
	By induction, there are $(d_1, \lambda_1), (d_2, \lambda_2) \in \breve \Lambda_{\rm pal}
$ such that 
	\begin{equation*}
		\dist x{\alpha^n \gamma^{-n}x} \asymp n^{d_1} \lambda_1^{n \pi(\alpha)}
		\quad \text{and} \quad
		\dist x{\gamma^n \beta^{-n}x} \asymp n^{d_2} \lambda_2^{n \pi(\beta)}.
	\end{equation*}
	By construction, no power of $\alpha$ fixes $e$.
	It follows from \autoref{chacal} that 
	\begin{equation*}
		\dist x{\alpha^n \beta^{-n}x} 
		\asymp \dist x{\alpha^n \gamma^{-n}x}  + \dist x{\gamma^n \beta^{-n}x} 
		\asymp n^d \lambda^{n \pi(\alpha)},
	\end{equation*}
	where $(d, \lambda)$ is the largest pair between $(d_1, \lambda_1)$ and $(d_2, \lambda_2)$.
This proves \autoref{res: growth general/general} in the elliptic/elliptic case.

\begin{coro}
\label{res: growth edge spaces}
	Suppose that   $(y,y') \in Y_e \times Y_{e'}$, with $e,e'$   two edges of $T$.
	For every $\alpha \in E$  with $\pi(\alpha) \geq 1$, there is $(d, \lambda) \in \breve \Lambda_{\rm pal}$ such that $  \dist{\alpha^n y}{\alpha^n y'}$ grows like $n^d \lambda^{n \pi(\alpha)}$.
\end{coro}

\begin{proof}
	Choose $\gamma, \gamma' \in G\alpha$ fixing $Y_e$ and $Y_{e'}$ pointwise  respectively,  hence acting elliptically on $T$. 
	Then 
	\begin{equation*}
		\dist{\alpha^n y}{\alpha^n y'}
		= \dist{\alpha^n \gamma^{-n}y}{\alpha^n {\gamma'}^{-n}y'}
		= \dist y{\gamma^n{\gamma'}^{-n} y'},
	\end{equation*}
and the result follows. 
\end{proof}

\subsubsection{Hyperbolic-elliptic pairs}	

We now assume that  $\alpha$ is hyperbolic and $\beta$ elliptic (for their action on $T$). 
	Let $v$ be a vertex of $T$ fixed by $\beta$, and $w$ its projection on the axis of $\alpha$.
	Let $u \in V_1$ be a parabolic vertex on the geodesic $\geo w{\alpha w}$.
	Observe that $u, \alpha u, \dots, \alpha^{n-1} u$ are aligned in this order along the geodesic $\geo v{\alpha^n\beta^{-n}v}$.

    Fix $x \in X_v$, $y \in X_u$, and  set $Y=X_u$.
	According to \autoref{res: dist decomposition several periph} we get
	\begin{align*}
    		\dist x{\alpha^n \beta^{-n}x}
    		 =\ &\dist x{q_Y(x)} + \dist{q_Y(x)}{q_{\alpha Y}(y)} \\
    		& + \sum_{k = 1}^{n-2} \dist{q_{\alpha^k Y}(\alpha^{k-1}y)}{q_{\alpha^{k+1}Y}(\alpha^k y)} \\
    		& + \dist{q_{\alpha^{n-1}Y}(\alpha^{n-2}y)}{\alpha^n\beta^{-n}x} + O(n),
    	\end{align*}
	where $O(n)$ grows  {at most} linearly in absolute value.
	Using quasi-equivariance of projections we can write
	\begin{align*}
    		\dist x{\alpha^n \beta^{-n}x}
    		= & \sum_{k = 1}^{n-2} \dist{\alpha^k q_Y(\alpha^{-1}y)}{\alpha^kq_{\alpha Y}(y)} \\
    		& + \dist{\alpha^nq_{\alpha^{-1}Y}(\alpha^{-2}y)}{\alpha^n\beta^{-n}x} + O(n).
    	\end{align*}
	We study the growth of each term separately.
	
	First observe that $q_Y(\alpha^{-1}y)$ and $q_{\alpha Y}(y)$   belong to  $Y_e$ and $Y_{e'}$ for some   edges $e,e'$ of $T$.
	By \autoref{res: growth edge spaces}, there is $(d_0,\lambda_0) \in \breve \Lambda_{\rm pal}$ such that 
	\begin{equation*}
		\dist{\alpha^k q_Y(\alpha^{-1}y)}{\alpha^kq_{\alpha Y}(y)} \asymp k^{d_0}\lambda_0^{k\pi(\alpha)}
	\end{equation*} 
	as $k$ tends to infinity. 
	The sum of these terms when $k$ runs over $\intvald 1{n-2}$ then grows like $n^{d_1}\lambda_1^{n\pi(\alpha)}$ 
	where $(d_1, \lambda_1) = (d_0, \lambda_0)$ if $\lambda_0 > 1$, and $(d_1, \lambda_1) = (d_0 + 1, 1)$ otherwise.
	In both cases $(d_1, \lambda_1)$ belongs to $\breve\Lambda_{\rm pal}^+$.
	
	For the second term, we choose $\gamma \in  G\alpha$ fixing $q_{\alpha^{-1}Y}(\alpha^{-2}y)$.
	Since $\alpha^n \gamma^{-n} \in G$ acts isometrically on  $X$, we have 	
	\begin{align*}
    		\dist{\alpha^nq_{\alpha^{-1}Y}(\alpha^{-2}y)}{\alpha^n\beta^{-n}x}
    		& = \dist{\alpha^n\gamma^{-n}q_{\alpha^{-1}Y}(\alpha^{-2}y)}{\alpha^n\beta^{-n}x} \\
    		& = \dist{q_{\alpha^{-1}Y}(\alpha^{-2}y)}{\gamma^n\beta^{-n}x}.
    	\end{align*}
	Both $\beta$ and $\gamma$ act elliptically on $T$, and $\pi(\gamma)=\pi(\alpha)=\pi(\beta)$, so we have seen that 
	the above term grows like $n^{d_2}\lambda_2^{n\pi(\beta)}$ for some $(d_2, \lambda_2) \in \breve \Lambda_{\rm pal}$.
	
	Combining our two estimates, we observe that, up to an    error term which grows  {at most} linearly,  $\dist x{\alpha^n \beta^{-n}x}$ grows like $n^d\lambda^{n\pi(\alpha)}$ where $(d,\lambda) \in \breve\Lambda_{\rm pal}^+$ is the largest pair between $(d_1, \lambda_1)$ and $(d_2, \lambda_2)$.
	
	The linear error may be neglected if $n^d\lambda^{n\pi(\alpha)}$ grows at least quadratically. 
	If not, we find that $\dist x{\alpha^n \beta^{-n}x}$ grows at most linearly. 
	However $\dist v{\alpha^n \beta^{-n}v} = \dist v{\alpha^nv}$ grows linearly because $\alpha$ acts hyperbolically on $T$.
	Since the projection $p \colon X \to T$ is Lipschitz, it follows that $\dist x{\alpha^n \beta^{-n}x}$ grows at least linearly, hence exactly linearly.
	Recalling that the growth type $(1,1)$, corresponding to linear growth,  belongs to $\breve\Lambda_{\rm pal}^+$, the proposition in the hyperbolic-elliptic case  follows.

\subsubsection{Hyperbolic-hyperbolic pairs}

   The last case is when  $\alpha$ and $\beta$ both act hyperbolically on $T$.  Suppose first that the intersection of the respective axes of $\alpha$ and $\beta$ has infinite length.
       The group generated by $\alpha$ and $\beta$  then fixes an end of $T$. By acylindricity (\autoref {res: pt stabilizer ray}) it  is virtually cyclic (see Lemma 7.9 of    \cite{Guirardel:2017te}), so there exists $k$ 
    such that 
      $\alpha^k$ and $\beta^k$ commute.

	For $x\in X$ we then have 
	\begin{equation*}
		\dist x{\alpha^{kn} \beta^{-kn}x}
		=\dist x{g^nx}
	\end{equation*}
	where $g = \alpha^{k}\beta ^{-k}$ belongs to $G$.
	But $G$ acts by isometries on $X$.
	It follows from the triangle inequality that $\dist x{g^nx}$ grows at most linearly.
	However, the analysis of the previous subsections, 
	applied with $\alpha' = g$ and $\beta' = 1$, ensures that $\dist x{g^nx}$ grows exactly like a polynomial,
	hence 
	is either bounded or grows linearly.
	The result follows  from \autoref{rk: growth-type-power} 
	since the growth types 
	$(0,1)$ and $(1,1)$ belong to $\breve\Lambda_{\rm pal}^+$.%
	
	We now assume that the axes of $\alpha$ and $\beta$ have   empty or bounded intersection.
	We fix two vertices $v$ and $w$ on the respective axes of $\alpha$ and $\beta$.
	There exists an edge $e$ of $T$ such that, for all but finitely many $n \in \N$, the edge $e$ lies on the geodesic
	$[\alpha^{-n} v,\beta^{-n}w]$. 
	Let $\gamma \in G\alpha$ acting as the identity on $Y_e$.
	Up to permuting $\alpha$ and $\beta$, we get from \autoref{chacal} that 
	\begin{equation*}
		\dist x{\alpha^n\beta^{-n}x} 
		\asymp \dist x{\alpha^n \gamma^{-n}x} + \dist x{\gamma^n \beta^{-n}x}.
	\end{equation*}
	Recall that $\alpha$ and $\beta$ are hyperbolic (for their action on $T$) while $\gamma$ is elliptic.
    Thus there are $(d_1, \lambda_1), (d_2, \lambda_2) \in \breve\Lambda_{\rm pal}^+$ such that the two terms in the right hand side of the above estimate grow like $n^{d_1}\lambda_1^{n \pi(\alpha)}$ and $n^{d_2}\lambda_2^{n \pi(\beta)}$ respectively, and the result follows, completing the proof of \autoref{res: growth general/general}.

\subsubsection{Growth of classes}\label{classconj}

Combined with \autoref{res: growth general/general}, the next two propositions complete the proof of \autoref{res: combination growth}.

\begin{prop}\label{clas elliptic}
    Let $g \in G$.
    If $g$ fixes a vertex $v$ in $T$, then $\norm[X]{\phi^n (g)} \asymp \norm[X_v]{\phi^n (g)}$
\end{prop}

\begin{proof}
Note that 
    $\norm[X]{\phi^n (g)}$ is equal to $\norm[X]{\alpha^n g \alpha^{-n}}$ for any $\alpha$ such that $\pi(\alpha)=1$. For any such $\alpha$, we will write $g_n =  
	\alpha^n g \alpha^{-n}
		$.
	We distinguish two cases.
	
	Assume first that $g$ fixes a parabolic vertex $v \in V_1$.
By suitably choosing    $\alpha$,
we can assume that $\alpha$ fixes $v$ as well, so that $g_n \in G_v$ for every $n \in \N$.
    Note that $\norm[X]{g_n} \asymp \norm[X_v]{g_n}$ because $X_v$ is contracting  (see \autoref{rem: growth in qc sbgp}) whence the result.

	Suppose now that $g$ is elliptic, but fixes no parabolic vertex.
	Since $T$ is bipartite, $g$ fixes a unique non-parabolic vertex $v \in V_0$.
	As before, we can   assume that $\alpha$ fixes $v$ as well, and we then claim the actual equality
	\begin{equation*}
		\norm[X] {g_n} = \norm[X_v]{g_n}, \quad \forall n \in \N.
	\end{equation*}
	
	Let $n \in \N$.
	Let $x \in X$.
	It   belongs to $X_w$ for some vertex $w \in T$.
	Recall that $v$ is the unique vertex of $T$ fixed by $g$.
	Since $\alpha$ also fixes $v$,
the point $v$ is the unique vertex of $T$ fixed by  $g_n= {\alpha^n g\alpha^{-n}}$.
	Thus it is the midpoint of the geodesic $\geo w{g_nw}$.
	Therefore any geodesic $\geo x{g_nx}$ crosses $X_v$, say at a point $y$.
	It then follows  from the triangle inequality that
	\begin{equation*}
	    \norm[X_v]{g_n}
		\leq \dist y{g_ny}
		\leq \dist y{g_nx} + \dist {g_nx}{g_ny}
		= \dist xy + \dist y{g_nx}
		= \dist x{g_nx}.
	\end{equation*}
	Taking the infimum over all points $x \in X$, we get $\norm[X_v]{g_n} \leq \norm[X]{g_n}$.
	The converse inequality is obvious, proving the claim.
\end{proof}

\begin{prop}\label{clas}
	For every $g \in G$, there is $(d,\lambda) \in
	\bigcup_{v \in V} (\Lambda_v\cup\Lambda_{\mathrm{pal},v})$ such that $\norm[X]{\phi^n (g)}$ grows like $n^d \lambda^{n  }$.
\end{prop}

\begin{proof}
As in the previous proof we write $g_n =  
	\alpha^n g \alpha^{-n}
		$ for some $\alpha \in E$ with $\pi(\alpha) = 1$.
	In view of \autoref{clas elliptic}, we can assume that $g$ acts hyperbolically on $T$.
  
	Let $v \in V_1$ be a parabolic vertex along the axis of $g$.
	For simplicity we let $Y = X_v$.
	We fix $x \in g^{-1}Y$ and we let $z = q_Y(x)$, so that $z$ belongs to the projection of $g^{-1}Y$ on $Y$.
	Observe that $\alpha^n v$ is a parabolic vertex on the axis of $g_n$.
	Moreover $z_n = q_{\alpha^nY}(\alpha^n x)$ is a point in the projection of 
	$\alpha^ng\m Y=g_n^{-1}\alpha^n Y$ onto $\alpha^n Y$.
	Combining \autoref{res: computing translation length} and the quasi-equivariance of projections, we get 
	\begin{align*}
		\norm[X]{g_n} 
		\asymp \dist {z_n}{g_nz_n}
		& \asymp \dist {q_{\alpha^nY}(\alpha^n x)}{\alpha^n g\alpha^{-n}q_{\alpha^nY}(\alpha^n x)} \\
		& \asymp \dist {\alpha^nq_{Y}(x)}{\alpha^n gq_{Y}(x)}.
	\end{align*}
	By construction $q_{Y}(x)$ is   a point of $Y_e$ for a suitable edge $e$ starting at $v$,  so 
the result follows from \autoref{res: growth edge spaces}.
\end{proof}

\subsection{\Total{}  in the one-ended case}\label{sec: endgame-one-ended}

  We can now prove the     main result of \autoref{part: one-ended}.

\begin{theo}
\label{res: recap one-ended groups}
Let $G$ be a one-ended toral relatively hyperbolic group. Then every $\varphi\in\Aut(G)$ has \total{}. 

Moreover there exists $K$,  depending only on $G$, such that 
the spectrum $\Lambda $ of $\varphi$ (as well as its palangre spectrum) satisfies the following properties:
\begin{enumerate}
    \item For every $(d,\lambda)\in\Lambda $, one has $d\leq K$, and $\lambda$ is an algebraic unit of degree at most $K$.
    \item One has $|\Lambda |\leq K$.
 \item If $G$ is hyperbolic, 
every $(d,\lambda)$ in the spectrum is $(0,1)$ or $(1,1)$ (bounded or linear growth), or   is of the form $(0,\lambda)$, with $\lambda$ an $r^{\rm th}$ root of the dilation factor of a pseudo-Anosov homeomorphism on a compact surface $\Sigma$ 
 (with $r$ and $\abs{\chi(\Sigma)}$   bounded by $K$).
\end{enumerate}
\end{theo}

\begin{proof}

We have seen (\autoref{pjsj3}) that there exists $k\geq1$, depending only on $G$, such that $\varphi^k$ is pure in the sense of \autoref{rtl}.  In view of \autoref{puissok} and \autoref{rem: spectrum-power}, we may therefore  assume with no loss of generality  that $\varphi$ is pure.

    \autoref{MSW} provides a $\varphi$-adapted metric decomposition $p:X\to T$ of $G$ (where $T$ is the refined JSJ tree of $\varphi$ as in \autoref{pjsj3}).  
    The number of vertices in $T/G$ is bounded in terms of $G$ only by \autoref{access}.
    
    We have  established in \autoref{sommets} 
     that, for every vertex $v$ of $T$, if we    let $\psi_v$ be a representative of the outer class of $\varphi$ such that $\psi_v(G_v)=G_v$, then the restriction of $\psi_v$ to $G_v$ has \total{}, and its spectrum and palangre spectrum satisfy   properties 1 and 2 of the theorem (see Remarks~\ref{unit} and~\ref{rem: spectrum pA homeo}). The same is therefore true for $\varphi$ by  Theorems \ref{thm: combination without spectra} and \ref{res: combination growth}.
    
     The conclusion in the case where $G$ is hyperbolic follows from the fact that there are no abelian vertices in $T$ with stabilizer $\Z^k$ with $k\ge 2$, so all the growth in the vertex groups comes from surfaces $\Sigma$, as in \autoref{res: growth surface}. The complexity of $\Sigma$ is bounded by  \autoref{access}.
\end{proof}

\part{Infinitely-ended groups}\label{part: infinitely-ended}

This part is devoted to the proof of  \autoref{enplus} in the case where $G$ has infinitely many ends.

 We will   consider a decomposition $G=G_1*\dots*G_q*\F_N$ of an arbitrary  finitely generated group $G$
as a free product, with $\F_N$ free,  and an automorphism $\Phi\in\Out(G)$ fixing each conjugacy class $[G_j]$.
For $G$ toral realtively hyperbolic we will use the Grushko decomposition, with each  $G_j$ one-ended; any $\Phi$  then has a power fixing each $[G_j]$.

The main technical tool in the proof will be completely split train tracks (simply  abbreviated as CT's) in the sense of \cite{Feighn:2011tt} (for free groups) and \cite{Lyman:2022wz} (for general free products).

\section{Completely split train track maps (CT's)} \label{cts}

In this section we   explain what a CT is, and we review the properties that we will use. At the end of the section we explain why every automorphism of  an infinitely-ended  toral relatively hyperbolic group has a power which may be represented by a CT. 

We follow the   terminology introduced in \cite{Feighn:2011tt} and \cite{Lyman:2022wz}.
(For experts we mention  the following differences: we will drop the word ``almost'' in almost Nielsen paths and almost INP's, and the words ``maximal taken'' when considering connecting paths; we view fixed edges and non-growing exceptional paths as INP's.)

Starting with a decomposition $G=G_1*\dots*G_q*\F_N$ as above,
we view $G$ as the fundamental group of a finite \emph{graph of groups $\Gamma$}  with trivial edge groups.
For each $j\in\{1,\dots,q\}$, there is a vertex $v_j$ with vertex group $G_j$; the other vertex groups are trivial. The vertices $v_j$ will be called \emph{fat} vertices (this terminology is not in \cite{Lyman:2022wz}).  We sometimes view $\Gamma$ as a topological graph $\Gamma_{top}$, with   fundamental group   $\F_N$. The group $G$ acts on the Bass-Serre tree $T$ with trivial edge stabilizers (as usual we assume that the decomposition is minimal, i.e.\ there is no proper invariant subtree).

In this section, we consider oriented edges $e$.
The vertices $o(e)$ and $t(e)$  stand for the origin and   terminal point of   $e$ respectively. 
The opposite edge is denoted by $\bar e$.

A \emph{path} $\gamma$ %of length $p$ 
in $\Gamma$ is a sequence $g_0e_1g_1\dots  e_pg_p$ where $e_1,\dots, e_p$ are  edges with $t(e_i)=o(e_{i+1})$ and $g_i$ is an element of the group carried by $t(e_i)$ for $i>0$ (with $g_0$ in the group carried by $o(e_1)$). We set $o(\gamma)=o(e_1)$ and $t(\gamma)=t(e_p)$.
The path is \emph{trivial} if $p=0$. We write $\gamma\m$ for the path $g_p\m\bar e_p\dots \bar e_1 g_0\m$.

A \emph{circuit} is a sequence
$g_0e_1g_1\dots  e_pg_p$ as above with the extra condition that $o(e_1)=t(e_p)$, up to cyclically permuting the indices and replacing $(g_0,g_p)$ by $(g_pg_0,1)$ or $(1,g_pg_0)$.
We will assume $g_0=1$ when convenient.
By abuse, we will often think of a circuit as a path 
whose origin and terminal point coincide.
 Since $G$ is the fundamental group of $\Gamma$, any circuit represents a conjugacy class in $G$.

A sequence $e_ig_ie_{i+1}$ is a \emph{turn} of $\gamma$ at $t(e_i)$. It is \emph{degenerate} if $g_i=1$ and $e_{i+1}=\bar e_i$. %, \emph{fat} if $t(e_i)$ is fat.  
 A path is \emph{tight} if it contains no degenerate turn. Since a non-tight path may be tightened in the obvious way, we always assume that paths are tight. 
A circuit $e_1g_1\dots  e_pg_p$ is \emph{tight} if it is tight as a path and moreover $e_1\ne \bar e_p$ if $g_p=1$.

 If $\gamma,\gamma'$ are paths with $t(\gamma)=o(\gamma')$, we can consider  their (possibly non-tight) \emph{concatenation} $\gamma\gamma'
 =g_0e_1g_1\dots  e_p (g_pg'_0) e'_1g'_1\dots  e'_{p'}g'_{p'}$, replacing 
$g_p$ and $g'_0$ by their product in the relevant vertex group.

Paths (not circuits) will often be viewed up to \emph{equivalence}, with two paths equivalent if they differ only by the values of $g_0$ and $g_p$.

To define a  CT, one must specify:  

\begin{itemize}
 
 \item a graph of groups $\Gamma$ as above;
 
 \item for every vertex $v$, a vertex $f(v)$   with the additional requirement that $f(v)=v$ if $v$ is fat;
 \item for every edge $e$, a non-trivial path $f(e)$ joining $f(o(e))$ to $f(t(e))$;  
  it is required that the resulting global map   $f:\Gamma\to\Gamma$    induces  a homotopy equivalence $f_{top}$ of   the underlying topological graph $\Gamma_{top}$;
 \item for each $j\in\{1,\dots,q\}$, an automorphism $\phi_j$ of the vertex group $G_j$ carried by the fat vertex $v_j$.  
\end{itemize}

We usually denote a CT by  $f:\Gamma\to \Gamma$, with  the $\phi_j$'s implicit. 
A CT  must satisfy many properties (\cite{Feighn:2011tt} and \cite{Lyman:2022wz}), some of which we now review.

Given a CT, one can consider the \emph{tightened image $f_\sharp(\gamma)$} of any path $\gamma=g_0e_1g_1\dots  e_pg_p$: %where $e_1,\dots e_p$: 
one replaces each $e_i$ by its image, each $g_i$ by $h_i = \phi_j(g_i)$ whenever $g_i$ belongs to some $G_j$ and by $h_i = 1$ otherwise;  one multiplies the last group element in $f(e_i)$ by $h_i$ and the initial element of $f(e_{i+1})$; and one tightens. One defines the tightened image of a circuit similarly.

This 
 yields a well-defined outer automorphism $\Phi$ of $G$, viewed as the fundamental group of the graph of groups $\Gamma$. 
Each conjugacy class $[G_j]$ is preserved;
thus, for every $j$, the automorphism $\Phi$ has a representative in $\aut G$ agreeing with $\varphi_j \in \aut{G_j}$ on $G_j$.
We say that the CT $f$ \emph{represents} $\Phi$,  and  we call $\varphi_j$ the \emph{$j^{\rm th}$-component} of $\Phi$ (and also of  its representatives $\varphi$).

A concatenation $\gamma=\gamma_1\gamma_2\dots\gamma_p$ is a \emph{splitting} of $\gamma$ if
\begin{equation*}
    f^k_\sharp(\gamma_1\gamma_2\dots\gamma_p)=f^k_\sharp(\gamma_1)f^k_\sharp(\gamma_2)\dots f^k_\sharp(\gamma_p), \quad \forall k \in \N,
\end{equation*}
i.e.\ $\gamma$ is tight and there is no cancellation of edges between the tightened images of $\gamma_i$ and $\gamma_{i+1}$ by powers of $f$.
 We then write $\gamma=\gamma_1\cdot\gamma_2 \cdot\ldots\cdot \gamma_p$.

There is a  filtration $\es=\Gamma_0\inc \Gamma_1\inc\dots\inc \Gamma_m=\Gamma_{top}$ by (possibly non-connected) $f_{top}$-invariant subgraphs. The \emph{$r^{\rm th}$ stratum} is the closure $H_r$ of $\Gamma_r\setminus \Gamma_{r-1}$. The \emph{height} of a path $\gamma$ is the smallest $r$ such that $\gamma\inc\Gamma_r$ (i.e.\ the edges of $\gamma$ belong to $\Gamma_r$). The invariance of $\Gamma_r$ implies that the height of $f\ti(\gamma)$ is at most the height of $\gamma$ (all edges of $f\ti(\gamma)$ are in  $\Gamma_r$ if $\gamma$ has height $r$).

There are three types of strata:

\begin{itemize}
\item   EG stratum: the transition matrix of $f_{top}$ on $H_r$ is irreducible with Perron-Frobenius eigenvalue $\lambda_r>1$;

 \item NEG stratum: $H_r$ consists of a single edge $e$, and (up to replacing $e$ by $\bar e$) $f(e)=ge\cdot u$ with $g $ in the vertex group carried by $o(e)$ and $u$ a (possibly trivial) path of height less than $r$ (note that $ge\cdot u$ is required to be a splitting);
 
 \item zero stratum:  $f_{top}(H_r) $ has height less than $r$.
 
 \end{itemize}
 
 We say that an edge is an EG edge, NEG edge, zero edge according to the type of the stratum that contains it. The fact that   $u$ (in the NEG case) and $f_{top}(H_r) $ (in the zero case) have  height strictly less than $r$ makes inductive arguments possible.  If an NEG edge $e$ is contained in a path $\gamma$, it receives a preferred orientation and its image may be either $ge\cdot u$ or $u\cdot eg$.
 
 A  non-trivial path $\gamma$ is a \emph{Nielsen path} (almost Nielsen path in \cite{Lyman:2022wz}) if $f_\sharp(\gamma)$ is equivalent to $\gamma$: they differ only by the values of $g_0$ and $g_p$ (we view Nielsen paths up to equivalence). 
 An \emph{indivisible Nielsen path} (INP) is a Nielsen path which cannot be split   into two Nielsen paths.
 Any Nielsen path is a concatenation of INP's.
 Unlike Feighn-Handel \cite{Feighn:2011tt} and Lyman \cite{Lyman:2022wz}, we consider  an edge $e$ in an NEG stratum with $f(e)=geg'$ (where $g$ and $g'$ are elements of the relevant vertex groups) as an INP. 
 
 An edge $e$ in an NEG stratum is \emph{linear} if $f(e)=ge\cdot u$ with $u$ a Nielsen path (if $u$ is trivial, we consider $e$ as an INP, not as a linear edge). 

A path $\gamma$ is an \emph{exceptional path} if $\gamma=gew^p\bar e'g'$ where:

\begin{itemize}
\item $g,g'$ are elements of the relevant vertex groups;
 \item $p\in\Z$  and $w$ is a Nielsen path with $f_\sharp(w)=w$ (equality, not just equivalence);
 \item $e,e'$ are linear edges with  $f(e)=he\cdot w^{d}$ and $f(e')=h'e'\cdot w^{d'}$ for some positive, distinct,  integers $d,d'$ and $h,h'\in G$.
\end{itemize}

  Unlike Feighn-Handel and Lyman,  we  require $d\ne d'$  (if $d=d'$, we  view $\gamma$ as an INP, not   an exceptional path).
Note that $w$ must be a circuit, and $f^k_\sharp(\gamma)$ is the exceptional path $g_kew^{p+k(d-d')}\bar e'g'_k$ for some group elements $g_k,g'_k$.

Given a zero stratum $H_r$, the theory distinguishes certain paths contained in $H_r$, called ``maximal taken connecting paths''. We  simply call them \emph{connecting paths}. Zero strata are contractible (because $f_{top}$ is a homotopy equivalence) and contain no fat vertex (see the property \emph{Zero Strata} in \cite{Feighn:2011tt} or \cite{Lyman:2022wz}).
Hence there are only finitely many connecting paths (not just up to equivalence).

A splitting $\gamma=\gamma_1\cdot\ldots\cdot\gamma_p$ is a \emph{complete splitting} if every $\gamma_i$ is one of the following:

\begin{itemize}
 \item an edge in an EG or NEG stratum (possibly with vertex group elements on either end);
 \item an INP;
 \item an exceptional path;
 \item a connecting path.
 \end{itemize}
 The subpaths $\gamma_i$ are the \emph{terms} of the complete splitting (also called splitting units).

We say that $\gamma$ is \emph{completely split}
if it  has a complete splitting. This  splitting is  unique up to  replacing the subpaths $\gamma_i$ by equivalent paths \cite[Lemma~6.3]{Lyman:2022wz}. The terms of the splitting are thus well-defined up to equivalence.   Most paths considered in the proof will be completely split, and we will only consider turns between terms (not turns between two edges belonging to the same term).

The key property of a CT is the following: \emph{if $\gamma$ is an edge in an EG or   NEG stratum, or a connecting path in a zero stratum, then $f_\sharp(\gamma)$ is completely split.} Since the tightened image of an INP/exceptional path is an INP/exceptional path, this implies: if $\gamma=\gamma_1\cdot\ldots\cdot\gamma_p$  is a complete splitting, then $f_\sharp(\gamma)$ has a complete splitting which refines the splitting $f _\sharp(\gamma_1)\cdot\ldots\cdot f _\sharp(\gamma_p)$, see \cite[Lemma~6.1]{Lyman:2022wz}.

If $e$ is an edge in an EG stratum $H_r$, the terms of the complete splitting of $f_\sharp(e)$ are edges of $H_r$ or have height at most $r-1$; the first and last terms are edges in $H_r$.

If $\gamma$ is any path, there exists $k$ such that $f^k_\sharp(\gamma)$ is completely split  (Lemma~6.12 of \cite{Lyman:2022wz}, Lemma 4.25 of \cite{Feighn:2011tt}). We will need this   fact for circuits (complete splittings of circuits are defined in the obvious way, and the proof is the same).

This completes our review of properties of CT's. We will be able to use them thanks to the following existence result, which we deduce from \cite{Lyman:2022wz}.

\paragraph{CT's exist for toral relatively hyperbolic groups.}

We now suppose that $G$ is  toral relatively hyperbolic, and $G=G_1*\dots*G_q*\F_N$ is a    Grushko decomposition, with each $G_j$ one-ended.

\begin{theo} \label{ct}
 Let $G$ be an infinitely-ended toral relatively hyperbolic group. There exists $M$ such that, for any $\Phi\in\Out(G)$, there is a  CT $f:\Gamma\to \Gamma$  representing $\Phi^M$.
\end{theo}

\begin{proof}
In  \cite{Feighn:2011tt} Feighn and Handel prove the following two statements:
 any rotationless $\Phi \in \out{\F_N}$ is represented by a CT, and   there exists $M$ (depending only on $N$) such that any $\Phi^M$ is rotationless. 

For automorphisms of free products,  the definition of rotationless in  \cite{Lyman:2022wz} involves a new condition, which does not appear for free groups. 
In that context, Lyman proves that the first statement holds, see  Theorem A in \cite{Lyman:2022wz}. 
However, the second statement is not known in general because of the new condition in the definition of rotationless. 
   We describe this condition for the convenience of the reader, and   we explain how to deal with it for toral relatively hyperbolic groups. 
 
Consider a relative train track map $f:\Gamma\to \Gamma$ representing $\Phi$. It induces a ``derivative'' map $f'_{top}$ on the set of directions in $\Gamma_{top}$
 (a direction is a germ of oriented edge at a vertex): the image of the germ of $e$ is   the germ of the initial edge of $f_{top}(e)$. 
Replacing $\Phi$ by a power, we may assume  that any germ which is periodic under the action of $f'_{top}$ is in fact fixed 
(in the terminology of \cite{Lyman:2022wz}, almost periodic directions are almost fixed). 

Let $T$ be the Bass-Serre tree of the graph of groups $\Gamma$. There is a bijection between lifts $\tilde f$ of $f$ to $T$ and representatives $\varphi\in\Aut(G)$ of the outer automorphism $\Phi$: the lift associated to $\varphi$ satisfies  $\tilde f (gx)=\varphi(g)\tilde f(x)$  for every $g\in G$ and $x\in T$ (see Section 1 of \cite{Lyman:2022wz}).

Let $\Lambda$ be the set of oriented edges $\tilde e$ in the Bass-Serre tree $T$ lifting  an oriented edge $e$ of $\Gamma$ such that $f'_{top}(e)=e$.   If $\tilde e\in\Lambda$, there is a unique lift  $\tilde f$ of $f$ (depending on $\tilde e$) such that the initial edge of $\tilde f(\tilde e)$ is $\tilde e$.  We denote  by $\tilde f'_{top}$ the derivative of $\tilde f$, and by $\varphi_{\tilde e}\in\Aut(G)$ the automorphism associated to $\tilde f$ by the formula $\tilde f (gx)=\varphi_{\tilde e}(g)\tilde f(x)$.

The new requirement for   rotationless in \cite{Lyman:2022wz} is that, for any edge $\tilde e$ in $\Lambda$ and any germ $d$ at the   origin  of $\tilde e$, if $d$ is periodic under the lift $\tilde f'_{top}$ associated to $\tilde e$, then it is fixed by $\tilde f'_{top}$. 

 In order to prove that $\Phi$ has a rotationless power, we use
 Proposition~5.7 of \cite{Lyman:2022wz}, which gives a sufficient condition on $\varphi_{\tilde e}$ for the requirement to be satisfied\footnote{This proposition may be proved by an argument used on page 32 of \cite{Levitt:2008go} since the element $h$ constructed in Lyman's proof satisfies (in their notation) $Dg(x,e)=(hx,e)$.}.
  
  Recalling that  $G$ is torsion-free, it follows from Proposition 5.7 of \cite{Lyman:2022wz} that $\Phi$ has a rotationless power provided that the following finiteness condition holds for every $\varphi_{\tilde e}$: there exists  a bound for the period of elements of $G$ which are periodic under iteration of $\varphi_{\tilde e}$. If this bound only depends on $G$, some fixed  power of $\Phi$ is rotationless. The   following result thus implies the theorem.  
  \end{proof}

   \begin{theo} \label{shor}
  Let $G$ be a toral relatively hyperbolic group. 
 There exists $M$ such that, if $g\in G$ is periodic under iteration of some $\varphi\in\Aut(G)$, then its period is at most $M$.
 \end{theo}

\begin{proof}
This is  proved  in \cite[Corollary~10.3]{Levitt:2008go} for $G$ hyperbolic. Our proof is similar, using arguments due to Shor \cite{shorScottConjectureHyperbolic}.

   Theorem~1.8 of \cite{Guirardel:2014bx} provides a bound for the period (but it depends   on $\varphi$). This allows us to consider  the periodic subgroup $P\inc G$ of $\varphi$: it consists of  all the elements which are periodic under $\varphi$, and $\varphi_{ | P}$ has finite order $k$. 
 
 There are two cases. If $P$ is abelian, we get a uniform bound because the rank of $P$ is bounded and every  $GL(n,\Z)$ is virtually torsion-free.

If $P$  is non-abelian, Theorem~8.2 of  \cite{Guirardel:2015fi} says that it  is contained in a $\varphi^k$-invariant vertex group $G_v$ of an abelian    splitting of $G$, and the class of $\varphi^k_{ | G_v}$ in $\Out(G_v)$ has finite order. In fact  $\varphi_{ | G_v}^k$ itself has finite order because its fixed subgroup is not abelian, and therefore $P=G_v$. 

By \cite{Guirardel:2016do} there are only finitely many isomorphism types of vertex groups in abelian splittings of $G$, and the result follows     since $\Out(G_v)$ is virtually torsion-free (Corollary~4.5 of \cite{Guirardel:2015fi}), so   $k$ is bounded.
\end{proof}

\section{Growth in free products}\label{gfp}

As above, let $G=G_1*\dots*G_q*\F_N$ be a decomposition of $G$ as a free product. 
Assume that $\varphi\in\Aut(G)$ sends each $G_j$   to a conjugate and may be represented by a CT.

The main result of this section is a general combination theorem, which allows us to conclude that conjugacy classes of $G$ have PolExp growth under   iteration of  $\phi$ if, for each $j$, total PolExp growth holds  for the $j^{\rm th}$-component $\varphi_j\in\Aut(G_j)$.
More precisely, we show the following.

\begin{theo}\label{infb0}
Let $G$ be a finitely generated group, with a decomposition $G=G_1*\dots*G_q*\F_N$. 
Assume that $\varphi\in\Aut(G)$ sends each $G_j$ to a conjugate 
and is represented by a CT $f$.
If the $j^{\rm th}$-component $\phi_j\in\Aut(G_j)$ of $\phi$ has \total{} for all $j\in\{1,\dots,q\}$, then:
\begin{enumerate}
\item
For every $g\in G$, there exist $\lambda\geq 1$ and $d\in\N$ such that $\norm{\varphi^n(g)}\asymp n^d\lambda^n$.
\item
Moreover, $\lambda$ is an eigenvalue of the transition matrix of $f$  or appears in the spectrum or palangre spectrum of some $\phi_j$. 
The number $d$ is bounded by the sum of the number of strata of the CT and the maximal degree appearing in the spectra and palangre spectra of the $\phi_j$'s (see \autoref {spec}). 
\end{enumerate}
\end{theo}
 
In the case of toral relatively hyperbolic groups we get:

\begin{coro} \label{infb}
Let    
$G$ be a toral relatively hyperbolic group, with Grushko decomposition $G=G_1*\dots*G_q*\F_N$, and let $\varphi\in\Aut(G)$. 
\begin{enumerate}
    \item \label{enu: infb - growth}
    For every $g\in G$, there exist $\lambda\geq 1$ and $d\in\N$ such that $\norm{\varphi^n(g)}\asymp n^d\lambda^n$.
    \item \label{enu: infb - spectrum}
     There exist a non-negative integral matrix $A$, an integer $K$, and, for each $j$, an  automorphism $\psi_j\in\Aut(G_j)$ such that, for each $g$, 
      the number $\lambda^K$ 
      is an eigenvalue of $A$ or appears in the spectrum or palangre spectrum
     of $\psi_j$ for some $j\in\{1,\dots,q\}$.
     The degree $d$ is bounded independently of $g$. 
\end{enumerate}
In particular, the spectrum of $\phi$ is finite.
\end{coro}

We note that \autoref{infb} completes the proof of \autoref{main}.
Indeed the spectrum of $\varphi$ comes from the matrix $A$ and the freely indecomposable factors $G_j$, 
and the latter are controlled by \autoref{res: recap one-ended groups}.

\begin{rema}[Uniform bounds]\label{unifo} 
Uniform bounds (depending only on $G$) on the spectrum as in 
(\ref{enu: enplus - unbout}) of \autoref{main} would follow from a bound on the number of EG and NEG strata  in CT's (and of edges in EG strata to bound the algebraic degree of $\lambda$). Such bounds do not seem to exist in the literature at the time of writing, even for $G=\F_N$ (in this case bounds for $d$ and $ | \Lambda | $ are given in \cite{Levitt:2009hx}, using \Rt s). 

Nevertheless, when $G$ is a torsion-free hyperbolic group, the fact that $d$ and $|\Lambda|$ are uniformly bounded (in terms of $G$ only) can be deduced from our work combined with \cite[Proposition~A.11]{Fio}. 
More precisely, given any $\varphi\in\Aut(G)$, we can find a $\varphi$-invariant chain $\{1\}=\mathcal{F}_0\sqsubset\mathcal{F}_1\sqsubset\dots\sqsubset\mathcal{F}_k=\{G\}$ of properly nested free factor systems of $G$, such that the following holds: for every $i\in\{1,\dots,k\}$, and every $\varphi$-invariant free factor $F$ of $G$ whose conjugacy class belongs to $\mathcal{F}_i$, the restriction $\varphi_{|F}$ is fully irreducible relative to $\mathcal{F}_{i-1}$. 
There is a uniform bound on $k$, and in fact on the length of any chain of free factor systems. For every $i\in\{1,\dots,k\}$, the following holds:
\begin{itemize}
\item If $(\mathcal{F}_{i-1})_{|F}$ is a sporadic free factor system of $F$, then there is a $\varphi$-invariant free splitting of $F$ relative to $\mathcal{F}_{i-1}$. In this case the growth rates of $\varphi_{|F}$ are controlled from those in $\mathcal{F}_{i-1}$ by our \autoref{res: combination growth}, see \autoref{rem: metric decomposition free splitting}.  
\item If $(\mathcal{F}_{i-1})_{|F}$ is a non-sporadic free factor system of $F$, then the growth rates for $\varphi_{|F}$ are controlled by \cite[Proposition~A.11(3)]{Fio}. In this case $\varphi_{|F}$ has exponential growth, see for instance \cite[Lemma~2.9]{Andrew:2023ho}.
The docility assumption in \cite[Proposition~A.11]{Fio} is therefore satisfied provided we know that the maximal growth rate for elements under iteration of all automorphisms in the same outer class as $\varphi$ is the same. This is given by our \autoref{res: word growth fixed conj class} below.
\end{itemize}

The same argument might also work for toral relatively hyperbolic groups, however this would require extending the results in \autoref{words} to that setting.
\end{rema}

\begin{rema}\label{vtf}
\autoref{main} may be extended to groups which are virtually torsion-free. In this case there exist a $\varphi$-invariant, torsion-free, finite index subgroup $G_0$ and $k\geq1$ such that $g^k \in G_0$ for every $g \in G$. PolExp growth then holds in $G$ if it does in $G_0$, provided that infinite cyclic subgroups of $G$ are uniformly undistorted: there exists $C$ depending only on $G$ such that  \begin{equation*}
         \frac 1k \norm {g^k} \leq \norm g \leq C \norm{g^k} + C
    \end{equation*}
    for all $g\in G$.
This is true for $G$   hyperbolic (see for instance \cite[Chapitre~10, Proposition~6.4]{Coornaert:1990tj}),
 and quite probably for $G$ virtually toral relatively hyperbolic, but we were not able to find a reference in the literature.
\end{rema}

\medskip
The remainder of \autoref{gfp} is devoted to the proof of Theorems~\ref{infb0} and \ref{infb}.
We fix $\varphi$ and a CT $f$ as in \autoref{infb0}.
We will allow ourselves to replace $\varphi$ by a power when needed. This is legitimate by \autoref{puiss}.

\subsection{Length}
 
In order to compute word length, we fix any finite generating set for $G$.

\begin{defi}[Length of a circuit,  a path, a turn] \label{lgt}
The length of a circuit 
 $\gamma=e_1g_1\dots  e_pg_p$ in $\Gamma$ is $ | \gamma | =p+\sum_{i=1}^p | g_i | $: we count the number of edges and the length of the elements $g_i$ (which  belong to some $G_j$ if $t(e_i)$ is fat,  and are trivial otherwise).
 
 The length of a path
  $\gamma=g_0e_1g_1\dots  e_pg_p$  is $ | \gamma | =p+\sum_{i=1}^{p-1}| g_i | $. Note that we do not take $g_0$ and $g_p$ into account, so that equivalent paths have the same length.
  
The length of a   turn $\tau=e_ig_ie_{i+1}$ of $\gamma$, with $g_i$ in some $G_j$,  is $ | \tau | = | g_i | $ (if $t(e_i)=o(e_{i+1})$ is a non-fat vertex,  then $g_i$ is trivial and the length is 0). 
\end{defi}

\begin{rema}\label{long}
  The length of a circuit is defined so that the length of any (tight) circuit $\gamma$ is equivalent to the length $\norm g$ of the conjugacy class that it represents.
\end{rema}

 Now consider a completely split path or circuit $\gamma$. Any two consecutive terms $\mu, \mu'$ of its complete splitting   determine a turn $\tau$, which we   call a \emph {turn of $\gamma $} with \emph{adjacent terms} $\mu,\mu'$.   We will not consider turns at points of $\gamma $ which are not splitting points. A turn of $\gamma $ at a vertex $v$ is called a \emph{fat turn} if $v$ is fat. 

\begin{rema}\label{pasbord}
Length is defined so that the length of  $\gamma $ is the sum of the lengths of the terms of its complete splitting and the lengths of its (fat) turns. 
\end{rema}

 Applying $f$ and tightening maps a turn $\tau$ of $\gamma $ to a turn of $f\ti(\gamma)$, which we denote $f\ti(\tau)$. A fat turn is mapped to a fat turn at the same vertex.  
Replacing $\phi$ (hence $f$) by  a power, we may   assume that the image by $f_{top}^r$ of any given (non-fat) vertex $v$ is independent of $r\geq 1$. This implies that there are only two possibilities for a given non-fat turn $\tau$ of $\gamma$: either $f\ti(\tau)$ is fat, or no $f\ti^r(\tau)$ is fat.

To prove \autoref{infb0}, we must compute the growth of $\norm{\varphi^n(g)}$.    
Since the components of $\varphi$ have PolExp growth, we may assume that $g$ is not conjugate into one of the  subgroups $G_j$. The class $[g]$  is then represented by a unique non-trivial tight circuit $\gamma$ in $\Gamma$. 

By \autoref{long},  the growth of $\norm{\varphi^n(g)}$ is equivalent to that of the length of $f_\sharp^n(\gamma)$, computed as in \autoref{lgt}, so it suffices to show 
that, for any circuit $\gamma$, the length $ | f_\sharp^n(\gamma) | $ grows like some $n^d\lambda^n$.
Since some $f\ti^k(\gamma)$ is completely split, we may assume that $\gamma$ itself is completely split. Each image $\gamma_n=f\ti^n(\gamma)$ is completely split, and its complete splitting   refines the image of that of $\gamma_{n-1}$.

To prove \autoref{infb0}, we shall associate a growth type $(d,\lambda )$   to fat  turns and completely split paths (or circuits) in \autoref{agt}, and then prove in \autoref{cg} that this growth type   captures the growth of the images by $f\ti^n$: they grow like  $n^d\lambda^n$.

\subsection{The growth   of a turn}

We first compute the growth of $| f\ti^n(\tau) | $, for $\tau$ a fat turn of a completely split path or circuit -- recall that we only consider turns between terms of the complete splitting.

\begin{lemm}\label{gtour}
 Let $\delta$ be a completely split path.   If $\tau=e_0g_0e'_0$ is a turn of $\delta$ at a fat vertex $v_j$, then $| f\ti^n(\tau) | $ grows like $n^d \lambda^n$  for some $(d,\lambda)$ in the palangre spectrum of $\phi_j$.
\end{lemm}

\begin{proof}
The image of $\tau$ by $f\ti^k$ is a fat turn $f^k\ti(\tau)=e_kg_ke'_k$ of $f^k\ti(\delta)$ at the same vertex $v_j$. 
Let $\theta_k$ and $\theta'_k$ be the terms 
of the complete splitting of $f^k\ti(\delta)$ adjacent to $f\ti^k(\tau)$, normalized so that the group elements at the ends of $\theta_k$ and  $\theta'_k$ are trivial (recall that $\theta_k,\theta'_k$ are only defined up to equivalence).

For $k\geq 0$,  the element $g_{k+1}\in G_j$ is   equal to   $ m_k\phi_j(g_k)m'_k$, where $\phi_j \in \Aut(G_j)$ and $m_k,m'_k$ are  elements of $G_j$ which depend only on  
$\theta_k$ and $\theta'_k$. 
We claim that there is $k_0$, depending only on $f$, such that, for  $k\geq k_0$, the sequences $m_k$ and $m'_k$ are periodic with period at most $k_0$. 

To prove the claim, it suffices to  
 consider $m_k$. We distinguish several cases, depending on the  nature of the term $\theta_0$ with last edge    $e_0$.

If $\theta_0$ is an INP,  or an NEG edge $e_0$  such that $f(e_0)=u\cdot e_0m_0$ with $u$ a path of lower height,
we can take $k_0=1$ since $\theta_k$ is independent of $k$. We can also take $k_0=1$ if 
$\theta_0$  
is  exceptional, i.e.\ of the form $\theta_0=ew^pe_0$, as $\theta_k$ differs from $\theta_0$ only by the exponent of $w$.
If $\theta_0$ is a connecting path or an NEG edge $e_0$ with $f(e_0)=ge_0\cdot u$, we use induction on height since $\theta_1$ has lower height than $\theta_0$. 

Finally, if $\theta_0=e_0$ is an EG edge in a stratum $H_r$,  then the last edge of $f(e_0)$  belongs to  $H_r$, and $k_0$ depends on the permutation of the set of oriented edges of $H_r$ taking $e$ to the last edge of $f(e)$. This proves the claim. 

First assume $k_0=1$: we have $m_k=m_1$ and $m'_k=m'_1$ for $k\geq 1$.  
Given a fat turn $\tau =e_0g_0e'_0$ as above, we  can now compute the  growth  of  $ | g_n | = | f\ti^n(\tau) | $.  
We have
\begin{align*}
g_{n+1} & =m_1\phi_j(m_{1})\dots\textbf{} \phi_j^{n-1}(m_1)\phi_j^n(g_1) \phi_j^{n-1}(m'_1)\dots \phi_j(m'_{1})m'_1\\
& =L_n(\varphi_j,m_1) \varphi^n_j(g_1)R_n(\varphi_j,m'_1).
\end{align*}
\autoref{rema:bete}  and  
\total{} of $\varphi_j$  
say that $| f\ti^n(\tau) | $ grows (in $G_j$, hence in $G$ by quasiconvexity of $G_j$) like some $n^d\lambda^n$ coming from the palangre spectrum of $\phi_j$.

If $k_0>1$, we apply the previous argument to $f^{k_0}$. As in   \autoref{puiss}, the lemma  is true for $f$ because it is true for $f^{k_0}$; indeed $g_{k+1}= m_k\phi_j(g_k)m'_k$ with $m_k,m'_k$ taking only finitely many values, so (up to equivalence)  the growth of $g_{i+k_0n}$ does not depend on the residue of $i$ mod $k_0$.
\end{proof}

\subsection{Assigning  growth types  }\label{agt}

If $\tau$ is a fat turn, we have just seen that $| f\ti^n(\tau) | $ grows like some $n^d\lambda^n$. We define the growth type $c(\tau)$ of $\tau$ as $(d,\lambda)$. Note that $c(f\ti(\tau))=c(\tau)$.
We shall now assign a growth type $c(\delta)$ to any completely split path $\delta$; it will depend only on the equivalence class of $\delta$. In \autoref{cg} we will prove that $c(\delta)$ captures the growth of $f\ti^n(\delta)$.

We first need to understand  how   fat turns $\tau$ between terms of the paths $f^k\ti(\delta)$   appear. 
Some of them are just the image of 
a fat turn of $f^{k-1}\ti(\delta)$. The others  appear in two ways. 

First, they may be    in the interior of the image of an edge or a connecting path contained in the complete splitting of $f^{k-1}\ti(\delta)$ (INP's and exceptional paths do not create turns, as their image is a single term). The other possibility is that $\tau$ is  the  image of a turn of $f^{k-1}\ti(\delta)$ at a non-fat vertex $v$. 
This $v$ cannot be the image of a splitting point $w$ of $f^{k-2}\ti(\delta)$ because we have assumed $f_{top}^2(w)=f_{top}(w)$ for any vertex $w$, so $v$ lies in the interior of the image of an edge or a connecting path. 
Thus fat turns are created only in the image of an edge or connecting path by $f\ti $ or $f^2\ti $.

Since $\Gamma$ only contains finitely many edges and connecting paths,  
and $ c(f\ti(\tau))=c(\tau)$, we deduce from the previous discussion that, given $\delta$,  only finitely many growth types $c(\tau)$ are associated to fat turns of the paths $f^k\ti(\delta)$. 

The definition of $c(\delta)$ is by induction on height.  
Recall that the length of a path $g_0e_1g_1\dots  e_pg_p$ does not take $g_0$ and $g_p$ into account (see \autoref{pasbord}). We first treat the case where $\delta$ is reduced to a single term.  
INP's do not grow,  
and exceptional  paths grow linearly,  so we define $c(\delta)$ as $(0,1)$ or $(1,1)$ respectively.  
The growth type of a connecting path is defined as that of its tightened image (which has lower height).
Now suppose that $\delta$ is an edge in an EG or NEG stratum $H_r$. 

We associate to $H_r$ a finite set $C$ of growth types as follows.
Consider all edges $e$ of $H_r$, and all terms of lower height in the complete splitting of $f(e)$. These lower terms have growth types by induction, and we include them in $C$. We also consider the fat turns of  the paths $f\ti(e)$, as well as those of $f^2\ti(e)$ (created as images of non-fat turns of $f\ti(e)$ as explained above), and we   include their growth types in $C$. 

We let $(d_C,\lambda_C)$  be the maximal growth type in $C$ (for the obvious order, see \autoref{gtype}). We compare it with $(0,\lambda_r)$, 
with $\lambda_r$ the Perron-Frobenius eigenvalue associated to the stratum $H_r$ (it is larger than $1$ if and only if $H_r$ is EG). 
For $e$ any edge in $H_r$, we define $c(e)$ as the maximum of these two growth types, with one exception: if $\lambda_r=\lambda_C$, we define   $c(e)=(d_C+1,\lambda_C)$.
Note that $c(e)$ only depends on the stratum $H_r$ containing $e$.

This definition is motivated by 
the following standard fact. 

\begin{lemm} \label{est} 
  For $\lambda_1,\lambda_2\geq 1$ and  $d\geq 0$ an integer,
  \begin{equation*}
    \sum _{k=1}^n (n-k)^d\lambda_1^{k} \lambda_2^{n-k} \asymp 
    \left\{
        \begin{split}
            \lambda_1^n,\quad  \text{if}\ \lambda_1>\lambda_2 \\
            n^d\lambda_2^n,\quad  \text{if}\ \lambda_2>\lambda_1 \\
            n^{d+1}\lambda_2^n,\quad  \text{if}\ \lambda_2=\lambda_1. 
        \end{split}
    \right.
  \end{equation*}
  \qed
\end{lemm}

We have now defined $c(\delta)$ for $\delta$ a term. If $\delta$ is a completely split path or circuit, 
 let $\delta=\delta_0\cdot\ldots \cdot\delta_p$ be its complete splitting. To motivate the definition of $c(\delta)$, note that   
 \begin{equation*}
     |f^k\ti(\delta) | =\sum_{i=0}^p| f  ^k\ti(\delta_i) | +\sum_{i=0}^{p-1} | f\ti^k(\tau_i) | ,
 \end{equation*}
with $\tau_i$ the turn between $\delta_i$ and $\delta_{i+1}$  (we include the turn between $\delta_p$ and $\delta_0$ if $\delta$ is a circuit). 
Recall that, for each $i \in \{0, \dots, p-1\}$, either $\tau_i$ is fat, or $\tau_i$ is not fat but $f\ti(\tau_i)$ is, or no $f\ti^k(\tau_i)$ is fat. 
This leads us to define $c(\delta)$ as the maximal growth type among those of the subpaths $\delta_i$, those of    the fat turns of $\delta$, and those of  the  fat turns of $f\ti(\delta)$ which are images of non-fat turns of $\delta$.

\begin{rema}\label{rema: growth-types}
Note that every $\lambda$ featured in growth types appears in the growth of a palangre involved in the growth of a fat turn (\autoref {gtour}), or is the eigenvalue $\lambda_r$  associated to an EG stratum.
\end{rema}

\subsection{Computing growth}\label{cg}

If $\tau$ is a fat turn, we have seen that $ | f^n\ti(\tau) | $ grows like some $n^d\lambda^n$ (\autoref {gtour}), and we have defined $c(\tau)=(d,\lambda)$. We now show:

\begin{lemm}\label{gr}
Let $\delta$ be   a completely split path or circuit, with $c(\delta)=(d,\lambda)$. Then 
 the length $ | f^n\ti(\delta) | $ grows like $n^d\lambda^n$.
\end{lemm}
\begin{proof}

The proof is by induction on height.  
Note that we have defined $c$
in the preceding section
in such a way that  the lemma is true for 
  any $\delta$ if it is true for its terms, so we consider terms. The only non-trivial case is when $\delta$ is an edge $e$ in an EG or NEG stratum $H_r$. Recall that the image of any edge of height $r$ has a complete splitting whose terms have height at most $r-1$ or are edges of height $r$.

We first show that $ | f^n\ti(e) |\succcurlyeq n^d\lambda^n $ if $c(e)=(d,\lambda)$.
Noting that  any edge $e'$ in $H_r$ appears as a term in some $f^k\ti(e)$, the result is clear from the definition of $c(e)$, except if $\lambda_r=\lambda_C$ because then $c(e)=(d_C+1,\lambda_C)$. 

We thus assume $\lambda_r=\lambda_C$. 
There is an edge $e'$ of $H_r$ such that the complete splitting of $f(e')$ or $f^2\ti(e')$ contains 
a turn or a term of height smaller than $r$, say $\mu$, with growth type $c(\mu)=(d_C,\lambda_C)$.

Each $ f^k\ti(e) $ contains $\lambda_r^k$ copies of $e'$, hence at least $\lambda_r^k$ copies of $\mu$  
(here and below we neglect multiplicative constants). For $n>k$, the image in $  f^n\ti(e)  $ of each copy has length $(n-k)^{d_C}\lambda_C^{n-k}$ (this uses the induction hypothesis if $\mu$ is a term). 
Thus $ | f^n\ti(e) | $ is at least $\sum _{k=1}^n(n-k)^{d_C}\lambda_C^{n-k}\lambda_r^k$, which grows like $n^{d_C+1}\lambda_C^{n}$ when $\lambda_r=\lambda_C$ by \autoref{est}.

 We now show the upper bound: $ | f^n\ti(e) |\preccurlyeq n^d\lambda^n $. We must make sure that all growth types contributing to the growth of $ | f^n\ti(e) | $ are accounted for in the definition of $c(e)$.

 For $k\geq 1$, we define a \emph{$k$-ancestor}  $\rho$ as  a subpath of $f^{k}\ti(e)$ which is 
  a maximal subpath of height less than $r$ in $f(e')$, for $e'$ an edge of height $r$ in the complete splitting of $f^{k-1}\ti(e)$. Ancestors have bounded length, and up to multiplicative constants the number of $k$-ancestors is $\lambda_r^k$ (unless all edges in the paths $f^{k}\ti(e)$ have height $r$, a trivial case). 
 
 We claim that the growth type $c(\rho)$ of any $k$-ancestor $\rho$ belongs to the set $C$ used to define $c(e)$. Indeed, $c(\rho)$ was defined using the growth type of the terms of its complete splitting, of its fat turns, and of the fat turns of $f\ti(\rho)$. All of these appear in $C$ (recall that  growth types of fat turns of $f^2\ti(e')$ are in $C$).

 Using the induction hypothesis,   we deduce that, if $\rho$ is  a $k$-ancestor and $n>k$, then $f^{n-k}\ti(\rho)$  is a subpath  of length at most $(n-k)^{d_C}\lambda_C^{n-k}$ of $f^n\ti(e)$, which we call a descendant of $\rho$.

 The path $f^n\ti(e)$, whose length we want to bound, has a splitting    into edges of height $r$  and descendants $f^{n-k}\ti(\rho)$, with $1\leq k\leq n$. We call  it the coarse splitting of $f^n\ti(e)$ because it is coarser than the complete splitting.
 
 To bound the length of  $f^n\ti(e)$, we estimate separately the total length of the terms of the coarse  splitting and the total length of the fat turns between these terms.

The total length of the descendants contained in $f^n\ti(e)$ is bounded by $ \sum_{k=1}^n \lambda_r^k (n-k)^{d_C}\lambda_C^{n-k}$, and there are $\lambda_r^n$ edges of height $r$.  Both numbers are bounded by $n^d\lambda^n$ by \autoref{est} and the definition of $c(e)$.

The  argument to bound the total length of all fat turns between terms of the coarse splitting of $f^n\ti(e)$ is similar. 
We now define a   $k$-ancestor %is now 
as a fat turn $\tau$ between two terms of the complete splitting of $f(e')$,  
for $e'$ an edge of height $r$ in the complete splitting of $f^{k-1}\ti(e)$, or a fat turn of $f^{2}\ti(e')$ which is the image of a non-fat turn of $f(e')$, for $e'$ of height $r$ in $f^{k-2}\ti(e)$. 

We claim that any fat turn $\tau$ between terms of the coarse splitting of $  f\ti^n(e)  $ has an ancestor (i.e.\ there exists a $k$-ancestor  $\tilde\tau$ such that  $\tau=f\ti^{n-k}(\tilde\tau) $). Indeed, the vertex $v$ carrying $\tau$ belongs to the image of an edge $e'$ of height $r$ in $f\ti^{n-1}(e)$. The turn $\tau$ is an $n$-ancestor if $v$ is the image of an interior point of $e'$, or is the image of a non-fat endpoint of $e'$ (which must be  in   the interior of $f(e'')$ for some edge $e''$ of height $r$ in $f\ti^{n-2}(e)$). Otherwise  $v$ is the image of a fat endpoint of $e'$ and we use induction on $n$.
This proves the claim.

It is again true that the number of $k$-ancestors is $\lambda_r^k$, and we know that the growth type $c(\tau)$ of a fat turn $\tau$ computes  $ | f\ti^n(\tau) | $.
We conclude by checking that
  $C$ was defined so as to  contain all growth types of ancestors.
\end{proof} 

\subsection {End of the proof of PolExp growth}

 We can now prove   \autoref{infb0} and 
\autoref{infb}, which imply  \autoref{main} as pointed out before.

The arguments of the previous sections 
   prove  
  \autoref{infb0} for some power   $\phi^p$ (we had to replace $f$ by a power to ensure that the image by $f_{top}^r$ of any   vertex   is independent of $r\geq 1$).
  
  Indeed, given $g$, PolExp growth follows from the assumptions if a conjugate of $g$ is contained in some $G_j$. Otherwise, we can represent  some $[\varphi^k(g)]$ by a completely split circuit $\gamma$, and   \autoref{gr} says that $ | f^n\ti(\gamma) | $ grows like some $n^d\lambda^n$.  
  The second assertion of \autoref{infb0} follows from \autoref{rema: growth-types}, with the bound on $d$ coming from the way growth types were defined in \autoref{agt}.

   By \autoref{puiss}, the theorem is true also for $\phi$ itself because the incidence matrix of $f^p$ is $A^p$, and $(d,\lambda)$ appears in the palangre spectrum of $\varphi_j$ if $(d,\lambda^p)$ does in that of $\varphi_j^p$. 
  
To prove \autoref{infb} we recall that,
    if $G$ is toral relatively hyperbolic, Theorems \ref{res: recap one-ended groups} and \ref{ct} ensure that \autoref{infb0} applies to a power of $\varphi$. 
 We conclude using \autoref{puiss} as before.

\part{Further results}

 In this part we assume that $G$ is a torsion-free hyperbolic group.

\section{Growth of elements}\label{words}

 We now consider growth of elements rather than conjugacy classes.
 Recall (\autoref{sga} and \autoref{main}) that every conjugacy class grows like some $n^d\lambda^n$. The set of growth types   of conjugacy classes that occur for a given $\varphi\in\Aut(G)$ is the spectrum of $\varphi$, denoted by $\Lambda$. 
 It is finite.

\begin{prop}
\label{res: word growth fixed conj class}
 Let $G$ be a torsion-free  hyperbolic group.
    Let $\phi \in \aut G$.

\begin{enumerate}
 \item

 For every element $g \in G$, the length $\abs{\phi^n(g)}$ grows like some $n^d\lambda^n$, with $d\in\N$ and $\lambda\geq 1$.
    \item  Let $(d_M, \lambda_M) =\max\Lambda$ be the maximal growth type of conjugacy classes under   iteration  of $\phi$.
    The growth type $(d,\lambda)$ of any $g\in G$ is bounded above by $(d_M, \lambda_M)$ if $\lambda_M > 1$, by $(d_M+1, 1)$ if $\lambda_M = 1$.
\item More generally, the  growth type $(d,\lambda)$ of any  $g\in G$   belongs to $\Lambda^+ = \Lambda \cup \set{ (d+1, 1)}{(d,1) \in \Lambda}$. 
 
    \end{enumerate}
\end{prop}

\begin{exam}\label{bgex}
  To illustrate 2,  consider 
the automorphism $\varphi:a\mapsto bab^{-1}, b\mapsto b^2ab^{-1}$ of $\F_2$ (representing a Dehn twist on  a punctured torus).   As noticed by Bridson-Groves \cite{Bridson:2010qu}, conjugacy classes grow at most linearly under $\varphi$, but the word  $b$ grows quadratically. Similar examples may be constructed in one-ended groups.
 We will see (\autoref{bg}) that, for any $\varphi$ having  elements growing faster than all conjugacy classes, all non-trivial $g\in G$ have the same growth type, unless some power of $\varphi$ is inner. 
 \end{exam}

 The proposition will be proved in this section, but the proof of the third assertion  uses  a quasiconvexity result (included in \autoref{slqc}) which will be proved in \autoref{hier}. Of course, this last assertion will not be used in \autoref{hier}.

 As mentioned in \autoref{elts},  PolExp growth of conjugacy classes implies that of elements, so the first assertion holds.
There are only finitely many growth types $(d,\lambda)$ as $g$ varies.

The proof of the  second assertion relies on the following lemma, which generalizes Lemma 2.3 of \cite{Levitt:2009hx} and will be used also to prove  malnormality in \autoref{theo: prelim filtration}.

\begin{lemm}\label{pareil}
Let $G$ be a torsion-free hyperbolic group.   
 If $\varphi\in\Aut(G)$ has infinite order in $\Out(G)$  and fixes some non-trivial $h\in G$, then elements cannot grow faster than conjugacy classes:
 $\abs{\phi^n(g)}\preccurlyeq n^{d_M} \lambda_M^n$ for every $g\in G$. 
 
\end{lemm}

\begin{proof}

We sketch the argument, following the proof of Lemma 2.3(2) in \cite{Levitt:2009hx}.

     We identify  $G$ with the set of vertices of a  Cayley graph, a $\delta$-hyperbolic space (rather than a tree in \cite{Levitt:2009hx}). We denote by $\sf d$ the distance function,  by $\gro x y {}$ the Gromov product based at the identity vertex $e$. 
         In this proof we write $A_n\lesssim B_n$ to  mean that $A_n- B_n $ has an upper bound independent on $n$ (the bounds will depend only on $\delta$ and  $\abs h$).
     
        Fixing $g\in G$, we first write 
        $$\dist e{\phi^n(g)}\lesssim \norm {\phi^n(g)}+2\gro { \phi^n(g\m)}{\phi^n(g)}{},$$ using a standard formula in hyperbolic spaces. Now fix $h\in G$ with $\varphi(h)=h$ and $\norm h$ large compared to $\delta$.
         We have 
        \begin{align*} 
        \dist e{\phi^n(g)} &=  \dist h{h\phi^n(g)} \\ 
         &\lesssim \dist e{h\phi^n(g)} \\ 
         &\lesssim \norm {h\phi^n(g)}+2\gro { \phi^n(g\m)h\m}{h\phi^n(g)}{} \\
         &\lesssim \norm { \phi^n(hg)}+2\gro { \phi^n(g\m) }{h\phi^n(g)}{}.
        \end{align*}

  Using the hyperbolicity formula $\min(   \gro x y {},\gro yz {})\leq \gro x z {}+\delta$, we conclude 
   $$
\abs{\phi^n(g)}=\dist e{\phi^n(g)}\lesssim\max(\norm { \phi^n(g)},\norm { \phi^n(hg)})+2 \gro { \phi^n(g)}{h\phi^n(g)}{} .
   $$  
   To prove the lemma, it now suffices to show that $ \gro { \phi^n(g)}{h\phi^n(g)}{}$ grows at most linearly. We assume  that it does not, and we argue towards a contradiction.

Let $A_h$ be an axis for $h$: an $h$-invariant quasigeodesic joining $h^{-\infty}$ to $h^{+\infty}$, with $h^{\pm\infty}=\lim_{n\to\pm\infty}h^n$. Let $p_n$ be a projection of $\phi^n(g)$ onto $A_h$, as in \autoref{periph}. If 
$ \gro { \phi^n(g)}{h\phi^n(g)}{}$ grows faster than linearly, so does $\abs{p_n}$:
indeed, since $\norm h$ has been chosen large compared to $\delta$, the distance from $p_n$ to any geodesic between $\phi^n(g)$ and $h\phi^n(g)$ is bounded independently of $n$. 

Since $\phi$  fixes $h$, points of $A_h$ are moved a bounded amount by $\phi$, so $d(p_n, \phi(p_n))$ is bounded.
We get a contradiction because $d( \phi(p_n), p_{n+1})$ is also bounded: indeed, $p_n$ may be viewed as a quasi-center of a triangle with vertices $e, \phi^n(g)$, and $h^{\pm\infty}$ (depending on whether $p_n$ goes to $h^{+\infty}$ or $h^{-\infty}$); fixing $e$ and $h^{\pm\infty}$, and being a  quasisometry, $\phi$ sends $p_n$ close to $p_{n+1}$.
\end{proof}

\begin{proof}[Proof of the second assertion of \autoref{res: word growth fixed conj class}]

We will use repeatedly  \autoref{est} in the following form:
$\sum _{k=1}^n k^d \lambda^{k}  $ grows like $n^d\lambda^{n}$ if $\lambda >1$, like $n^{d+1}$ if $\lambda =1$.

\autoref{pareil} implies the proposition when
 there is a non-trivial $\varphi$-fixed conjugacy class: write $\varphi= {\rm ad}_a\psi$ %\gc\cc{notation Ã  harmoniser : j'ai vu $ad_a$ en partie 2} 
 with ${\rm ad}_a(x)=axa^{-1}$   and $\psi(h)=h$, and combine the formula 
 \begin{equation}
 \label{res: changing representation vs}
     \varphi^n(g) =  a\psi(a)\psi^2(a)\dots \psi^{n-1}(a) \psi^n(g)   \psi^{n-1}(a\m)\dots \psi^2(a\m)\psi(a\m)a\m
 \end{equation}
 with \autoref{est} to bound $\abs{\phi^n(g)}$.   By \autoref{puiss}, this applies also when there is a periodic conjugacy class.
  
If $G$ is one-ended and there is no non-trivial $\varphi$-periodic conjugacy class, the refined JSJ decomposition of \autoref{jsjdec} must be trivial, and $\varphi$ must be induced by a pseudo-Anosov homeomorphism of a closed surface.  In this case the theorem follows from \autoref{rema:bete} and \autoref{res: growth surface}.

The only remaining case is when $G$ has a  Grushko decomposition    $G= G_1 \ast \dots \ast G_q \ast \mathbf F_N$, where each $G_j$ is a surface group as above. We may assume that $\varphi$ is represented by a CT $f$ as in \autoref{cts}. 
By \autoref{res: growth surface}, palangres in $G_j$ grow at most like $\lambda_j^n$, with $\lambda_j$ the dilation factor of the associated pseudo-Anosov homeomorphism.
Hence, according to \autoref{gtour},  
each non-trivial turn at a fat vertex $v_j$ grows at most like $\lambda_j^n$ as well.

By the analysis of \autoref{gfp}, the maximal growth rate $(d_M, \lambda_M)$  of both conjugacy classes and terms of complete splittings under iteration of $\phi$ and $f$ is the supremum of the following: the growth rates of edges in EG strata of $f$, and the $\lambda_j^n$'s. We show that elements do not grow faster.

 Identifying $G$ with the fundamental group of the graph of groups $\Gamma$ and choosing a suitable $f$-fixed basepoint, $f$ induces an automorphism $\psi\in\Aut(G)$ in the same outer class as $\phi$. The growth of elements under $\psi$ is that of closed paths (which may be assumed to be completely split) under $f$, and is not faster than 
 $(d_M,\lambda_M)$. The same holds for $\phi$, using the same formula (\ref{res: changing representation vs}) as above.
\end{proof}
 
\begin{proof}[Proof of the third assertion of \autoref{res: word growth fixed conj class}]
 We know that there are only finitely many growth types of elements. Let $(d_M^{el}, \lambda_M^{el})$ be the largest one. One clearly   has $(d_M^{el}, \lambda_M^{el})\geq (d_M,\lambda_M)$, and the second assertion implies  $(d_M^{el}, \lambda_M^{el})\in\Lambda^+$.
 
 Let $G_{sl}$ be the $\varphi$-invariant subgroup of $G$  consisting of all ``slow'' elements: those whose growth type is less than $(d_M^{el}, \lambda_M^{el})$. As will be proved in \autoref{slqc},  it is quasiconvex, hence hyperbolic. In particular, for $g\in G_{sl}$, the growth of  $|\varphi^n(g)|$ and $\norm{\varphi^n(g)}$     may be computed indifferently in $G_{sl}$ or in $G$ (see \autoref{rem: growth in qc sbgp}), and the (conjugacy class) spectrum of the restriction $\varphi_{ | G_{sl}}$ is contained in $\Lambda$.
 
 We can now apply the same argument as above to $\varphi_{ | G_{sl}}$, and iterate. This process terminates because there are  fewer growth types of elements for $\varphi_{ | G_{sl}}$ than for $\varphi$, thus controlling all growth types of elements and completing the proof. 
\end{proof}

\section{A growth hierarchy} \label{hier}

In this section we combine  \autoref{main} with a construction due to Paulin \cite{Paulin:1997ba} to generalize the polynomial subgroups introduced in \cite{Levitt:2009hx, Dahmani:2023re} (see \autoref{raff}). We note that the arguments below rely on the existence and finiteness of growth types, and therefore cannot be used to give an alternative proof of our main theorem.
 
\begin{theo}\label{theo: prelim filtration}
Let $G$ be a torsion-free hyperbolic group, and let $\varphi\in\Aut(G)$  have infinite order in  $\Out(G)$.
Let $(d_M,\lambda_M)$ be the maximal growth type of conjugacy classes under iteration of $\phi$.
There exists a unique $\phi$-invariant set $\{[H_1],\dots, [H_k]\}$ of conjugacy classes of proper (possibly trivial) quasiconvex  subgroups $H_i$ of $G$ such that, for every $g\in G$:
\begin{enumerate}
    \item if $g$ is not conjugate into any of the subgroups $H_i$, then $\norm{\varphi^n(g)}\asymp  n^{d_M} \lambda_M^n$;
    \item if $g$ is conjugate into one of the subgroups $H_i$, then $\norm{\varphi^n(g)}$ grows strictly slower than $n^{d_M}\lambda_M^n$;
    \item  $H_i$ is not   contained in a conjugate of $H_j$ if $i\ne j$.
\end{enumerate}
Moreover, if  the maximal growth type $(d_M,\lambda_M)$ is at least quadratic, then $H_1,\dots,H_k$ is a malnormal family in $G$ (i.e.\ if $gH_ig^{-1}\cap H_j\ne\{1\}$, then $i=j$ and $g\in H_i$).

\end{theo}

 \begin{rema}\label{rk:question-rh}
We suspect that a similar statement should hold,    more generally, for toral relatively hyperbolic groups. However, our proof involves a limiting $\mathbb{R}$-tree obtained by iterating $\varphi$. In the toral relatively hyperbolic case, Paulin's construction would yield a limiting tree-graded space with CAT(0) pieces \emph{à la} Dru\c tu-Sapir \cite{Drutu:2005ab}. Adapting our arguments to such a space is beyond the scope of the present work. 
\end{rema}

\begin{rema}
 When $\varphi$ has polynomial growth, one may show that the malnormal family $H_1,\dots,H_k$ is a free factor system  (Y.\ Guerch, private communication).
\end{rema}

\begin{proof} 
We follow \cite{Paulin:1997ba} for 1 and 2. The only real novelty in our proof will be  the malnormality. 

Taking an ultralimit of $G$, equipped with the word metric divided by a suitable renormalizing factor $a_n$ and the natural  action of $G$ twisted by $\phi^n$, Paulin constructs an $\mathbb{R}$-tree $T$ equipped with a non-trivial isometric  very small action of $G$ extending to an action of $G\rtimes_\varphi \Z$ by bilipschitz homeomorphisms 
(he also constructs another tree on which $G\rtimes_\varphi \Z$ acts affinely, but his Remark 3.6 does not apply to it).

  Recall that the action is very small if arc stabilizers are cyclic, tripod stabilizers are trivial, and the fixed point set of $g^n$ is the same as that of $g$ for $n\geq 2$. By \cite{Guirardel:NA} point stabilizers are quasiconvex, and there are only finitely many orbits of branch points and branching directions.

According to \cite[Remarque~3.6]{Paulin:1997ba},
\begin{equation*}
    \lim_\omega \frac 1{a_n} \norm{\phi^n(g)} = \norm[T]g, \quad \forall g \in G,
\end{equation*}
where $\omega$ is the non-principal ultrafilter used to build $T$ and $\norm[T]g$ is the translation length of $g$ in $T$.
In particular, for the maximal growth type $(d_M,\lambda_M)$, we have  
\begin{equation*}
   0< \lim_\omega \frac {n^{d_M}\lambda_M^n}{a_n} <\infty.
\end{equation*}
In other words, $(a_n)$ captures the maximal growth type.
Hence, for any $g \in G$, the sequence $\norm{\phi^n(g)}$ has maximal growth   if and only if $g$ acts hyperbolically on $T$.
 
Let $H_1,\dots,H_k$ be representatives for conjugacy classes of maximal elliptic subgroups; there are finitely many of them 
because each $H_i$ fixes a branch point (there is no inversion because $T$ is very small), and  there are finitely many orbits of branch points by \cite{Guirardel:NA}.

The groups $H_i$
are quasiconvex by \cite{Guirardel:NA}, and self-normalizing   because $T$ is very small. Since $G\rtimes_\varphi \Z$ acts on $T$ by homeomorphisms, the set $[H_1],\dots, [H_k]$ is $\phi$-invariant. An element $g\in G$ is elliptic if and only if it is contained in a conjugate of some $H_i$, and 1, 2, 3 are proved. Uniqueness holds because any subgroup consisting of elements whose conjugacy class has growth type smaller than $(d_M,\lambda_M)$ is elliptic in $T$.

To show malnormality assuming supralinear growth, we choose for each $i$ a branch point $p_i$ fixed by $H_i$. We   suppose that there are two distinct branch points $v,w$, each in the orbit of some $p_i$, whose stabilizers $G_v,G_w$ intersect non-trivially, and we argue towards a contradiction. 

Because $T$ is very small, $G_v  \cap G_w$ is a maximal 
cyclic subgroup $Z$. If one of $G_v,G_w$ is equal to $Z$, the way we defined the groups $H_i$ and the points  $p_i$ implies that the other group also equals $Z$, and $v,w$ are in the same orbit. This is impossible because  $Z$ is self-normalizing. 

Thus none of $G_v,G_w$ is cyclic. Our next goal   is to show that there is  an automorphism $\psi$ in the outer class of a power $\phi^p$ leaving both $G_v$ and $G_w$ invariant (in the terminology of \cite{Dahmani:2022re},
 $G_v$ and $G_w$ are \emph{twin subgroups}). This will lead to a contradiction.

First note that  the arc $\geo vw$ only contains finitely many branch points: otherwise, the  finiteness of orbits of branching directions \cite{Guirardel:NA} yields a hyperbolic element normalizing $Z$, a contradiction.

Again using finiteness of orbits of branching directions, we   find   $\psi$ in the outer class of some $\phi^p$ whose action on $T$ fixes the initial edge $e$ of $\geo vw$. Since stabilizers of tripods are trivial and $\psi$ leaves $G_e$ (which is equal to $Z $)  invariant,   the whole arc is fixed. In particular, $\psi$ leaves both   $G_v$ and $G_w$ invariant.  
Up to replacing $\psi$ by $\psi^2$, we can also assume that $\psi$ is the identity on $Z$.

We claim that there exists $g_v \in G_v$, fixing a single point of $T$,  whose conjugacy class achieves the maximal growth type in $G_v$ under $\psi$
(this growth type is the same in $G_v$ or in $G$, see \autoref{rem: growth in qc sbgp}). Indeed, if some conjugacy class in $G_v$ is not $\psi$-periodic, we can choose $g_v$ in  any conjugacy class in $G_v$ with maximal growth: it cannot fix any edge because it would be $\psi$-periodic. If all conjugacy classes in $G_v$ are $\psi$-periodic, then the existence of $g_v$ follows from the fact that $G_v$ is non-abelian, while there are only finitely many conjugacy classes of incident edge stabilizers. 

We choose $g_w\in G_w$ similarly. 
We  now
consider the product $g_vg_w$. It acts hyperbolically on $T$, so by 1 and 2 of the theorem its conjugacy class grows strictly  faster than those of $g_v$ and $g_w$, and supralinearly by assumption.

If $\psi$ has finite order in both $\Out(G_v)$ and $\Out(G_w)$, then  $g_v,g_w,g_vg_w$ grow  at most linearly (as elements), contradicting supralinear growth   of $g_vg_w$. Otherwise, since $\psi$   fixes $Z$,   applying \autoref{pareil} to the restrictions of $\psi$ to $G_v$ and/or $G_w$ shows that the element $g_v$ (\resp $g_w$) grows linearly or has the same growth as its conjugacy class, which grows strictly slower  than   the class of $g_vg_w$. This contradiction completes the proof.
\end{proof}

\begin{coro}\label{theo:filtration}
Let $G$ be a torsion-free hyperbolic group, and let $\varphi\in\Aut(G)$. There exist  
  a finite rooted tree $\tau$ with root $v_0$  and,
 for every vertex $v$ of $\tau$, a (possibly trivial) quasiconvex subgroup $G_v\subseteq G$ and a growth type $(d_v, \lambda_v)$ with the following properties: 
\begin{enumerate}
\item $G_{v_0}=G$ and $(d_{v_0}, \lambda_{v_0})=(d_M,\lambda_M)$ is the maximal growth type of conjugacy classes under $\varphi$;
 \item if $w$ is a descendant of $v$, then $G_w\subsetneq G_v$ and $(d_w, \lambda_w)<(d_v,\lambda_v)$;
\item the conjugacy class of each $G_v$ is $\varphi$-periodic;
\item for every $g\in G$ which is conjugate into $G_v$ but not into $G_w$ for any child  $w$ of $v$, one has %then
$\norm{\varphi^n(g)}\asymp n^{d_v} \lambda_v^n$; 
\item if the growth in $G_v$ is at least quadratic and $w$ is a child of $v$, then $G_w$ is malnormal.
\end{enumerate}
\end{coro}

\begin{proof}%[Proof of~\autoref{theo:filtration}]
We start the construction with  $G_{v_0}$ and $(d_{v_0}, \lambda_{v_0})$ as in 1. If $\phi$ has finite order in $\Out(G)$, we stop there. Otherwise, the  children of $v_0$ carry  the groups $H_i$ provided by \autoref{theo: prelim filtration}. They are quasiconvex, hence hyperbolic, and we can iterate, using for each $i$ the automorphism $\psi_i$ of $H_i$ induced by a suitable representative of a power of $\phi$. This allows us to construct inductively a locally finite tree $\tau$.

The process stops after finitely many steps because  there are only finitely many different growth types  in $G$ under iteration of $\varphi$  by (\ref{enu: main - finiteness}) of \autoref{main}.  
\end{proof}

     Given a growth type $(d,\lambda)<(d_M,\lambda_M)$, one may consider the subgroups carried by vertices $w$, with parent $v$, such that $(d_w,\lambda_w)\le(d,\lambda)<(d_v,\lambda_v)$. This yields:

\begin{coro}\label{raff}
Given $(d,\lambda)\ne (0,1)$, there exists a finite malnormal family $K_1,\dots,K_p$ of quasiconvex subgroups such that a conjugacy class $\norm g$ grows at most like $n^d \lambda^n$ under $\varphi$ if and only if $g$ has a conjugate belonging to some $K_i$. This  is also true for $(d,\lambda)=(0,1)$ if no conjugacy class grows linearly.\qed
\end{coro}

\begin{rema}
Quasiconvexity and malnormality imply that $G$ is hyperbolic relative to the family $K_1,\dots,K_p$ \cite {Bowditch:2012ga, Osin:2006gm}. We do not know whether the mapping torus $G\rtimes_\varphi \Z$ is hyperbolic relative to the groups $K_i\rtimes_\varphi \Z$ when $\lambda>1$ (this is proved in  \cite {Dahmani:2023re} when $K_1,\dots,K_p$ is the family of polynomial subgroups, i.e.\ for $\lambda =1$ and $d$ large enough). Note that the family $K_i\rtimes_\varphi \Z$ is malnormal because there are no twin subgroups (see the proof of \autoref{theo: prelim filtration} and  Lemma 2.12 of \cite{Dowdall:2025la})
\end{rema}

 We now prove an analogue of \autoref{theo:filtration} for elements, getting a chain of subgroups rather than a tree.
 
 \begin{prop}\label{slqc}
There exists a sequence of quasiconvex $\varphi$-invariant subgroups $L_q\inc L_{q-1}\inc\dots \inc L_0=G$ such that all elements in $L_i\setminus L_{i+1}$ have the same growth type, and elements in $L_{i+1}$ have smaller growth. The subgroups $L_i$ are malnormal, except possibly   $L_q$ (the periodic subgroup).
\end{prop}

\begin{proof}
 The existence is clear: we consider the growth types occuring for elements, and   the $\varphi$-invariant subgroups consisting of elements whose growth type is not bigger than a given type. We now deduce the  quasiconvexity and malnormality of $L_i$  from \autoref {theo: prelim filtration} applied in $G*\Z$.
 
 It suffices to consider $L_{1}$, which is equal to the ``slow'' subgroup $G_{sl}$ introduced at the end of \autoref{words}. 
 We may assume that it is not trivial.
 
We extend $\varphi$ to $G*\grp{t}$ by sending $t$ to itself, and we consider  the \Rt{} $T$ constructed in the proof of \autoref {theo: prelim filtration}. The elements of $G_{sl}*\grp{t}$ grow slower than the conjugacy class of $tg$ for $g\in G\setminus G_{sl}$, so 
are contained in a conjugate of one of the subgroups $H_i$ provided by the theorem. 

These subgroups are maximal elliptic subgroups, and by Serre's lemma \cite{Serre:1977wy} $G_{sl}*\grp{t}$ itself must be contained in a single conjugate, which we may assume to be $H_1$ (we do not need to know that $G_{sl}$ is finitely generated, as any non-cyclic finitely generated subgroup of $G_{sl}*\grp{t}$ fixes a single point in $T$). 

Clearly $G_{sl} \inc H_1\cap G$. Conversely, if $g\in H_1\cap G$, then the conjugacy class of $tg$ does  not have maximal growth (as a conjugacy class), so $g$ does not have maximal growth (as an element), hence belongs to $G_{sl}$. Thus  $G_{sl}= H_1\cap G$, and we deduce the 
quasiconvexity of  $G_{sl}$ (and its malnormality if the growth in $G$ is supralinear) from the corresponding properties of 
  $H_1$ stated in \autoref {theo: prelim filtration}.
\end{proof}

We may go further if we assume that some element of $G$ grows faster than $(d_M,\lambda_M)$ (the maximal growth type of conjugacy classes). 
Recall from \autoref{res: word growth fixed conj class} that this phenomenon may   occur only when $\lambda_M = 1$.

If this happens, then $G$ itself is elliptic in $T$. It follows that $G_{sl}$ is cyclic, as otherwise it would fix a unique point in $T$, so $G *\grp{t}$ would be elliptic and $T$ would be trivial. If $G_{sl}$ is not trivial, $\varphi^2$ has a non-trivial fixed subgroup, hence has finite order in $\Out(G)$ by \autoref{pareil}. We have proved:

\begin{coro}\label{bg}
If $\varphi$ has infinite order in $\Out(G)$, and some element of $G$ grows faster than $(d_M,\lambda_M)$ under $\varphi$, then all non-trivial elements of $G$ have the same growth type. \qed
\end{coro}

In \autoref{bgex}, all non-trivial elements grow quadratically.

\bibliographystyle{abbrv}
\bibliography{bibliography}

\vspace{0.5cm}

\normalsize

\noindent{\textsc{Université Bourgogne Europe, CNRS, IMB UMR 5584, 21000 Dijon, France}} \\
\noindent{\textit{E-mail address:} \texttt{remi.coulon@cnrs.fr}} \\

\noindent{\textsc{Institut Mathématique de Toulouse; UMR 5219; Université de Toulouse;
CNRS; UPS, F-31062 Toulouse Cedex 9, France}} \\
\noindent{\textit{E-mail address:} \texttt{arnaud.hilion@math.univ-toulouse.fr}} \\

\noindent{\textsc{Universit\'e Paris-Saclay, CNRS,  Laboratoire de math\'ematiques d'Orsay, 91405, Orsay, France}} \\
\noindent{\textit{E-mail address:} \texttt{camille.horbez@universite-paris-saclay.fr}\\

\noindent{\textsc{Laboratoire de Math\'ematiques Nicolas Oresme (LMNO)\\
Universit\'e de Caen et CNRS (UMR 6139)\\
(Pour Shanghai : Normandie Univ, UNICAEN, CNRS, LMNO, 14000 Caen, France)}} \\
\noindent{\textit{E-mail address:} \texttt{levitt@unicaen.fr}\\

\end{document}